\documentclass[12pt]{article}

\usepackage[colorlinks, linkcolor=blue,citecolor=blue]{hyperref}
\usepackage{color}

\usepackage{graphicx,subfigure,amsmath,amssymb,amsfonts,bm,epsfig,epsf,url,dsfont}
\usepackage{amsthm,mathrsfs}
\usepackage{tikz}
\usepackage{multirow}
\usepackage{bbm}
\usepackage{booktabs}
\usepackage{cases}
\usepackage{fullpage,enumitem}
\usepackage[small,bf]{caption}
\usepackage[top=1in,bottom=1in,left=1in,right=1in]{geometry}
\usepackage{fancybox}
\usepackage{hyperref}
\usepackage{algorithm}
\usepackage{verbatim}
\usepackage{amsthm}
\usepackage{algpseudocode}
\usepackage{mathtools}

\usepackage{scalerel,stackengine}
\stackMath
\newcommand\reallywidehat[1]{%
\savestack{\tmpbox}{\stretchto{%
0 \scaleto{%
    \scalerel*[\widthof{\ensuremath{#1}}]{\kern-.6pt\bigwedge\kern-.6pt}%
    {\rule[-\textheight/2]{1ex}{\textheight}}%WIDTH-LIMITED BIG WEDGE
  }{\textheight}%
}{0.5ex}}%
\stackon[1pt]{#1}{\tmpbox}%
}

\newcommand{\ceil}[1]{\left \lceil #1 \right \rceil}
\newcommand{\bE}{\mathbb{E}}

\newtheorem{definition}{Definition}

\newtheorem{theorem}{Theorem}
\newtheorem{corollary}{Corollary}

\newtheorem{lemma}{Lemma}
\newtheorem{prop}{Proposition}
\newtheorem{problem}{Problem}

\newenvironment{fminipage}%
  {\begin{Sbox}\begin{minipage}}%
  {\end{minipage}\end{Sbox}\fbox{\TheSbox}}

\makeatletter
\newcommand*{\rom}[1]{\expandafter\@slowromancap\romannumeral #1@}
\makeatother
\newtheorem{remark}{Remark}

\newcommand{\Ind}{\mathbbm{1}}

\newcommand{\sd}{\mu}

\newcommand{\abs}[1]{\left|#1\right|}

\newcommand{\N}{\mathbb{N}}

\newcommand{\E}{\mathbb{E}}

\def\P{{\pr}}

\newcommand{\indep}{\perp \!\!\! \perp}

\newcommand {\pr} {\mathbb{P}}

\newcommand{\calA}{{\cal A}}
\newcommand{\calB}{{\cal B}}
\newcommand{\calC}{{\cal C}}

\newcommand{\calE}{{\cal E}}

\newcommand{\calG}{{\cal G}}
\newcommand{\calH}{{\cal H}}

\newcommand{\calK}{{\cal K}}
\newcommand{\calL}{{\cal L}}

\newcommand{\calO}{{\cal O}}

\newcommand{\calR}{{\cal R}}
\newcommand{\calS}{{\cal S}}

\DeclarePairedDelimiterX{\set}[1]{\{}{\}}{\setargs{#1}}
\DeclarePairedDelimiterX{\cond}[1]{[}{]}{\setargs{#1}}
\NewDocumentCommand{\setargs}{>{\SplitArgument{1}{;}}m}
{\setargsaux#1}
\NewDocumentCommand{\setargsaux}{mm}
{\IfNoValueTF{#2}{#1} {#1\,\delimsize|\,\mathopen{}#2}}%{#1\:;\:#2}

\newcommand{\be}{\begin{equation}}
\newcommand{\ee}{\end{equation}}
\newcommand{\beqna}{\begin{eqnarray}}
\newcommand{\eeqna}{\end{eqnarray}}

\newcommand{\p}[1]{\left(#1\right)}
\newcommand{\pp}[1]{\left[#1\right]}
\newcommand{\ppp}[1]{\left\{#1\right\}}

\usepackage{xspace}

\newcommand{\s}[1]{\mathsf{#1}}

\makeatletter
\def\thanks#1{\protected@xdef\@thanks{\@thanks
        \protect\footnotetext{#1}}}
\makeatother

\makeatletter
\renewcommand{\paragraph}{%
  \@startsection{paragraph}{4}%
  {\z@}{1.25ex \@plus 1ex \@minus.2ex}{-1em}%
  {\normalfont\normalsize\bfseries}%
}
\makeatother

\usepackage{varwidth}
\usepackage[most]{tcolorbox}

\usetikzlibrary{arrows.meta,positioning,calc}

\usepackage[most]{tcolorbox}

\newtcolorbox{algobox}{
  colback=white,
  colframe=black,
  boxrule=0.4pt,
  arc=1pt,
  left=6pt,
  right=6pt,
  top=4pt,
  bottom=4pt
}

\allowdisplaybreaks

\begin{document}
%\title{Detecting Planted Subgraphs under Query Constraints}
%\title{Subgraph Detection Query Complexity}
\title{Fundamental Limits of Query-Based Subgraph Detection}
%\title{Query Complexity of Detecting Planted Subgraphs}
\author{Wasim Huleihel\thanks{W. Huleihel is with the School of Electrical Engineering and Computer Engineering, at Tel Aviv University, {T}el {A}viv 6997801, Israel (e-mail: \texttt{wasimh@tauex.tau.ac.il}).}}

\maketitle

\begin{abstract}
The planted subgraph detection problem asks whether a random graph contains a hidden structured subgraph. In the classical formulation, the entire adjacency matrix is observed and one distinguishes between an Erd\H{o}s--R\'enyi random graph and one obtained by planting a copy of a prescribed graph inside an Erd\H{o}s--R\'enyi random graph. The statistical and computational limits of this problem under full observation are now well understood, even for arbitrary planted subgraphs.

In this paper, we investigate an information-limited version of the problem in which the planted structure is an arbitrary sequence of graphs $\Gamma=(\Gamma_n)_{n\geq1}$, where $\Gamma_n$ is embedded in an ambient graph on $n$ vertices, but the observer does not have access to the full adjacency matrix. Instead, information is acquired through a limited number of non-adaptive edge queries. We study the minimum query complexity required for reliable detection.

We derive general information-theoretic lower bounds and algorithmic upper bounds on the query complexity as functions of the query budget and structural properties of the planted graph. The proposed algorithms exploit three distinct structural mechanisms: dense local motifs, high-degree vertices, and global edge density. We establish matching bounds, up to polylogarithmic factors, for several broad families of planted graphs, including clique-like, bounded-cover, and hub-dominated graph classes, as well as complete bipartite graphs of arbitrary aspect ratio. Our framework substantially generalizes existing query-complexity results for planted clique and planted dense subgraph models, providing a unified treatment of arbitrary planted subgraphs under restricted graph access.
\end{abstract}

%\newpage
%\tableofcontents

\section{Introduction}\label{sec:intro}

Identifying hidden structure in random graphs is a central problem in modern statistics, theoretical computer science, and network science. A standard model assumes that one observes a graph on $n$ vertices drawn either from an Erd\H{o}s--R\'enyi random graph or from a modified distribution in which a structured subgraph has been planted. The inferential goal may be either to decide whether such a planted structure is present, or to localize the planted vertices. Over the past three decades, this general paradigm has led to a rich body of work spanning hypothesis testing, estimation, average-case complexity, and algorithm design.

Much of the classical literature has focused on specific planted structures, most notably planted clique, planted dense subgraph, biclique, and related community-detection models. For these problems, the statistical and computational thresholds are by now fairly well understood in many parameter regimes; see, for example, \cite{arias2014community,butucea2013detection,verzelen2015community,chen2016statistical,montanari2015finding,candogan2018finding,hajek2016achieving,ma2015computational,hajek2015computational,brennan2018reducibility}. A recurring theme in this line of work is the appearance of statistical--computational gaps: detection may be information-theoretically possible well before any known polynomial-time algorithm succeeds.

More recently, there has been substantial progress toward a unified theory of planted subgraph inference beyond these ad hoc examples. On the \emph{detection} side, a sequence of works has studied arbitrary planted subgraphs and identified general graph-theoretic quantities that govern the testing problem in several regimes \cite{addario2010combinatorial,Huleihel2022,pmlr-v247-yu24a,elimelech2025detecting,ElimelechHuleihel2025SemiRandom}. In particular, the recent work \cite{elimelech2025detecting} develops a general full-observation theory for detecting arbitrary planted subgraphs in Erd\H{o}s--R\'enyi graphs, establishing sharp statistical and computational thresholds in the dense regime and broad results in sparse and critical regimes as well. These developments suggest that, at least when the full graph is observed, the detection problem is close to being understood at a rather high level of generality.

%The situation is less complete for the \emph{recovery} problem. Recent advances in \cite{pmlr-v195-mossel23a,LeePerniceRajaramanZadik2025} provide a general statistical characterization of weak recovery for arbitrary planted subgraphs through variational formulas for the limiting MMSE under suitable structural assumptions. These results are powerful, but they focus on approximate recovery in a weak sense. Exact recovery, partial recovery under alternative loss functions, and algorithmic aspects of recovery for general planted structures remain much less understood. In particular, unlike the detection problem, there is still no general picture describing when statistically optimal recovery is computationally feasible and when genuine computational barriers arise.

The landscape for the \emph{recovery} problem has also seen significant recent progress. Early works such as \cite{pmlr-v195-mossel23a,LeePerniceRajaramanZadik2025} provided a general statistical characterization of weak recovery for arbitrary planted subgraphs, via variational formulas describing the limiting MMSE under suitable structural assumptions. These results established a unified framework for approximate recovery. More recently, \cite{Huleihel26} develops a general theory for \emph{exact recovery} of arbitrary planted subgraphs in dense Erd\H{o}s-R\'{e}nyi graphs. It identifies sharp necessary and sufficient conditions for exact recovery, characterized by a novel graph-theoretic quantity termed the \emph{minimal maximum subgraph density}. In addition, it provides efficient recovery algorithms and establishes computational lower bounds using the low-degree framework, thereby initiating a systematic study of statistical--computational tradeoffs for recovery in this general setting.

In many applications, however, one does not have access to the full graph. Large networks may be too expensive to inspect exhaustively, privacy or storage constraints may restrict access to edges, and in interactive or crowdsourced settings one may only be allowed to probe selected pairs of items. This naturally leads to a query-limited version of planted subgraph inference: rather than observing the entire adjacency matrix, the algorithm chooses a set of pairs and only learns whether those queried pairs are connected. Understanding how many such queries are needed in order to detect hidden structure is a fundamental question, both for theory and for applications involving graph sampling, active data acquisition, and similarity-based exploration.

This query perspective has already attracted attention in several related settings. In random graphs, adaptive edge-query models were studied for finding large cliques and other target structures; see, e.g., \cite{Feige20,Alweiss2020,Ferber15,Ferber17,Conlon18}. Query models also arise naturally in clustering and community inference, where pairwise comparison or similarity queries serve as a primitive for learning latent labels or partitions \cite{Mazumdar2017ClusteringWN,Mazumdar2017QueryCO,Vinayak2016CrowdsourcedCQ,Hartmann2016ClusteringEN,Anagnostopoulos2016CommunityDO}. Most closely related to our work are the studies of planted clique under query access in \cite{racz2020finding,mardia2020finding,mardia2021space,mardia2024lowdegree,HuleihelMazumdarPal2024QueryPDS}, which characterize the statistical and computational barriers for detecting a planted clique (or, slightly more generally, a planted dense subgraph). Those works already indicate that restricted graph access creates a new layer of difficulty and can magnify statistical--computational tradeoffs.

The goal of the present paper is to develop a general theory of \emph{query complexity for detecting arbitrary planted subgraphs}. We study a model in which the planted structure is an arbitrary graph sequence $\Gamma_n$, embedded uniformly at random into an ambient Erd\H{o}s-R\'{e}nyi graph, and the observer is allowed only a prescribed number of \emph{non-adaptive edge queries}. Thus, the paper may be viewed as bringing together two previously separate directions: the recent full-observation theory for arbitrary planted subgraph detection, and the query-complexity theory developed so far mainly for planted clique and planted dense subgraph models.

Our main objective is to understand how the structural properties of $\Gamma_n$ determine the number of non-adaptive queries required for reliable detection. We derive several general lower bounds, each capturing a different obstruction to detection. The first is a direct ``edge-hit'' lower bound, showing that unless the query set intersects the planted structure often enough, the transcript under the alternative is statistically indistinguishable from that under the null. This bound is especially effective for clique-like or witness-rich graphs, where detection can in principle be triggered by observing even a small dense planted pattern. Our second, more structural lower bound is based on a vertex-cover decomposition of $\Gamma_n$. It separates a cover-induced core from a residual part and shows that if the core is itself undetectable and the residual attachment to the cover is sufficiently diffuse relative to the query profile, then detection remains impossible. This theorem is flexible and leads to sharp consequences for broad sparse families, including bounded-cover graphs, star forests, pendant-leaf constructions, bounded-degree trees, paths, and several bipartite regimes.

On the algorithmic side, we propose three complementary non-adaptive algorithms. The first is a \emph{witness scan test}, which samples a random vertex subset, queries all edges inside it, and searches for a suitable witness subgraph $H\subseteq \Gamma_n$. This test is effective when $\Gamma_n$ contains many sufficiently dense local patterns, and it recovers the optimal scale, up to polylogarithmic factors, for clique-like and biclique-like structures. The second is a \emph{biclique-scan test}, a bipartite variant of the witness scan: it queries all pairs between two small random vertex sets and looks for a complete bipartite pattern of logarithmic size. Whenever $\Gamma_n$ contains a complete bipartite subgraph $K_{a_n,b_n}$ with $a_n,b_n\gtrsim\log n$, it achieves strong detection with $O(n^2\log^2 n/(a_nb_n))$ queries, which matches the edge-hit lower bound up to a $\log^2 n$ factor; this covers cliques and complete bipartite graphs of arbitrary aspect ratio. The third is a polynomial-time \emph{degree-on-a-cut} test, which samples a bipartite cut and looks for unusually large query degree on one side. Its performance is governed by the high-degree mass
\[
\kappa(\Gamma_n)\triangleq \max_{d\ge 1} m_d(\Gamma_n)d^2,
\qquad
m_d(\Gamma_n)\triangleq \big|\{u\in v(\Gamma_n):\deg_{\Gamma_n}(u)\ge d\}\big|,
\]
and it is near-optimal for hub-dominated families such as stars, star forests, and more generally bounded-cover graphs with a non-negligible high-degree layer.

Taken together, our results show that query complexity for arbitrary planted subgraph detection is governed by a small collection of structural principles. Dense local witnesses favor scan-based procedures and are well captured by the edge-hit lower bound. Hub-dominated geometries are governed instead by degree exposure across suitably chosen cuts, and are captured by the cover-based lower bounds. This leads to a unified picture in which different algorithmic mechanisms are optimal for different structural classes of planted graphs.

At a broader level, our results reveal that the query model has its own statistical and computational landscape, distinct from the one arising under full observation. In some families, the information-theoretic threshold under query access remains close to the full-observation one, while in others the cost of restricted access is substantial. Likewise, there are regimes in which quasi-optimal non-adaptive detection is possible but the best known polynomial-time algorithms require many more queries. This provides further evidence that active access constraints can generate new statistical--computational tradeoffs even for inference tasks that are already well understood in the classical full-observation setting.

\paragraph{Notation.} We collect here the main notational conventions used throughout the paper. For any finite set $\calS$, we write $|\calS|$ for its cardinality. For $n\in\mathbb{N}$, we denote $[n]\triangleq\{1,\ldots,n\}$, and for any set $\calS$, we let $\binom{\calS}{k}\triangleq\{\calA\subseteq\calS:|\calA|=k\}$. For a subset $\calS\subseteq\mathbb{R}$, we denote by $\Ind\pp{\calS}$ its indicator function.

We denote by $\s{Bern}(p)$ and $\s{Binomial}(n,p)$ the Bernoulli and binomial distributions with success probability $p$, and by $\s{Hypergeometric}(N,K,n)$ the hypergeometric distribution with the usual parameters. For a finite or measurable set $\mathcal{X}$, $\s{Unif}[\mathcal{X}]$ denotes the uniform distribution over $\mathcal{X}$. For random variables $\s{X}$ and $\s{Y}$, we write $\s{X}\indep\s{Y}$ to indicate independence. For probability measures $\mathbb{P}$ and $\mathbb{Q}$, we write $d_{\s{TV}}(\mathbb{P},\mathbb{Q})$ and $\chi^2(\mathbb{P}||\mathbb{Q}),$ for the total variation distance and $\chi^2$-divergence, respectively. In particular, for Bernoulli distributions we use $\chi^2(p||q)=\frac{(p-q)^2}{q(1-q)}$. We will frequently use standard asymptotic notation $O(\cdot),o(\cdot),\Omega(\cdot),\omega(\cdot),\Theta(\cdot)$. We write $a_n\ll b_n$ (equivalently, $b_n\gg a_n$) if $a_n=o(b_n)$. We write $a_n\lesssim b_n$ if there exists a constant $\s{C}>0$, independent of $n$ (but possibly depending on $q$ and on other fixed parameters, such as $D$, $\tau_0$, $\rho$ below), such that $a_n\le \s{C}b_n$ for all sufficiently large $n$. The notation $a_n\gtrsim b_n$ is defined analogously. Throughout, $\s{C}$ denotes a generic constant that may change from line to line.

A graph $\s{G}$ is a pair $(v(\s{G}),e(\s{G}))$, where $v(\s{G})$ is the vertex set and $e(\s{G})\subseteq \binom{v(\s{G})}{2}$ is the edge set. We write $|v(\s{G})|$ for the number of vertices and $|e(\s{G})|$ (or $|\s{G}|$) for the number of edges. We say that $\s{H}\subseteq \s{G}$ is a subgraph if $v(\s{H})\subseteq v(\s{G})$ and $e(\s{H})\subseteq e(\s{G})$. When $\s{H}$ has no isolated vertices, we often identify it with its edge set and write $|\s{H}|$ for $|e(\s{H})|$. For a vertex subset $U\subseteq v(\s{G})$, we denote by $\s{G}_U$ the induced subgraph. For a vertex $v\in v(\s{G})$, we denote its neighborhood by $N_{\s{G}}(v)\triangleq \ppp{u\in v(\s{G}): \{u,v\}\in e(\s{G})}$, and its degree by $\deg_{\s{G}}(v)=|N_{\s{G}}(v)|$. The complete graph on $n$ vertices is denoted by $\calK_n$. Two graphs $\s{G}_1$ and $\s{G}_2$ are said to be isomorphic, denoted $\s{G}_1\cong \s{G}_2$, if there exists a bijection between their vertex sets preserving adjacency. The automorphism group of a graph $\s{G}$ is denoted by $\s{Aut}(\s{G})$. For a graph $\s{G}$ with $|v(\s{G})|\le n$, we denote by $\calS_{\s{G}}$ the collection of all labeled copies of $\s{G}$ in $\calK_n$. It is well known that $|\calS_{\s{G}}| = \binom{n}{|v(\s{G})|}\frac{|v(\s{G})|!}{|\s{Aut}(\s{G})|}$ (see, e.g., \cite[Lemma 5.1]{FriezeKaronski2015}).

Finally, for two graphs $\s{H}$ and $\Gamma$ with $v\triangleq|v(\s{H})|$, we let
\begin{align}
\calC(\s{H},\Gamma)\triangleq\ppp{W\subseteq v(\Gamma):\ |W|=v\ \text{and}\ \Gamma[W]\ \text{contains a subgraph isomorphic to}\ \s{H}}
\end{align}
be the collection of \emph{witness vertex sets} of $\s{H}$ in $\Gamma$, and we write $N(\s{H},\Gamma)\triangleq|\calC(\s{H},\Gamma)|$. For $t\in\{1,\ldots,v-1\}$, we let $N_t(\s{H},\Gamma)$ denote the number of ordered pairs $(W,W')\in\calC(\s{H},\Gamma)^2$ with $W\neq W'$ and $|W\cap W'|=t$. (For $\s{H}=\calK_h$ and $\Gamma=\calK_k$ one has $N(\s{H},\Gamma)=\binom{k}{h}$, and for $\s{H}=K_{s,s}$ and $\Gamma=K_{a,b}$ one has $N(\s{H},\Gamma)=\binom{a}{s}\binom{b}{s}$.)

\paragraph{Organization.}
The remainder of the paper is organized as follows. In Section~\ref{sec:problem} we formalize the model and the query-based detection problem. Section~\ref{sec:main_results} states our main lower and upper bounds and compares them across different graph families. The general lower bounds are proved in Section~\ref{sec:lower_bounds}, while the algorithmic upper bounds are established in Section~\ref{sec:upper_bounds}. We then specialize the general theory to several representative examples---including cliques, stars, trees, triangle unions, bounded-cover graphs, and complete bipartite graphs---in order to identify regimes in which the lower and upper bounds coincide up to constant or polylogarithmic factors.

\section{Setup}\label{sec:problem}

In this section we describe the setting we plan to study alongside several important preliminaries. Let $\Gamma=(\Gamma_n)_{n\in\N}$ be a sequence of graphs such that for all $ n\in\mathbb{N}$, $\Gamma_n=(v(\Gamma_n),e(\Gamma_n))$ is an \emph{arbitrary} undirected graph without isolated vertices, and $|v(\Gamma_n)|\leq n$. %Throughout the paper $v(\Gamma_n)$ and $e(\Gamma_n)$ denote the set of vertices and edges of $\Gamma_n$, respectively; $|v(\Gamma_n)|\leq n$ and $|e(\Gamma_n)|=|\Gamma_n|$ denote the sizes of thereof.
We let $\calS_{\Gamma_n}$ be the set of isomorphic copies of $\Gamma_n$ in $\calK_n$.
% a simple combinatorial counting argument reveals that $|\calS_{\Gamma_n}| = \binom{n}{|v(\Gamma_n)|} \frac{|v(\Gamma_n)|!}{\vert\mathsf{Aut}{\p{\Gamma_n}}\vert}$ (see, e.g., \cite[Lemma 5.1]{FriezeKaronski2015})
We shall refer to $\Gamma_n$ as the \emph{planted/hidden} structure.

We start with the formulation of the vanilla planted subgraph hypothesis testing problem. Here, under the null hypothesis $\calH_0$, the observed graph $\s{G}$ is an instance from the Erd\H{o}s--R\'enyi random distribution $\calG(n,q)$, with edge density $0<q<1$. Under the alternative hypothesis $\calH_1$, the observed graph $\s{G}$ is the union of an Erd\H{o}s--R\'enyi random graph with a uniform random copy of $\Gamma_n$. To wit, under $\calH_1$ the graph $\s{G}$ on $n$ vertices is constructed as follows: we first draw $\Gamma_n\sim\s{Unif}(\calS_{\Gamma_n})$. Then, each edge of $\Gamma_n$ is included in $\s{G}$ with probability one, while the leftover edges in $\s{G}$ are included with probability $q$. In short, we have the following hypothesis testing problem:
\begin{align}
\calH_0: \s{G} \sim \calG(n,q) \quad \s{vs.} \quad \calH_1: \s{G} \sim \calG_{\Gamma_n}(n,q),\label{eqn:super_hypo}
\end{align}
where $\calG_{\Gamma_n}(n,q)$ denotes the ensemble of planted graphs, as defined above. We study the above framework in the asymptotic regime where $n\to \infty$. Throughout, we focus on the dense regime where $q\in(0,1)$ is fixed.

In this work, we consider a variant of the detection problem above, where one can only inspect a small part of the graph by \emph{non-adaptive edge queries}, defined as follows.
\begin{definition}[Oracle/Edge queries]\label{def:oracle}
Consider a graph $\s{G}_n=([n],\calE)$ with $n$ vertices, where $\calE$ denotes the set of edges. An oracle $\calO:[n]\times[n]\to\{0,1\}$, takes as input a pair of vertices $i,j\in[n]\times[n]$, and if $(i,j)\in\calE$, namely, there exists an edges between the chosen vertices, then $\calO(i,j)=1$, otherwise, $\calO(i,j)=0$.
\end{definition}

We consider query mechanisms that evolve dynamically over $\s{Q}$ steps/queries in the following form: in step number $\ell\in[\s{Q}]$, the mechanism chooses a pair of vertices $\s{e}_\ell\triangleq (i_\ell,j_\ell)$ and asks the oracle whether these vertices are connected by an edge or not. Generally speaking, either \emph{adaptive} or \emph{non-adaptive} query mechanisms can be considered. In the former, the chosen $\ell$th pair may depend on the previously chosen pairs $\{\s{e}_i\}_{i<\ell}$, as well as on past responses $\{\calO(\s{e}_i)\}_{i<\ell}$. In non-adaptive mechanisms, on the other hand, all queries must be made upfront. In this paper we focus on non-adaptive mechanisms. The query detection problem is defined as follows.
\begin{problem}[Query-$\s{PDS}$ detection problem]\label{def:PDSquery}
Consider the detection problem in \eqref{eqn:super_hypo}. There is an oracle $\calO$ as defined in Definition~\ref{def:oracle}. Find a set of queries $\mathbb{Q}\subseteq\binom{[n]}{2}$ such that $|\mathbb{Q}|=\s{Q}$, and from the oracle responses it is possible to solve the detection problem in \eqref{eqn:super_hypo}.
\end{problem}

For a given (possibly random) query set $\mathbb{Q}_n\subseteq \binom{[n]}{2}$, we denote by
\begin{align}
\s{G}_{\mathbb{Q}_n} \triangleq (\s{G}_e)_{e\in \mathbb{Q}_n}\in\{0,1\}^{|\mathbb{Q}_n|}
\end{align}
the observed transcript of oracle responses. We let $\pr_{\calH_0,\mathbb{Q}_n}$ and $\pr_{\calH_1,\mathbb{Q}_n}$ denote the distributions of $\s{G}_{\mathbb{Q}_n}$ under the null and alternative hypotheses, respectively. A (possibly randomized) detection algorithm $\calA_n$ is a measurable function
\begin{align}
\calA_n:\{0,1\}^{|\mathbb{Q}_n|}\to\{0,1\},
\end{align}
which, based on the observed transcript $\s{G}_{\mathbb{Q}_n}$, outputs a decision in $\{0,1\}$. The \emph{risk} of $\calA_n$ (with respect to the query set $\mathbb{Q}_n$) is defined as
\begin{align}
\s{R}_n(\calA_n;\mathbb{Q}_n)
\triangleq
\pr_{\calH_0,\mathbb{Q}_n}\p{\calA_n(\s{G}_{\mathbb{Q}_n})=1}
+
\pr_{\calH_1,\mathbb{Q}_n}\p{\calA_n(\s{G}_{\mathbb{Q}_n})=0},
\end{align}
that is, the sum of the Type-I and Type-II error probabilities. The \emph{optimal risk} associated with the query set $\mathbb{Q}_n$ is
\begin{align}
\s{R}_n^\star(\mathbb{Q}_n)
\triangleq
\inf_{\calA_n}\s{R}_n(\calA_n;\mathbb{Q}_n).
\end{align}
It is well known (see, e.g., Le Cam's lemma \cite[Theorem 2.2]{tsybakov2004introduction}) that
\begin{align}
\s{R}_n^\star(\mathbb{Q}_n)
=
1 - d_{\s{TV}}\p{\pr_{\calH_1,\mathbb{Q}_n},\pr_{\calH_0,\mathbb{Q}_n}},
\label{eq:opt_risk_tv}
\end{align}
so that the statistical difficulty of the problem is fully characterized by the total variation distance between the two induced transcript distributions. Finally, the \emph{optimal risk under query budget $\s{Q}_n$} is defined by
\begin{align}
\s{R}_{n,\s{Q}_n}^\star
\triangleq
\inf_{\substack{\mathbb{Q}_n\subseteq \binom{[n]}{2}\\ |\mathbb{Q}_n|\le \s{Q}_n}}
\s{R}_n^\star(\mathbb{Q}_n)
=
1-
\sup_{\substack{\mathbb{Q}_n\subseteq \binom{[n]}{2}\\ |\mathbb{Q}_n|\le \s{Q}_n}}
d_{\s{TV}}\p{
\pr_{\calH_1,\mathbb{Q}_n},
\pr_{\calH_0,\mathbb{Q}_n}
}.
\label{eq:budget_opt_risk_tv}
\end{align}
We consider the following types of detection guarantees.
\begin{definition}[Strong and weak detection]\label{def:detection}
Let $\s{Q}_n$ be a sequence of query budgets. We say that a sequence of (possibly randomized) tests $\calA_n$, using query sets $\mathbb{Q}_n$ with $|\mathbb{Q}_n|\le \s{Q}_n$, achieves \emph{strong detection} if $\limsup_{n\to\infty}\s{R}_n(\calA_n;\mathbb{Q}_n)=0$, and achieves \emph{weak detection} if $\limsup_{n\to\infty}\s{R}_n(\calA_n;\mathbb{Q}_n)<1$. Conversely, we say that strong detection is \emph{impossible} under query budget $\s{Q}_n$ if $\liminf_{n\to\infty}\s{R}_{n,\s{Q}_n}^\star>0,$ and weak detection is \emph{impossible} if $\lim_{n\to\infty}\s{R}_{n,\s{Q}_n}^\star=1$.
\end{definition}

Our results will be expressed in terms of the following graph theoretic measures. We let $d_{\s{max}}\p{\Gamma_n}$ denote the maximum degree in $\Gamma_n$, and we define the density of $\Gamma_n$ as $\eta(\Gamma_n)\triangleq |e(\Gamma_n)|/|v(\Gamma_n)|$. Finally, we recall the following definitions of the \emph{maximum subgraph density} and \emph{vertex cover}.
\begin{definition}[Maximum subgraph density \cite{bollobas_2001}]\label{def:maxsubden}
Let $\s{G}$ be an undirected graph. The maximum subgraph density of $\s{G}$ is
\begin{align}
\sd(\s{G})\triangleq \max\ppp{\eta(\s{H}):\s{H}\subseteq\s{G},\s{H}\neq\emptyset}.\label{eqn:maxDensity}
\end{align}
\end{definition}
\begin{definition}[Vertex cover \cite{elimelech2025detecting}]\label{def:vertex_cover}
Let $\s{G}=(v(\s{G}),e(\s{G}))$ be an undirected graph. A subset of vertices $U \subseteq v(\s{G})$ is called a \emph{vertex cover} of $\s{G}$ if every edge of $\s{G}$ is incident to at least one vertex in $U$, namely, $\forall\{u,v\}\in e(\s{G})$ it holds $\{u,v\}\cap U \neq \emptyset$. The \emph{vertex cover number} of $\s{G}$, denoted by $\tau(\s{G})$, is the minimum size of a vertex cover:
\begin{align}
\tau(\s{G}) \triangleq \min\{|U|: U \subseteq v(\s{G})\; \s{is}\; \s{a}\; \s{vertex}\; \s{cover}\; \s{of}\; \s{G} \}.
\end{align}
\end{definition}

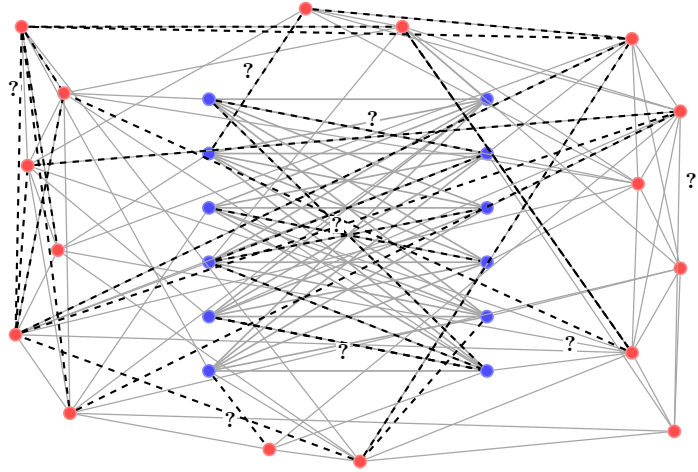
\begin{figure}[t]
\centering

\scalebox{0.8}{\begin{tikzpicture}[
    >=Latex,
    blue node/.style={
        circle,
        draw=blue!50,
        fill=blue!70,
        thick,
        minimum size=2mm,
        inner sep=0pt},
    red node/.style={
        circle,
        draw=red!50,
        fill=red!70,
        thick,
        minimum size=2mm,
        inner sep=0pt},
    blue edge/.style={
        draw=gray!70,
        line width=0.8pt},
    red edge/.style={
        draw=gray!70,
        line width=0.6pt},
    query edge/.style={
        draw=black,dashed,
        line width=1pt},
    qmark/.style={
        fill=white,
        inner sep=0.5pt,
        font=\bfseries\footnotesize},
    box/.style={
        draw,
        rounded corners,
        thick,
        minimum width=2.5cm,
        minimum height=8mm,
        align=center,
        fill=gray!8}
]

%%%%%%%%%%%%%%%%%%%%%%%%%%%%%%%%%%%%%%%%%%%%%%%%%%%%%%%%%%%%
%% LEFT SIDE: GRAPH
%%%%%%%%%%%%%%%%%%%%%%%%%%%%%%%%%%%%%%%%%%%%%%%%%%%%%%%%%%%%

%----------------------------------------------------------
% Planted left partition
%----------------------------------------------------------

\node[blue node] (a1) at (-2.3, 2.2) {};
\node[blue node] (a2) at (-2.3, 1.3) {};
\node[blue node] (a3) at (-2.3, 0.4) {};
\node[blue node] (a4) at (-2.3, -0.5) {};
\node[blue node] (a5) at (-2.3, -1.4) {};
\node[blue node] (a6) at (-2.3, -2.3) {};

%----------------------------------------------------------
% Planted right partition
%----------------------------------------------------------

\node[blue node] (b1) at ( 2.3, 2.2) {};
\node[blue node] (b2) at ( 2.3, 1.3) {};
\node[blue node] (b3) at ( 2.3, 0.4) {};
\node[blue node] (b4) at ( 2.3, -0.5) {};
\node[blue node] (b5) at ( 2.3, -1.4) {};
\node[blue node] (b6) at ( 2.3, -2.3) {};

%----------------------------------------------------------
% Complete planted bipartite graph
%----------------------------------------------------------

\foreach \i in {1,...,6}
{
    \foreach \j in {1,...,6}
    {
        \draw[blue edge] (a\i)--(b\j);
    }
}

%%%%%%%%%%%%%%%%%%%%%%%%%%%%%%%%%%%%%%%%%%%%%%%%%%%%%%%%%%%%
%% Background vertices
%%%%%%%%%%%%%%%%%%%%%%%%%%%%%%%%%%%%%%%%%%%%%%%%%%%%%%%%%%%%

\node[red node] (r1) at (-5.4, 3.4) {};
\node[red node] (r2) at (-4.7, 2.3) {};
\node[red node] (r3) at (-5.3, 1.1) {};
\node[red node] (r4) at (-4.8, -0.3) {};
\node[red node] (r5) at (-5.5, -1.7) {};
\node[red node] (r6) at (-4.6, -3.0) {};

\node[red node] (r7) at (-0.7, 3.7) {};
\node[red node] (r8) at ( 0.9, 3.4) {};
\node[red node] (r9) at ( 4.7, 3.2) {};
\node[red node] (r10) at ( 5.5, 2.0) {};
\node[red node] (r11) at ( 4.8, 0.8) {};
\node[red node] (r12) at ( 5.5, -0.6) {};
\node[red node] (r13) at ( 4.7, -2.0) {};
\node[red node] (r14) at ( 5.4, -3.3) {};

\node[red node] (r15) at ( 0.2,-3.8) {};
\node[red node] (r16) at (-1.3,-3.6) {};

%%%%%%%%%%%%%%%%%%%%%%%%%%%%%%%%%%%%%%%%%%%%%%%%%%%%%%%%%%%%
%% Background graph
%%%%%%%%%%%%%%%%%%%%%%%%%%%%%%%%%%%%%%%%%%%%%%%%%%%%%%%%%%%%

%%%%%%%%%%%%%%%%%%%%%%%%%%%%%%%%%%%%%%%%%%%%%%%%%%%%%%%%%%%%
%% Dense deterministic background graph
%%%%%%%%%%%%%%%%%%%%%%%%%%%%%%%%%%%%%%%%%%%%%%%%%%%%%%%%%%%%

% Left cluster
\draw[red edge] (r1)--(r2);
\draw[red edge] (r1)--(r3);
\draw[red edge] (r1)--(r4);
\draw[red edge] (r2)--(r3);
\draw[red edge] (r2)--(r5);
\draw[red edge] (r3)--(r4);
\draw[red edge] (r3)--(r6);
\draw[red edge] (r4)--(r5);
\draw[red edge] (r5)--(r6);
\draw[red edge] (r2)--(r6);

% Top
\draw[red edge] (r7)--(r8);
\draw[red edge] (r7)--(r9);
\draw[red edge] (r8)--(r10);
\draw[red edge] (r8)--(r11);

% Right cluster
\draw[red edge] (r9)--(r10);
\draw[red edge] (r9)--(r11);
\draw[red edge] (r9)--(r12);
\draw[red edge] (r10)--(r11);
\draw[red edge] (r10)--(r13);
\draw[red edge] (r10)--(r14);
\draw[red edge] (r11)--(r12);
\draw[red edge] (r11)--(r13);
\draw[red edge] (r12)--(r13);
\draw[red edge] (r12)--(r14);
\draw[red edge] (r13)--(r14);

% Bottom
\draw[red edge] (r15)--(r16);
\draw[red edge] (r14)--(r15);
\draw[red edge] (r13)--(r15);
\draw[red edge] (r6)--(r16);

% Long-range edges
\draw[red edge] (r2)--(r8);
\draw[red edge] (r3)--(r7);
\draw[red edge] (r4)--(r15);
\draw[red edge] (r5)--(r11);
\draw[red edge] (r6)--(r12);
\draw[red edge] (r7)--(r10);
\draw[red edge] (r8)--(r13);
\draw[red edge] (r3)--(r10);
\draw[red edge] (r4)--(r9);
\draw[red edge] (r1)--(r15);
\draw[red edge] (r5)--(r13);
\draw[red edge] (r2)--(r11);
\draw[red edge] (r6)--(r14);
\draw[red edge] (r8)--(r12);
\draw[red edge] (r7)--(r11);
\draw[red edge] (r9)--(r15);
\draw[red edge] (r10)--(r16);
\draw[red edge] (r5)--(r9);
\draw[red edge] (r1)--(r8);
\draw[red edge] (r4)--(r13);

%%%%%%%%%%%%%%%%%%%%%%%%%%%%%%%%%%%%%%%%%%%%%%%%%%%%%%%%%%%%
%% Sparse connections to planted graph
%%%%%%%%%%%%%%%%%%%%%%%%%%%%%%%%%%%%%%%%%%%%%%%%%%%%%%%%%%%%

\draw[red edge] (r2)--(a1);
\draw[red edge] (r4)--(a2);
\draw[red edge] (r6)--(a3);
\draw[red edge] (r5)--(a4);
\draw[red edge] (r3)--(a5);
\draw[red edge] (r1)--(a6);

\draw[red edge] (r7)--(a2);
\draw[red edge] (r8)--(a4);

\draw[red edge] (r8)--(b1);
\draw[red edge] (r9)--(b2);
\draw[red edge] (r10)--(b3);
\draw[red edge] (r11)--(b4);
\draw[red edge] (r12)--(b5);
\draw[red edge] (r13)--(b6);

\draw[red edge] (r14)--(b2);
\draw[red edge] (r15)--(b4);
\draw[red edge] (r16)--(b6);

\draw[red edge] (r5)--(b2);
\draw[red edge] (r3)--(b5);
\draw[red edge] (r6)--(b1);
\draw[red edge] (r2)--(b6);

\draw[red edge] (r11)--(a1);
\draw[red edge] (r13)--(a3);
\draw[red edge] (r10)--(a5);
\draw[red edge] (r15)--(a6);

%%%%%%%%%%%%%%%%%%%%%%%%%%%%%%%%%%%%%%%%%%%%%%%%%%%%%%%%%%%%
%% QUERIED EDGES
%%%%%%%%%%%%%%%%%%%%%%%%%%%%%%%%%%%%%%%%%%%%%%%%%%%%%%%%%%%%

% inside planted graph
% inside planted
\draw[query edge] (a1)--(b2);
\draw[query edge] (a3)--(b4);
\draw[query edge] (a5)--(b6);
\draw[query edge] (a1)--(b6);
\draw[query edge] (a4)--(b6);
\draw[query edge] (a4)--(b2);
\draw[query edge] (a4)--(b3);
\draw[query edge] (a5)--(b6);

% inside background
\draw[query edge] (r1)--(r5);
\draw[query edge] (r3)--(r10);
\draw[query edge] (r8)--(r13);
\draw[query edge] (r1)--(r6);
\draw[query edge] (r6)--(r10);
\draw[query edge] (r5)--(r9);
\draw[query edge] (r5)--(r10);
\draw[query edge] (r5)--(r2);
\draw[query edge] (r5)--(r3);
\draw[query edge] (r9)--(r1);
\draw[query edge] (r7)--(r9);
\draw[query edge] (r1)--(r2);
\draw[query edge] (r1)--(r4);
\draw[query edge] (r1)--(r8);
\draw[query edge] (r13)--(r8);
\draw[query edge] (r13)--(r2);
\draw[query edge] (r15)--(r5);
\draw[query edge] (r15)--(r9);

% planted-background
\draw[query edge] (a2)--(r7);
\draw[query edge] (b5)--(r15);
\draw[query edge] (a6)--(r16);

%%%%%%%%%%%%%%%%%%%%%%%%%%%%%%%%%%%%%%%%%%%%%%%%%%%%%%%%%%%%
%% Question marks
%%%%%%%%%%%%%%%%%%%%%%%%%%%%%%%%%%%%%%%%%%%%%%%%%%%%%%%%%%%%

\node[qmark] at ($(a1)!0.55!(b2)+(0.18,0.18)$) {?};
\node[qmark] at ($(a3)!0.5!(b4)+(-0.18,0.18)$) {?};
\node[qmark] at ($(a5)!0.45!(b6)+(0.15,-0.18)$) {?};

\node[qmark] at ($(r1)!0.5!(r3)+(-0.18,0.12)$) {?};
\node[qmark] at ($(r10)!0.5!(r12)+(0.18,0.15)$) {?};

\node[qmark] at ($(a2)!0.5!(r7)+(-0.15,0.18)$) {?};
\node[qmark] at ($(b5)!0.5!(r13)+(0.18,-0.15)$) {?};
\node[qmark] at ($(a6)!0.5!(r16)+(-0.15,-0.15)$) {?};

%%%%%%%%%%%%%%%%%%%%%%%%%%%%%%%%%%%%%%%%%%%%%%%%%%%%%%%%%%%%
%% ORACLE DIAGRAM
%%%%%%%%%%%%%%%%%%%%%%%%%%%%%%%%%%%%%%%%%%%%%%%%%%%%%%%%%%%%

\end{tikzpicture}}

\caption{Illustration of the edge-query model. The graph contains a planted complete bipartite subgraph supported on the blue vertices, while the remaining vertices are colored red. Dashed black edges denote the queried vertex pairs, and gray edges denote unqueried edges. The algorithm selects all queried pairs in advance, before observing any responses from the edge oracle.}
\label{fig:recovery-figs}
\end{figure}

\section{Main results}\label{sec:main_results}

In this section we present our main results. We begin with general statistical lower bounds that apply to any planted graph sequence $(\Gamma_n)$, followed by general upper bounds achieved by concrete algorithms. We then summarize how these bounds compare across different structural regimes.

\subsection{Statistical lower bounds}

Our first result is a universal lower bound based only on the total number of edges.
\begin{theorem}[Edge-hit lower bound]
\label{thm:edge_hit_lower_bound}
Fix $q\in(0,1)$, and let $(\Gamma_n)_{n\ge 1}$ be any sequence of graphs with
$|v(\Gamma_n)|\le n$. If
\begin{align}
\s{Q}_n=o\p{\frac{n^2}{|e(\Gamma_n)|}},
\label{eq:edge_hit_condition}
\end{align}
then weak detection is impossible with query budget $\s{Q}_n$.
\end{theorem}

The proof of Theorem~\ref{thm:edge_hit_lower_bound} is based on a simple but fundamental observation: unless the query set $\mathbb{Q}$ intersects the planted edges, the two models are statistically indistinguishable. Indeed, under the alternative, the only difference from the null model occurs on the planted edges. If none of the queried edges belongs to the planted copy $\Gamma^\star$, then all queried edges are distributed as $\s{Bern}(q)$ exactly as under the null. Therefore, conditioned on the event that $\mathbb{Q}\cap e(\Gamma^\star)=\emptyset$, the observed transcript has the same distribution under $\calH_0$ and $\calH_1$. The probability of observing at least one planted edge is at most $\s{Q}_n |e(\Gamma_n)|/\binom{n}{2}$, since each queried edge hits the planted subgraph with probability $|e(\Gamma_n)|/\binom{n}{2}$. When $\s{Q}_n \ll n^2/|e(\Gamma_n)|$, this probability tends to zero, meaning that with high probability the algorithm never ``touches'' the planted structure at all. In this regime, no algorithm can perform better than random guessing.

\begin{remark}[Quantitative form; bounded budgets]\label{rem:edge_hit_quant}
The proof of Theorem~\ref{thm:edge_hit_lower_bound} gives the non-asymptotic bound
\begin{align}
\sup_{\mathbb{Q}_n\subseteq\binom{[n]}{2}:\ |\mathbb{Q}_n|\le \s{Q}_n}
d_{\s{TV}}\p{\pr_{\calH_1,\mathbb{Q}_n},\pr_{\calH_0,\mathbb{Q}_n}}
\le
\frac{\s{Q}_n|e(\Gamma_n)|}{\binom n2},
\label{eq:edge_hit_quant}
\end{align}
for every $n$. In addition, for \emph{every} planted graph $\Gamma_n$ and every query set $\mathbb{Q}_n$ of size $\s{Q}_n$, the all-ones transcript $\mathbf{1}\in\{0,1\}^{\s{Q}_n}$ satisfies $\pr_{\calH_0,\mathbb{Q}_n}(\mathbf{1})=q^{\s{Q}_n}$ and $\pr_{\calH_1,\mathbb{Q}_n}(\mathbf{1})\ge q^{\s{Q}_n}$ (planting can only add edges), whence
\begin{align}
\s{R}^\star_{n,\s{Q}_n}
=
1-\sup_{|\mathbb{Q}_n|\le\s{Q}_n}d_{\s{TV}}\p{\pr_{\calH_1,\mathbb{Q}_n},\pr_{\calH_0,\mathbb{Q}_n}}
\ge
\inf_{|\mathbb{Q}_n|\le\s{Q}_n}\sum_{x}\min\ppp{\pr_{\calH_0,\mathbb{Q}_n}(x),\pr_{\calH_1,\mathbb{Q}_n}(x)}
\ge q^{\s{Q}_n}.
\label{eq:bounded_budget}
\end{align}
Thus strong detection always requires $\s{Q}_n\to\infty$, whatever the planted graph is. Conversely, for globally dense planted graphs a bounded number of queries already gives weak detection; see Corollary~\ref{cor:globally_dense} below.
\end{remark}

Roughly speaking, the edge-hit lower bound is tight (up to polylogarithmic factors) whenever detection can be achieved by identifying a single planted edge. This occurs when the graph $\Gamma_n$ contains many dense local structures, or more generally when there exists a subgraph $H_n\subseteq \Gamma_n$ that appears sufficiently frequently and can be detected via brute-force scanning; this will be made precise in the upper bounds section. Typical examples include clique-like or biclique-like graphs, where dense substructures allow detection as soon as one planted region is queried. In contrast, for hub-dominated or bounded-cover graphs, detection requires aggregating information across many edges, and the edge-hit bound is no longer tight. This motivates the need for a more refined lower bound that captures structural properties beyond the total number of edges. To this end, we next present a general structural lower bound based on a vertex-cover decomposition of $\Gamma_n$.

To present this bound, we introduce a few important notations. For each $n$, choose a vertex cover $U_n\subseteq v(\Gamma_n)$, and denote by $\tau_n\triangleq |U_n|$ its size. Define $W_n\triangleq v(\Gamma_n)\setminus U_n$ and $H_n\triangleq \Gamma_n[U_n]$ to be a \emph{cover-induced subgraph} of $\Gamma_n$ on the vertex cover $U_n$.
For each $w\in W_n$, define its neighborhood inside the chosen cover by
\begin{align}
N_n(w)\triangleq N_{\Gamma_n}(w)\cap U_n.
\end{align}
For each integer $r\in\{1,\ldots,\tau_n\}$, define
\begin{align}
\Theta_{r,n}
\triangleq
\sum_{w,w'\in W_n}
\binom{|N_n(w)\cap N_n(w')|}{r}.
\label{eq:Theta_r_cover_reduction_equiv}
\end{align}
For each $n$, let $\mathbb{Q}_n\subseteq \binom{[n]}{2}$ be a deterministic non-adaptive query set of size $\s{Q}_n\triangleq |\mathbb{Q}_n|$. For each vertex $i\in[n]$, define its query degree
\begin{align}
t_i^{(n)}(\mathbb{Q}_n)
\triangleq
\abs{\{j\in[n]\setminus\{i\}:\{i,j\}\in\mathbb{Q}_n\}},
\end{align}
and for each integer $r\in\{1,\ldots,\tau_n\}$ set
\begin{align}
\zeta_{r,n}(\mathbb{Q}_n)
\triangleq
\sum_{i=1}^n
\frac{(t_i^{(n)})_r}{(n)_r},
\label{eq:zeta_r_cover_reduction}
\end{align}
where $(a)_r\triangleq a(a-1)\cdots(a-r+1)$ denotes the falling factorial. Finally, define an auxiliary planted model $\pr_{H_n,\mathbb{Q}_n}$ as follows: sample a uniformly random injective map $\varphi_n:U_n\hookrightarrow [n]$, and plant only the graph $H_n$ on the image $\varphi_n(U_n)$, and then draw every remaining ambient edge independently as $\s{Bern}(q)$. We have the following result.
\begin{comment}
\begin{theorem}[Cover-based lower bound]
\label{thm:cover_reduction_query_lb}
Fix $q\in(0,1)$, let $(\Gamma_n)_{n\ge 1}$ be a sequence of finite simple graphs, and $H_n$ be a cover induced subgraph of $\Gamma_n$ on a vertex cover $U_n\subset v(\Gamma_n)$, as defined above. Let $\mathbb{Q}_n\subseteq \binom{[n]}{2}$ be a deterministic non-adaptive query set of size $\s{Q}_n\triangleq |\mathbb{Q}_n|$. Assume that
\begin{align}
d_{\s{TV}}(\pr_{H_n,\mathbb{Q}_n},\pr_{\calH_0,\mathbb{Q}_n})&\to 0,\label{eq:assump_core_undetectable}\\
\frac{1}{(n-\tau_n)^2}
\sum_{r=1}^{\tau_n}
(q^{-1}-1)^r
\Theta_{r,n}\zeta_{r,n}
&\to 0,
\label{eq:assump_remainder_negligible}
\end{align}
as $n\to\infty$. Then $d_{\s{TV}}(\pr_{\calH_1,\mathbb{Q}_n},\pr_{\calH_0,\mathbb{Q}_n})\to 0$, and weak detection is impossible.
\end{theorem}
\end{comment}
\begin{theorem}[Core reduction via vertex covers]
\label{thm:cover_reduction_query_lb}
Fix $q\in(0,1)$, let $(\Gamma_n)_{n\ge 1}$ be a sequence of finite simple graphs,
and let $H_n$ be the cover-induced subgraph of $\Gamma_n$ on a vertex cover
$U_n\subseteq v(\Gamma_n)$, as defined above. Let $\s{Q}_n$ be a query budget. Assume that
\begin{align}
&\sup_{\mathbb{Q}_n\subseteq \binom{[n]}{2}:\ |\mathbb{Q}_n|\le \s{Q}_n}
d_{\s{TV}}\p{
\pr_{H_n,\mathbb{Q}_n},
\pr_{\calH_0,\mathbb{Q}_n}
}
\to 0,\label{eq:assump_core_undetectable}\\
&\sup_{\mathbb{Q}_n\subseteq \binom{[n]}{2}:\ |\mathbb{Q}_n|\le \s{Q}_n}
\frac{1}{(n-\tau_n)^2}
\sum_{r=1}^{\tau_n}
(q^{-1}-1)^r
\Theta_{r,n}\zeta_{r,n}(\mathbb{Q}_n)
\to 0.
\label{eq:assump_remainder_negligible}
\end{align}
Then weak detection is impossible with $\s{Q}_n$ non-adaptive queries.
%\begin{align}
%\sup_{\mathbb{Q}_n\subseteq \binom{[n]}{2}:\ |\mathbb{Q}_n|\le \s{Q}_n}
%d_{\s{TV}}\p{
%\pr_{\calH_1,\mathbb{Q}_n},
%\pr_{\calH_0,\mathbb{Q}_n}
%}
%\to 0.
%\end{align}
%Consequently,
%\begin{align}
%\s{R}_{n,\s{Q}_n}^{\star}\to 1,
%\end{align}
%and weak detection is impossible with $\s{Q}_n$ non-adaptive queries.
\end{theorem}

Theorem~\ref{thm:cover_reduction_query_lb} has several noteworthy features. The theorem is based on a structural decomposition of the planted graph $\Gamma_n$ via a vertex cover $U_n$. The vertex cover isolates a \emph{core subgraph} $H_n=\Gamma_n[U_n]$, while the remaining vertices $W_n=v(\Gamma_n)\setminus U_n$ form an independent set. Consequently, every edge of $\Gamma_n$ is either contained in the core or connects a vertex in $W_n$ to the core. The theorem compares the original planted model to an auxiliary model in which only the core $H_n$ is planted, while all edges incident to $W_n$ are generated according to the null model. The first condition, \eqref{eq:assump_core_undetectable}, requires that the core itself be statistically indistinguishable from the null under the available query budget. The second condition, \eqref{eq:assump_remainder_negligible}, quantifies the additional information carried by the attachment layer. Here, $\Theta_{r,n}$ measures the amount of overlap between neighborhoods of vertices in $W_n$, while $\zeta_{r,n}$ measures the ability of the query mechanism to simultaneously probe $r$ vertices of the core. Their weighted combination therefore captures the total second-moment contribution of the attachment layer.

Theorem~\ref{thm:cover_reduction_query_lb} should therefore be viewed as a \emph{reduction principle} rather than as a standalone lower bound. It separates the problem of proving undetectability into two simpler tasks: establishing that a suitably chosen core $H_n$ is itself undetectable, and showing that the remaining attachment layer contributes negligibly. This viewpoint makes it possible to combine lower bounds tailored to different graph families within a single unified framework. For example, clique-like or biclique-like cores can be handled using the edge-hit lower bound, whereas for hub-dominated graph families the core is often trivial and the attachment layer completely determines the statistical threshold. More generally, different lower-bound techniques may be used to verify the first condition, while the second condition provides a universal characterization of the contribution of the attachment layer. As will be demonstrated later, this decomposition can yield substantially sharper lower bounds than the direct edge-hit argument, and in many examples an appropriate choice of the vertex cover eliminates all but a few terms in the sum defining \eqref{eq:assump_remainder_negligible}, leading to simple and nearly tight impossibility results.

\begin{remark}[Recursive use]\label{rem:recursive}
The auxiliary model $\pr_{H_n,\mathbb{Q}_n}$ is itself a planted-subgraph model, with $H_n$ in place of $\Gamma_n$. Hence condition~\eqref{eq:assump_core_undetectable} may be verified by applying Theorem~\ref{thm:cover_reduction_query_lb} once more, now to $H_n$ and a vertex cover of $H_n$, and so on. This is used in the proof of Corollary~\ref{cor:triangles_full_observation_main}, where the core is a matching and is handled by Corollary~\ref{prop:pendant_layer_tight_main}; in this way all lower bounds in this paper are self-contained.
\end{remark}

At first sight, condition~\eqref{eq:assump_remainder_negligible} may appear difficult to verify because it depends on the query mechanism through the quantities $\zeta_{r,n}(\mathbb{Q}_n)$. Fortunately, a simple combinatorial argument shows that $\zeta_{r,n}(\mathbb{Q}_n)\le 2\s{Q}_n/n$ for every $r$ and every query set $\mathbb{Q}_n$. Consequently, \eqref{eq:assump_remainder_negligible} admits a simple sufficient condition depending only on graph-theoretic parameters of $\Gamma_n$ and the query budget $\s{Q}_n$, completely eliminating the optimization over $\mathbb{Q}_n$. This simplified form is presented in Corollary~\ref{cor:weighted_cover_reduction_simplified_main} below and is the version used throughout the corollaries and examples in the remainder of the paper.

\subsection{Upper bounds}\label{subsec:queryAlg}

We now present four detection procedures, each tailored to a different structural feature of the planted graph $\Gamma_n$. All procedures are non-adaptive. We start with the \emph{witness scan test}, which attempts to detect a fixed subgraph $H\subseteq \Gamma_n$ by querying a random subset of vertices and checking whether the induced queried graph contains a copy of $H$; see Algorithm~\ref{alg:witness_scan}. We have the following result.
%\paragraph{Witness-scan procedure.}
%Given a query budget $Q$ and a witness graph $H$ with $v=|v(H)|$, let
%\begin{align}
%M \triangleq \max\ppp{m\in\mathbb{N}:\binom{m}{2}\le Q}.
%\end{align}
%Sample a subset $S\subseteq [n]$ uniformly at random with $|S|=M$, and query all $\binom{M}{2}$ edges inside $S$. Let $\s{G}_S$ be the induced queried subgraph. The test outputs $1$ if $\s{G}_S$ contains a copy of $H$, and outputs $0$ otherwise.
\begin{theorem}[Scan test upper bound]
\label{thm:scan_upper}
Fix $q\in(0,1)$ and let $H_n\subseteq \Gamma_n$ be a sequence of witness subgraphs without isolated vertices, with $v_n\triangleq |v(H_n)|$ and $e_n\triangleq |e(H_n)|$. Let $\s{Q}_n\le\binom n2$ be the query budget, let $\mathbb{Q}_n$ be the query set generated by Algorithm~\ref{alg:witness_scan}, and define
\begin{align}
M_n\triangleq \max\ppp{m\in\mathbb{N}:\binom{m}{2}\le \s{Q}_n}\le n,
\qquad
p_n\triangleq\frac{M_n}{n}.
\end{align}
Denote
\begin{align}
\mu_n
&\triangleq
N(H_n,\Gamma_n)\,p_n^{v_n},\\
\Delta_n
&\triangleq
\sum_{t=1}^{v_n-1}
N_t(H_n,\Gamma_n)\,p_n^{2v_n-t}.
\end{align}
Then
\begin{align}
\s{R}_n(\calA_{\s{scan}};\mathbb{Q}_n)\le M_n^{v_n}q^{e_n}+2\exp\p{-\frac{\mu_n^2}{2(\mu_n+\Delta_n)}}.
\label{eqn:RiskScanTestMain}
\end{align}
In particular, if
\begin{align}
M_n^{v_n}q^{e_n}\to 0,
\qquad
\mu_n\to\infty,
\qquad
\Delta_n=o(\mu_n^2),
\end{align}
then the non-adaptive scan test $\calA_{\s{scan}}$ achieves strong detection, i.e., $\s{R}_n(\calA_{\s{scan}};\mathbb{Q}_n)\to 0$, as $n\to\infty$.
\end{theorem}
The scan test succeeds by finding a \emph{planted witness} inside the queried subgraph. The first condition controls false positives under the null (requiring $H_n$ to be sufficiently dense), while the second ensures that at least one planted copy of $H_n$ falls entirely inside the queried set. The third condition guarantees that different planted copies do not overlap too heavily. Ignoring lower-order effects, the scan test operates at the scale
\begin{align}
\s{Q}_n \asymp \frac{n^2}{N(H_n,\Gamma_n)^{2/v_n}},
\end{align}
and is therefore most effective when $\Gamma_n$ contains many dense substructures. In particular, for clique-like or biclique-like graphs, the scan test matches the edge-hit lower bound up to polylogarithmic factors.

When the dense substructure is a single large complete bipartite subgraph, a bipartite variant of the scan test is both simpler to analyze and more economical: it spends its queries across a cut between two small random vertex sets, one tuned to each side of the biclique, rather than inside one vertex set. The resulting \emph{biclique-scan test} is given in Algorithm~\ref{alg:biclique_scan}.
\begin{theorem}[Biclique-scan upper bound]
\label{thm:biclique_scan}
Fix $q\in(0,1)$ and set $L\triangleq\ceil{3/\log(1/q)}+3$ and $C_0\triangleq 64L$. Suppose that $\Gamma_n$ contains a complete bipartite subgraph $K_{a_n,b_n}$ with $C_0\log n\le a_n\le b_n$. Let
\begin{align}
s_n\triangleq\ceil{L\log n},\qquad
M_n\triangleq\ceil{\frac{8ns_n}{a_n}},\qquad
m_n\triangleq\ceil{\frac{8ns_n}{b_n}}.
\label{eq:bscan_params}
\end{align}
Then $m_n\le M_n\le n/2$, the biclique-scan test $\calA_{\s{bscan}}$ of Algorithm~\ref{alg:biclique_scan} uses
\begin{align}
\s{Q}_{\s{bscan}}\triangleq M_nm_n\le \min\ppp{\frac{n^2}{4},\ \frac{1024L^2\,n^2\log^2 n}{a_nb_n}}
\label{eq:bscan_budget}
\end{align}
non-adaptive queries, and $\s{R}_n(\calA_{\s{bscan}})\le 3n^{-3}$ for all sufficiently large $n$. In particular, strong detection is possible with $O(n^2\log^2 n/(a_nb_n))$ queries.
\end{theorem}
Since $\calK_{k}\supseteq K_{\lfloor k/2\rfloor,\lceil k/2\rceil}$, Theorem~\ref{thm:biclique_scan} applies to any planted graph containing a clique on $k_n\ge 3C_0\log n$ vertices, with $O(n^2\log^2 n/k_n^2)$ queries, and to complete bipartite graphs $K_{a_n,b_n}$ of arbitrary aspect ratio as soon as $a_n\ge C_0\log n$. In both cases this matches the edge-hit lower bound of Theorem~\ref{thm:edge_hit_lower_bound} up to a $\log^2 n$ factor. The proof (Section~\ref{subsec:proof_bscan}) is elementary and does not use the Janson inequality, because the planted biclique is present deterministically once each side is hit by the corresponding random set. The price is that, like the witness scan, the biclique scan runs in quasi-polynomial time $n^{O(\log n)}$.

We next present the \emph{degree-on-a-cut} test, a polynomial-time algorithm that exploits high-degree vertices; see Algorithm~\ref{alg:degcut_unsat}. We have the following result.
\begin{theorem}[Degree-on-a-cut upper bound]
\label{thm:degree_cut_general_unsat_main}
Fix $q\in(0,1)$ and let $\Gamma_n$ be a sequence of graphs. For $d\ge 1$, define
\begin{align}
m_d(\Gamma_n)
&\triangleq \big|\{u\in v(\Gamma_n):\deg_{\Gamma_n}(u)\ge d\}\big|,
\\
\kappa(\Gamma_n)
&\triangleq \max_{d\ge 1} m_d(\Gamma_n) d^2.
\end{align}
There exist constants $C_1,C_2,\s{C}>0$ depending only on $q$ such that the following holds. Fix $d\in\{1,\dots,d_{\s{max}}(\Gamma_n)\}$ and define
\begin{align}
m \triangleq \left\lceil C_1\frac{n^2}{d^2}\log n\right\rceil,
\quad
M \triangleq \min\left\{
n-m,\left\lceil C_2\frac{n}{m_d(\Gamma_n)}\log n\right\rceil
\right\},
\quad
\s{Q}_{\s{cut}}\triangleq Mm.
\label{eq:degcut_params_unsat}
\end{align}
We say that the degree level $d$ is \emph{feasible} if
\begin{align}
m\le \frac{n}{2};
\label{eq:degcut_feas_unsat}
\end{align}
this holds whenever $d^2\ge 4C_1 n\log n$ and fails whenever $d^2<2C_1n\log n$. If $d$ is feasible, then the degree-on-a-cut test $\calA_{\s{cut}}$ in Algorithm~\ref{alg:degcut_unsat} uses the non-adaptive query set $\mathbb{Q}_n=S\times U$ of size $|\mathbb{Q}_n|=\s{Q}_{\s{cut}}$, satisfies $\s{R}_n(\calA_{\s{cut}};\mathbb{Q}_n)\le 4n^{-2}$ for all sufficiently large $n$, and
\begin{align}
\s{Q}_{\s{cut}}
\le
\s{C}\,\frac{n^3\log^2 n}{d^2\max\{m_d(\Gamma_n),\log n\}}.
\label{eq:degcut_budget_explicit}
\end{align}
\end{theorem}
The DegreeCut algorithm detects vertices whose degree toward a random subset is significantly larger than expected under the null. This method is particularly effective for \emph{hub-dominated} structures, such as stars or bounded-cover graphs, where a small number of vertices carry a large fraction of the edges. In these regimes, DegreeCut matches the structural lower bounds up to logarithmic factors. However, for dense homogeneous graphs (e.g., cliques), it is suboptimal.

To better appreciate Theorem~\ref{thm:degree_cut_general_unsat_main}, define the \emph{saturated high-degree mass}
\begin{align}
\kappa_{\s{sat}}(\Gamma_n)
\triangleq
\max_{\substack{1\le d\leq d_{\max}(\Gamma_n)\\ d\ \text{feasible}}}
d^2\max\{m_d(\Gamma_n),\log n\},
\label{eq:kappa_sat_def}
\end{align}
with the convention $\kappa_{\s{sat}}(\Gamma_n)\triangleq 0$ (and $n^3\log^2n/\kappa_{\s{sat}}\triangleq\infty$) if no degree level is feasible; recall that every feasible $d$ satisfies $d^2\ge 2C_1n\log n$. Then \eqref{eq:degcut_budget_explicit}, applied at a maximizing degree level, shows that
\begin{align}
\s{Q}_n
\ge
\s{C}\,\frac{n^3\log^2 n}{\kappa_{\s{sat}}(\Gamma_n)}
\label{eq:degcut_saturated_budget}
\end{align}
is a sufficient condition for strong detection. Note that $\max_{d\ \text{feasible}}m_d(\Gamma_n)d^2\le\kappa_{\s{sat}}(\Gamma_n)\le \max\{\kappa(\Gamma_n),\ d_{\max}(\Gamma_n)^2\log n\}$, and that the first inequality is an equality whenever a level $d$ maximizing $m_d(\Gamma_n)d^2$ over feasible levels has $m_d(\Gamma_n)\ge\log n$.

Finally, we propose the \emph{edge count} test, defined as follows. Let $\mathbb{Q}_n$ be a set of $\s{Q}$ edges drawn uniformly at random from $\binom{[n]}{2}$. Define the statistic
\begin{align}
\s{X}_{\s{count}}\triangleq \sum_{e\in\mathbb{Q}_n}\mathds{1}\{e\in e(\s{G})\},
\end{align}
and consider the test $\calA_{\s{count}}\triangleq \mathds{1}\{\s{X}_{\s{count}}\ge q\s{Q} + \lambda\sqrt{q(1-q)\s{Q}}\}$, for any $\lambda=\lambda_n\to\infty$.
\begin{theorem}[Edge-count upper bound]\label{thm:edgecount_query}
Fix $q\in(0,1)$. If
\begin{align}
\frac{1-q}{q}\cdot \s{Q}_n\cdot \p{\frac{|e(\Gamma_n)|}{\binom{n}{2}}}^2\to\infty,
\qquad\text{equivalently}\qquad
\frac{\s{Q}_n|e(\Gamma_n)|^2}{n^4}\to\infty,
\label{eq:edgecount_snr_thm}
\end{align}
then $\s{R}_n(\calA_{\s{count}};\mathbb{Q}_n)\to0$, as $n\to\infty$, for every $\lambda_n\to\infty$ with $\lambda_n^2=o\big(\s{Q}_n(|e(\Gamma_n)|/\binom n2)^2\big)$.
\end{theorem}
(The factor $(1-q)/q=\chi^2(1\|q)$ is the $\chi^2$-divergence between a planted edge and a null edge; in our model planted edges are present with probability $p=1$. The proof in Section~\ref{subsec:count_proof} is written for a general $p\in(q,1]$.)

As we will show in the next section, each of these algorithms is optimal (up to polylogarithmic factors) in a different regime.

\begin{algorithm}[t!]
\caption{\textsc{Witness-Scan}$(Q,H)$}\label{alg:witness_scan}
\begin{algorithmic}[1]
\Require Budget $Q$, witness graph $H$ with $v=|v(H)|$ and $e=|e(H)|$
\State Set $M \gets \max\{m\in\mathbb N:\binom{m}{2}\leq Q\}$
\State Sample $S\subseteq[n]$ uniformly at random subject to $|S|=M$
\State Query all pairs $\{i,j\}\subseteq S$; let $\s{G}_S$ be the induced queried graph on $S$
\If{$\s{G}_S$ contains a copy of $H$ as a subgraph}
    \State \Return $1$
\Else
    \State \Return $0$
\EndIf
\end{algorithmic}
\end{algorithm}

\begin{algorithm}[t!]
\caption{\textsc{DegreeCut}$(m,M,d)$}
\label{alg:degcut_unsat}
\begin{algorithmic}[1]
\State Sample $S\subset[n]$ uniformly at random with $|S|=m$.
\State Sample $U\subset[n]\setminus S$ uniformly at random with $|U|=M$, independently of $S$ (so $U=[n]\setminus S$ when $M=n-m$).
\State Query all pairs $(u,s)$ with $u\in U$ and $s\in S$ (total $Mm$ queries).
\State For each $u\in U$ compute $D(u)\triangleq \sum_{s\in S}\mathds{1}\{(u,s)\in e(\s{G})\}$, and for each $s\in S$ compute $D'(s)\triangleq \sum_{u\in U}\mathds{1}\{(u,s)\in e(\s{G})\}$.
\State Set the thresholds $T_U \triangleq qm + \frac{(1-q)dm}{8n}$ and $T_S\triangleq qM+\frac{(1-q)d}{8}$.
\State Output $1$ if $\max_{u\in U}D(u)\ge T_U$, or if $M=n-m$ and $\max_{s\in S}D'(s)\ge T_S$; otherwise output $0$.
\end{algorithmic}
\end{algorithm}

\begin{algorithm}[t!]
\caption{\textsc{Biclique-Scan}$(m,M,s)$}
\label{alg:biclique_scan}
\begin{algorithmic}[1]
\State Sample $S\subset[n]$ uniformly at random with $|S|=m$, and $U\subset[n]\setminus S$ uniformly at random with $|U|=M$, independently of $S$.
\State Query all pairs $(u,s)$ with $u\in U$ and $s\in S$ (total $Mm$ queries); let $\s{G}_{U,S}$ be the queried bipartite graph.
\If{there exist $A\subseteq U$ and $B\subseteq S$ with $|A|=|B|=s$ such that every pair in $A\times B$ is an edge of $\s{G}_{U,S}$}
    \State \Return $1$
\Else
    \State \Return $0$
\EndIf
\end{algorithmic}
\end{algorithm}

\subsection{Corollaries and examples}

We now illustrate the general lower and upper bounds on several natural graph families. These examples show that the upper bounds are each optimal in different structural regimes.

We begin with a convenient simplification of Theorem~\ref{thm:cover_reduction_query_lb}. A simple combinatorial bound shows that $\zeta_{r,n}\le 2\s{Q}_n/n$ for every $r$, and therefore the weighted remainder condition in Theorem~\ref{thm:cover_reduction_query_lb} may often be reduced to a more explicit expression.
\begin{corollary}
\label{cor:weighted_cover_reduction_simplified_main}
Consider the same setting in Theorem~\ref{thm:cover_reduction_query_lb}. Then the condition in \eqref{eq:assump_remainder_negligible} can be replaced with
\begin{align}
\frac{\s{Q}_n}{n(n-\tau_n)^2}
\sum_{r=1}^{\tau_n}
(q^{-1}-1)^r\Theta_{r,n}
\to 0.\label{eq:assump_weighted_simplified}
\end{align}
\end{corollary}
This is the main tool used below. In most examples, one chooses a vertex cover $U_n$ so that the induced core $H_n$ is extremely simple (often edgeless or a matching), and then computes the quantities $\Theta_{r,n}$ explicitly. In many cases only the first one or two terms in the sum over $r$ survive.

Before proceeding, it is worth noting that in the full-observation regime, where $\mathbb{Q} _n=\binom{[n]}{2}$, the three query algorithms introduced in Subsection~\ref{subsec:queryAlg} reduce to the corresponding optimal full-observation tests. In particular, the sufficient conditions in Theorems~\ref{thm:scan_upper}--\ref{thm:edgecount_query} simplify to
\begin{align}
\mu(\Gamma_n)=\Omega(\log n)\quad\s{or}\quad|e(\Gamma_n)|=\omega(n)\quad\s{or}\quad d_{\s{max}}(\Gamma_n)=\Omega(\sqrt{n\log n}).
\end{align}
These coincide (up to constant factors) with the (tight) statistical and computational thresholds established for the full-observation model in \cite{elimelech2025detecting}.

\paragraph{Hub-dominated graphs.}

Our first example corresponds to the regime where the graph is dominated by high-degree vertices. This includes, as a special case, arbitrary star forests. In this regime, the lower bound is captured by the cover-reduction theorem, while the matching upper bound is achieved by the degree-on-a-cut test.

\begin{corollary}[Pendant-leaf graphs]
\label{prop:pendant_layer_tight_main}
Fix $q\in(0,1)$ and let $(\Gamma_n)_{n\ge 1}$ be a sequence of graphs with a vertex cover $U_n$ such that:
\begin{enumerate}
\item $H_n=\Gamma_n[U_n]$ is edgeless;
\item every vertex $w\in W_n\triangleq v(\Gamma_n)\setminus U_n$ has exactly one neighbor in $U_n$.
\end{enumerate}
For each $u\in U_n$, define
\begin{align}
d_{W_n}(u)\triangleq |\{w\in W_n:\{u,w\}\in e(\Gamma_n)\}|,
\quad
\Delta_n\triangleq \max_{u\in U_n} d_{W_n}(u),
\quad
S_{2,n}\triangleq \sum_{u\in U_n} d_{W_n}^2(u).
\end{align}
If $\s{Q}_n=o\p{\frac{n^3}{S_{2,n}}}$ then weak detection is impossible. Conversely, suppose that some feasible degree level $d^\star$ (in the sense of Theorem~\ref{thm:degree_cut_general_unsat_main}) attains the maximum in $\kappa(\Gamma_n)=\max_{d\ge1}m_d(\Gamma_n)d^2$. Then $\s{Q}_n\ge \s{C}\frac{n^3}{\kappa(\Gamma_n)}\log^2 n$ implies strong detection. Moreover,
\begin{align}
\max_{d\ge 2}m_d(\Gamma_n)d^2
\le
S_{2,n}
\le
4\p{1+\lfloor \log_2 \Delta_n\rfloor}\kappa(\Gamma_n),
\label{eq:S2_vs_kappa_main}
\end{align}
and $\kappa(\Gamma_n)=\max_{d\ge2}m_d(\Gamma_n)d^2$ whenever a feasible maximizer $d^\star$ exists (feasible levels satisfy $d^\star\ge 2$). Consequently, under the above assumption the lower and upper bounds differ by at most a factor of order $\log \Delta_n\cdot \log^2 n$. In particular, if the values $d_{W_n}(u)$ are all of the same order, then the two bounds match up to a factor of $\log^2 n$.
\end{corollary}

This corollary is obtained by choosing the high-degree layer $U_n$ as the vertex cover. Since the induced core $H_n$ is edgeless, the core is automatically undetectable. Moreover, every vertex in $W_n$ contributes to exactly one singleton neighborhood inside the cover, so only the $r=1$ term survives in the cover-reduction lower bound. This yields the lower-bound scale
\begin{align}
\s{Q}_n \asymp \frac{n^3}{S_{2,n}},
\qquad
S_{2,n}=\sum_{u\in U_n} d_{W_n}^2(u).
\end{align}
On the upper-bound side, the degree-on-a-cut test is controlled by $\kappa_n=\max_d m_d(\Gamma_n)d^2$ which measures the mass of the high-degree layer. A dyadic decomposition shows that $S_{2,n}$ and $\kappa_n$ differ by at most a logarithmic factor, so the two bounds essentially coincide.

Star forests are a special case of this proposition: if
\begin{align}
\Gamma_n=\bigsqcup_{j=1}^{r_n} K_{1,d_{j,n}},
\end{align}
then one takes $U_n$ to be the set of star centers, and obtains
\begin{align}
S_{2,n}=\sum_{j=1}^{r_n} d_{j,n}^2.
\end{align}
In the equal-degree case $\Gamma_n=\bigsqcup_{j=1}^{r_n} K_{1,\Delta_n}$, this becomes
\begin{align}
S_{2,n}=r_n\Delta_n^2,\quad m_{\Delta_n}(\Gamma_n)=r_n,
\end{align}
and therefore
\begin{align}
\kappa _{\s{sat}}(\Gamma_n)
=
\Delta_n^2\max\{r_n,\log n\},
\end{align}
provided the level $d=\Delta_n$ is feasible (which holds if $\Delta_n^2\ge 4C_1 n\log n$); otherwise no degree level satisfies the feasibility condition and the degree-on-a-cut bound does not apply. Consequently, the sufficient query complexity becomes
\begin{align}
\s{Q}_n
\gtrsim
\frac{n^3\log^2n}
{\Delta_n^2\max\{r_n,\log n\}}.
\end{align}
Thus, we see that the lower and upper bounds match up to the $\log^2 n$ factor of the degree-on-a-cut test:
\begin{align}
\s{Q}_n \asymp \frac{n^3}{\Delta_n^2\max\{r_n,\log n\}}
\qquad\text{up to polylogarithmic factors.}
\end{align}
Finally, note that in the full-observation regime, where $ \s{Q}_n=\binom{n}{2}$, this condition reduces to $\Delta_n\gtrsim\sqrt{n\log n}$, independently of the number of stars $r_n$, recovering the optimal maximum-degree threshold established in~\cite{elimelech2025detecting} (when $r_n\Delta_n=\omega(n)$ the count test detects under full observation even if $\Delta_n\ll\sqrt{n\log n}$, in agreement with the edge-count threshold of~\cite{elimelech2025detecting}).

\paragraph{Sparse low-density structures.} The next examples show that for certain sparse families, the problem is already impossible under full observation. In such cases, no query-based procedure can be sharp, since even seeing the entire graph does not help.

\begin{corollary}[Paths]
\label{cor:path_weighted_main}
Fix $q\in(0,1)$ and let $\Gamma_n=P_{k_n}$ be the path on $k_n$ vertices. If $\frac{\s{Q}_n k_n}{n^3}\to 0$,
then weak detection is impossible. In particular, if $k_n=o(n)$, then weak detection is impossible even under full observation.
\end{corollary}
Choose every other vertex of the path as a cover. Then the induced core is edgeless, and each uncovered vertex has at most two neighbors in the cover. As a result, only the terms $r=1,2$ contribute in the cover-reduction bound, both at total scale $O(k_n)$. Under full observation this gives impossibility whenever $k_n=o(n)$.
\begin{corollary}[Disjoint unions of triangles]
\label{cor:triangles_full_observation_main}
Fix $q\in(0,1)$ and let
\begin{align}
\Gamma_n=\bigsqcup_{i=1}^{m_n} K_3,
\qquad
k_n\triangleq |v(\Gamma_n)|=3m_n.
\end{align}
If $\s{Q}_nk_n=o(n^3)$, then weak detection is impossible. In particular, if $k_n=o(n)$ then weak detection is impossible even under full observation.
\end{corollary}
For each triangle, place two vertices in the cover and one outside. Then the induced core becomes a matching of size $m_n$, which is a star forest with all degrees equal to one, and is therefore undetectable with $\s{Q}_n=o(n^3/m_n)$ queries by Corollary~\ref{prop:pendant_layer_tight_main} (this is an instance of the recursive use of Theorem~\ref{thm:cover_reduction_query_lb}, Remark~\ref{rem:recursive}). The remainder term is of order $\s{Q}_nk_n/n^3$, since only the terms $r=1,2$ contribute. Hence the full planted graph is undetectable as well.

\begin{corollary}[Bounded-degree trees]
\label{cor:trees_full_observation_main}
Fix $q\in(0,1)$ and let $\Gamma_n=T_n$ be a tree on $k_n$ vertices. Assume that there exists a constant $D<\infty$ such that $d_{\max}(T_n)\le D$ for all $n$. If $\s{Q}_nk_n=o(n^3)$, then weak detection is impossible. In particular, if $k_n=o(n)$, then weak detection is impossible even under full observation.
\end{corollary}
Choose the smaller side of a bipartition of the tree as the cover. Then the induced core is edgeless. Since degrees are uniformly bounded by $D$, only finitely many terms in the cover-reduction sum contribute, and each contributes at scale $O(k_n)$, which gives the condition $\s{Q}_nk_n=o(n^3)$. Under full observation this yields impossibility whenever $k_n=o(n)$.

\paragraph{Bounded-cover graphs.} We next obtain a general impossibility result for all graph families with bounded vertex-cover number, such as a star graph.

\begin{corollary}[Bounded vertex cover]
\label{cor:bounded_vertex_cover_weighted_main}
Fix $q\in(0,1)$, and let $(\Gamma_n)$ be a sequence of graphs such that $\tau(\Gamma_n)\le \tau_0$ for some constant $\tau_0<\infty$. Let $\Delta_n\triangleq d_{\max}(\Gamma_n)$. If
\begin{align}
\s{Q}_n=o\p{\frac{n^3}{\Delta_n^2}},
\end{align}
then weak detection is impossible.
\end{corollary}
The proof uses a minimum vertex cover $U_n$. Since $|U_n|$ is uniformly bounded, the core $H_n=\Gamma_n[U_n]$ has bounded size and is therefore undetectable. The remainder term is controlled by $\Delta_n^2$, since each uncovered vertex can only attach to the bounded cover. This shows that bounded-cover families have lower-bound scale $n^3/\Delta_n^2$. This lower bound is complemented by the degree-on-a-cut upper bound whenever the high-degree mass is concentrated. In particular, if $\kappa(\Gamma_n)\asymp \Delta_n^2$ and the level $d=\Delta_n$ is feasible (e.g., $\Delta_n^2\ge 4C_1n\log n$), then DegreeCut achieves strong detection with
\begin{align}
\s{Q}_n\gtrsim \frac{n^3}{\Delta_n^2}\log^2 n.
\end{align}

\paragraph{Clique-like structures.} For clique-like graphs, the direct edge-hit lower bound is the relevant one, and it is complemented by the scan upper bound.
\begin{corollary}[Cliques]
\label{cor:cliques_direct_main}
Fix $q\in(0,1)$ and let $\Gamma_n=K_{k_n}$ be the clique on $k_n$ vertices.
\begin{enumerate}
\item[(i)] If $\s{Q}_n=o\p{\frac{n^2}{k_n^2}}$, then weak detection is impossible.
\item[(ii)] If $k_n\ge 3C_0\log n$, with $C_0=C_0(q)$ as in Theorem~\ref{thm:biclique_scan}, then the biclique-scan test achieves strong detection with $O\p{\frac{n^2}{k_n^2}\log^2 n}$ queries.
\end{enumerate}
Hence, for $k_n\ge 3C_0\log n$, the query complexity of detecting a planted clique is $n^2/k_n^2$ up to a factor $\log^2 n$.
\end{corollary}
For cliques, the direct edge-hit lower bound already gives the correct scale, since
\begin{align}
|e(K_{k_n})|\asymp k_n^2.
\end{align}
On the upper-bound side, one can choose a smaller clique witness $H_n=K_{h_n}$ with $h_n\asymp C\log n$ and apply the scan test (this requires $k_n\ge h_n$, i.e. $k_n\gtrsim\log n$, which is also necessary for detection under full observation). This yields strong detection at query complexity
\begin{align}
\s{Q}_n\asymp \frac{n^2}{k_n^2}\cdot \mathrm{polylog}(n),
\end{align}
(a direct computation along the lines of the proof of Theorem~\ref{thm:biclique_scan} gives $\log^4 n$ for the witness scan, while the biclique scan with the bipartition $K_{\lfloor k_n/2\rfloor,\lceil k_n/2\rceil}\subseteq K_{k_n}$ gives $\log^2 n$, see Section~\ref{subsec:proof_cliques}), so the lower and upper bounds match up to polylogarithmic factors, coinciding with the classical results \cite{racz2020finding} (see also \cite{mardia2020finding,mardia2021space,mardia2024lowdegree,HuleihelMazumdarPal2024QueryPDS}).

\sloppy
\begin{corollary}[Clique with pendant leaves]
\label{cor:clique_pendant_main}
Fix $q\in(0,1)$ and let $\Gamma_n=K_{k_n}\cup
\{\Delta_n\text{ pendant leaves attached to each clique vertex}\}$. If
\begin{align}
\s{Q}_n
=
o\left(
\min\left\{
\frac{n^2}{k_n^2},
\,
\frac{n^3}{k_n\Delta_n^2}
\right\}
\right),
\end{align}
then weak detection is impossible.
\end{corollary}

Unlike the clique example, the direct edge-hit lower bound applied to the entire graph only yields
\begin{align}
\s{Q}_n
=
o\left(n^2/[k_n^2+k_n\Delta_n]
\right),
\end{align}
which can be substantially weaker. The improvement comes from Theorem~\ref{thm:cover_reduction_query_lb}: choosing the clique as the core separates the dense part of the graph from the pendant leaves. The clique itself is handled by the edge-hit lower bound, while the contribution of the leaves is controlled independently through the attachment term in Theorem~\ref{thm:cover_reduction_query_lb}. Thus, the decomposition recognizes that the pendant edges do not carry independent statistical information and avoids paying for their total number.

The upper bounds exhibit exactly the same two competing mechanisms. Using the clique as a witness, the biclique-scan test (Theorem~\ref{thm:biclique_scan} applied to $K_{\lfloor k_n/2\rfloor,\lceil k_n/2\rceil}\subseteq K_{k_n}\subseteq\Gamma_n$, for $k_n\ge 3C_0\log n$) achieves strong detection with
\begin{align}
\s{Q}_n
\asymp n^2/k_n^2,
\end{align}
up to a $\log^2n$ factor. On the other hand, the degree-on-a-cut test exploits the large degrees of the clique vertices and succeeds with
\begin{align}
\s{Q}_n
\asymp
\frac{n^3}
{(k_n+\Delta_n)^2
\max\{k_n,\log n\}}
\log^2 n,
\end{align}
provided the level $d=k_n-1+\Delta_n$ is feasible.
Consequently, the query complexity is determined by whichever of these two mechanisms is stronger, and complements Corollary~\ref{cor:clique_pendant_main}. This example illustrates the advantage of the vertex-cover reduction in Theorem~\ref{thm:cover_reduction_query_lb}: by separating the statistically informative clique from the low-complexity attachment layer, it yields a substantially sharper lower bound than treating the planted graph as a whole. Moreover, it highlights the complementary nature of the scan and degree-on-a-cut algorithms, whose performances are governed by the dense core and the high-degree attachment layer, respectively.

\paragraph{Complete bipartite graphs.} Complete bipartite graphs interpolate between the clique-like scan regime and the bounded-cover degree regime.
\begin{corollary}[Complete bipartite graphs]
\label{prop:complete_bipartite_tradeoff_main}
Fix $q\in(0,1)$ and let $\Gamma_n = K_{a_n,b_n}$ with $1\le a_n\le b_n$.
\begin{enumerate}
\item[(i)] If $\s{Q}_n=o\p{\frac{n^2}{a_n b_n}}$, then weak detection is impossible.
\item[(ii)] If the level $d=b_n$ is feasible in the sense of Theorem~\ref{thm:degree_cut_general_unsat_main} (e.g., $b_n^2\ge 4C_1n\log n$), then
\begin{align}
\s{Q}_n\ge \s{C}\frac{n^3}{a_n b_n^2}\log^2 n
\end{align}
implies strong detection via the degree-on-a-cut test.
\item[(iii)] If $a_n\ge C_0\log n$, with $C_0=C_0(q)$ as in Theorem~\ref{thm:biclique_scan}, then the biclique-scan test achieves strong detection with $O\p{\frac{n^2}{a_nb_n}\log^2 n}$ queries. In particular, for $a_n\ge C_0\log n$ the lower bound in (i) is tight up to a $\log^2n$ factor, for every aspect ratio $b_n/a_n$.
\end{enumerate}
\end{corollary}

The cover-reduction theorem gives an additional lower bound which is much stronger than (i) when the smaller side is small.
\begin{corollary}[Complete bipartite graphs with a small side]
\label{cor:biclique_small_side}
Fix $q\in(0,1)$ and let $\Gamma_n = K_{a_n,b_n}$ with $1\le a_n\le b_n$. If
\begin{align}
\frac{\s{Q}_n\,(q^{-a_n}-1)\,b_n^2}{n^3}\to 0,
\qquad\text{in particular if}\qquad
\s{Q}_n=o\p{q^{a_n}\frac{n^3}{b_n^2}},
\label{eq:biclique_small_side}
\end{align}
then weak detection is impossible.
\end{corollary}
For $a_n=o(\log n)$ one has $q^{-a_n}=n^{o(1)}$, so Corollary~\ref{cor:biclique_small_side} and part (ii) of Corollary~\ref{prop:complete_bipartite_tradeoff_main} match up to a factor $n^{o(1)}$; for bounded $a_n$ they match up to $\log^2 n$ (Corollary~\ref{prop:biclique_bounded_left_main}). For $a_n\ge C\log n$ the bound \eqref{eq:biclique_small_side} is polynomially weaker than (i).
There are two different upper-bound mechanisms here: 1) If the smaller side is below logarithmic size, the graph behaves like a hub-dominated structure, and DegreeCut is the relevant algorithm. 2) If both sides are at least logarithmic, then one can scan for a logarithmic complete bipartite witness across a cut, and the biclique-scan test matches the lower bound up to a $\log^2 n$ factor, regardless of the aspect ratio. In particular, the regime $a_n\gtrsim\log n$ behaves like the clique case, while the small-side regime behaves like a star.
\begin{corollary}[Complete bipartite graphs with bounded smaller side]
\label{prop:biclique_bounded_left_main}
Fix $q\in(0,1)$ and let $\Gamma_n = K_{a_n,b_n}$ with $1\le a_n\le b_n$, where $a_n\le a_0$ for some constant $a_0<\infty$. If $\s{Q}_n=o\p{\frac{n^3}{b_n^2}}$, then weak detection is impossible, while if $\s{Q}_n\gtrsim \frac{n^3}{b_n^2}\log^2 n$, then strong detection is possible, provided the level $d=b_n$ is feasible in the sense of Theorem~\ref{thm:degree_cut_general_unsat_main}.
\end{corollary}

To conclude, the picture for complete bipartite graphs is now the following.
\begin{corollary}[Summary for complete bipartite graphs]
\label{cor:biclique_summary}
Fix $q\in(0,1)$ and let $\Gamma_n=K_{a_n,b_n}$ with $1\le a_n\le b_n$.
\begin{enumerate}
\item[(a)] If $a_n\ge C_0\log n$, then weak detection is impossible when $\s{Q}_n=o(n^2/(a_nb_n))$, and strong detection is possible with $O(n^2\log^2 n/(a_nb_n))$ queries.
\item[(b)] If $a_n=o(\log n)$ and the level $d=b_n$ is feasible, then weak detection is impossible when $\s{Q}_n=o(q^{a_n}n^3/b_n^2)=o(n^{3-o(1)}/b_n^2)$, and strong detection is possible with $O(n^3\log^2 n/(a_nb_n^2))=O(n^{3+o(1)}/b_n^2)$ queries. If $a_n=O(1)$ the two bounds differ by $O(\log^2n)$.
\end{enumerate}
\end{corollary}
Thus the only regime that remains open up to polynomial factors is $a_n=\Theta(\log n)$, more precisely $\varepsilon\log n\le a_n\le C_0\log n$ for a fixed $\varepsilon>0$, where the available lower bounds are $\max\{n^2/(a_nb_n),\,q^{a_n}n^3/b_n^2\}$ and the available upper bound is $n^3\log^2n/(a_nb_n^2)$ (and $n^2\log^2 n/(a_nb_n)$ once $a_n\ge C_0\log n$).

\paragraph{Globally dense graphs.} At the opposite end of the spectrum, when $|e(\Gamma_n)|$ is a constant fraction of all pairs, the edge-count test is optimal and the whole question reduces to whether the budget grows at all.
\begin{corollary}[Globally dense graphs]
\label{cor:globally_dense}
Fix $q\in(0,1)$ and let $(\Gamma_n)$ satisfy $|e(\Gamma_n)|\ge c\binom n2$ for a constant $c>0$. Then:
\begin{enumerate}
\item[(i)] strong detection is possible if and only if $\s{Q}_n\to\infty$, and it is achieved by the edge-count test;
\item[(ii)] weak detection is possible with a single query: the test $\mathds{1}\{\s{G}_e=1\}$ for an arbitrary pair $e$ has risk at most $1-c(1-q)$.
\end{enumerate}
\end{corollary}
More generally, the edge-count test needs $\s{Q}_n\gg n^4/|e(\Gamma_n)|^2$ queries, which exceeds the edge-hit scale $n^2/|e(\Gamma_n)|$ by the factor $n^2/|e(\Gamma_n)|$; the count test is therefore optimal (up to constants) exactly for globally dense planted graphs.

The examples above illustrate the general picture: the \emph{scan test} and its bipartite variant, the \emph{biclique-scan test}, are sharp for clique-like and dense local structures (in particular for every planted graph containing a biclique with both sides of size $\gtrsim\log n$); the \emph{degree-on-a-cut test} is sharp for hub-dominated and bounded-cover structures; the \emph{edge-count test} is sharp for globally dense structures. At the opposite extreme, for sparse low-density structures such as bounded-degree trees, paths, or sparse unions of triangles, detection is already impossible under full observation.

\section{Proofs of lower bounds}\label{sec:lower_bounds}

\subsection{Preliminaries}

We start with some preliminaries. Our goal is to establish lower bounds on the optimal risk, thereby ruling out the possibility of successful detection. We begin by introducing the likelihood ratio,
\begin{align}
\s{L}(\s{G}) \triangleq \frac{\mathrm{d}\P_{\calH_1}}{\mathrm{d}\P_{\calH_0}}(\s{G}),
\label{eqn:LIKlihood}
\end{align}
which is the Radon--Nikodym derivative of $\P_{\calH_1}$ with respect to $\P_{\calH_0}$. It is well known (see, e.g., \cite[Theorem 2.2]{tsybakov2004introduction}) that the test minimizing the risk $\s{R} _n$ is the likelihood ratio test
\begin{align}
\phi^{\star}(\s{G}) \triangleq
\begin{cases}
1, & \text{if } \s{L}(\s{G}) \ge 1,\\
0, & \text{otherwise},
\end{cases}
\end{align}
and the corresponding optimal risk is given by
\begin{align}
\s{R}_n^\star &= \inf_{\calA_n}\pr_{\calH_0}\p{\calA_n(\s{G})=1}
+
\pr_{\calH_1}\p{\calA_n(\s{G})=0} \\
&=1 - d_{\s{TV}}(\P_{\calH_0}, \P_{\calH_1}).
\end{align}
Recalling that
\begin{align}
\chi^2(\P_{\calH_1}\|\P_{\calH_0}) = \E_{\calH_0}[\s{L}(\s{G})^2] - 1,
\end{align}
standard inequalities relating total variation and $\chi^2$ divergence (see, e.g., \cite[Sec. 2]{tsybakov2004introduction}) imply that
\begin{align}
\s{R}_n^\star
\ge
1 - \frac{1}{2}\sqrt{\chi^2(\P_{\calH_1}\|\P_{\calH_0})}.
\label{eqn:lowerBoundSecond}
\end{align}
In particular, if $\E_{\calH_0}[\s{L}(\s{G})^2] = 1 + o(1)$, then $\chi^2(\P_{\calH_1}\|\P_{\calH_0}) \to 0$, which implies $\s{R}_n^\star \to 1$, and hence weak detection is impossible. This shows that controlling the second moment of the likelihood ratio under the null is sufficient to establish impossibility. In the remainder of this section, we derive bounds on $\E_{\calH_0}[\s{L}(\s{G})^2]$ and identify conditions under which it converges to one.

The following result is a straightforward characterization of the chi-square distance between the null and alternative distributions, given any non-adaptive query mechanism $\mathbb{Q}$.
\begin{lemma}
\label{lem:chi2_identity_dense}
Fix $q\in(0,1)$ and a deterministic query set $\mathbb{Q}\subseteq \binom{[n]}{2}$ of size $|\mathbb{Q}|=\s{Q}$.
Let $\pr_{\calH_0,\mathbb{Q}}$ and $\pr_{\calH_1,\mathbb{Q}}$ be the laws of the transcript
$\s{G}_{\mathbb{Q}}=(\s{G}_e)_{e\in\mathbb{Q}}\in\{0,1\}^Q$ under $\calH_0$ and $\calH_1$, respectively. Let $\pi$ denote the uniform distribution over labeled copies of $\Gamma$ into $\mathcal K_n$.
Let $\Gamma,\Gamma'\stackrel{\s{i.i.d.}}{\sim}\pi$ be two independent random planted copies, and define
\begin{align}
\s{X}_{\mathbb{Q}}\triangleq \big|e(\Gamma)\cap e(\Gamma')\cap \mathbb{Q}\big|.
\end{align}
Then
\begin{align}
1+\chi^2(\pr_{\calH_1,\mathbb{Q}}\|\pr_{\calH_0,\mathbb{Q}})
=
\bE_{\Gamma,\Gamma'}\pp{q^{-\s{X}_{\mathbb{Q}}}}.
\label{eq:chi2_identity_dense}
\end{align}
\end{lemma}

\begin{proof}
Under $\mathcal H_0$ the queried edges $\{\s{G}_e\}_{e\in\mathbb{Q}}$ are i.i.d.\ $\s{Bern}(q)$.
Under $\mathcal H_1$, condition on the planted copy $\Gamma$:
if $e\in e(\Gamma)$ then $\s{G}_e\equiv 1$, while if $e\notin e(\Gamma)$ then $\s{G}_e\sim\s{Bern}(q)$ independently.
Thus the likelihood ratio on the transcript is
\begin{align}
\calL_{\mathbb{Q}}(\s{G}_{\mathbb{Q}})
=
\bE_{\Gamma\sim\pi}\pp{\prod_{e\in \mathbb{Q}\cap e(\Gamma)}\frac{\mathds{1}\{\s{G}_e=1\}}{q}}.
\end{align}
Therefore,
\begin{align}
1+\chi^2(\pr_{\calH_1,\mathbb{Q}}\|\pr_{\calH_0,\mathbb{Q}})
&=
\bE_{\pr_{\calH_0,\mathbb{Q}}}[\calL_{\mathbb{Q}}^2]\\
&=
\bE_{\Gamma,\Gamma'\sim\pi}
\bE_{\pr_{\calH_0,\mathbb{Q}}}\pp{
\prod_{e\in\mathbb{Q}\cap e(\Gamma)}\frac{\mathds{1}\{\s{G}_e=1\}}{q}\cdot
\prod_{e\in\mathbb{Q}\cap e(\Gamma')}\frac{\mathds{1}\{\s{G}_e=1\}}{q}
}.
\end{align}
Fix $\Gamma,\Gamma'$. For any queried edge $e\in\mathbb{Q}$:
\begin{itemize}
\item if $e\in e(\Gamma)\cap e(\Gamma')$, the factor is $\mathds{1}\{\s{G}_e=1\}/q^2$, whose $\pr_{\calH_0,\mathbb{Q}}$-expectation equals $q/q^2=1/q$;
\item otherwise the factor has expectation $1$.
\end{itemize}
Hence the inner expectation equals $\prod_{e\in \mathbb{Q}\cap e(\Gamma)\cap e(\Gamma')} (1/q)=q^{-\s{X}_{\mathbb{Q}}}$, and \eqref{eq:chi2_identity_dense} follows.
\end{proof}

\subsection{Proof of Theorem~\ref{thm:edge_hit_lower_bound}}
Fix an arbitrary deterministic non-adaptive query set
$\mathbb{Q}_n\subseteq\binom{[n]}{2}$ with $|\mathbb{Q}_n|\le \s{Q}_n$. We will couple the transcript under $\mathcal H_1$ to that under $\mathcal H_0$ on an event of high probability.
Let $\Gamma^\star$ be the planted copy under $\mathcal H_1$.
Define the event
\begin{align}
\mathcal E \triangleq \{\mathbb{Q}\cap e(\Gamma^\star)=\emptyset\},
\end{align}
i.e., none of the queried edges is planted. On the event $\mathcal E$, every queried edge $e\in\mathbb{Q}$ is \emph{not} a planted edge. Under the alternative, all non-planted edges are distributed as $\s{Bern}(q)$ exactly as under $\mathcal H_0$, and are independent across queried edges. Hence the conditional law of $\s{G}_{\mathbb{Q}}$ given $\mathcal E$ under $\mathcal H_1$ coincides with the unconditional law of $\s{G}_{\mathbb{Q}}$ under $\mathcal H_0$, namely,
\begin{align}
\pr_{\calH_1,\mathbb{Q}}(\cdot\vert \mathcal E)=\pr_{\calH_0,\mathbb{Q}}(\cdot).
\end{align}
Therefore,
\begin{align}
d_{\s{TV}}(\pr_{\calH_0,\mathbb{Q}},\pr_{\calH_1,\mathbb{Q}})
\leq \pr_{\calH_1,\mathbb{Q}}(\mathcal E^c),
\end{align}
which implies that
\begin{align}
\s{R}^\star_{\mathbb{Q}}\ge 1-d_{\s{TV}}(\pr_{\calH_0,\mathbb{Q}},\pr_{\calH_1,\mathbb{Q}})
\ge 1-\pr_{\calH_1,\mathbb{Q}}(\mathcal E^c).
\end{align}
Thus it suffices to show that $\pr_{\calH_1,\mathbb{Q}}(\mathcal E^c)\to 0$. Since $\Gamma^\star$ is a uniform random copy of $\Gamma$ in $\mathcal K_n$, its planted edge set $e(\Gamma^\star)$ has size $m_n\triangleq|e(\Gamma_n)|$ and, by symmetry over the $N\triangleq \binom{n}{2}$ unordered vertex pairs, each fixed edge $e\in\binom{[n]}{2}$ satisfies
\begin{align}
\pr[e\in e(\Gamma^\star)]=\frac{m_n}{N}.
\end{align}
Therefore, by a union bound,
\begin{align}
\pr_{\calH_1,\mathbb{Q}_n}(\mathcal E^c)
&=
\pr\pp{\exists e\in\mathbb{Q}_n:\; e\in e(\Gamma^\star)}\\
&\le
\sum_{e\in\mathbb{Q}_n}\pr[e\in e(\Gamma^\star)]\\
&=
|\mathbb{Q}_n|\cdot \frac{m_n}{N}\\
&\le
\s{Q}_n\cdot \frac{m_n}{N}.
\label{eqn:boundBudgError}
\end{align}
Under \eqref{eq:edge_hit_condition}, the right-hand side tends to $0$. Since \eqref{eqn:boundBudgError} holds uniformly over all deterministic non-adaptive query sets $\mathbb{Q}_n$ with $|\mathbb{Q}_n|\le \s{Q}_n$, we obtain
\begin{align}
\sup_{|\mathbb{Q}_n|\le \s{Q}_n}
d_{\s{TV}}\p{\pr_{\calH_1,\mathbb{Q}_n},\pr_{\calH_0,\mathbb{Q}_n}}
\to 0.
\end{align}
Consequently, $\s{R}_{n,\s{Q}_n}^{\star}\to 1$, which completes the proof.

\subsection{Proof of Theorem~\ref{thm:cover_reduction_query_lb}}

We recall the following definitions. For each $n$, we let $\Gamma_n=(V_n,E_n)$, and $|V_n|\leq n$, where $V_n = v(\Gamma_n)$ and $E_n = e(\Gamma_n)$. For each $n$, choose a vertex cover $U_n\subseteq V_n$, and denote by $\tau_n\triangleq |U_n|$ its size. Define $W_n\triangleq V_n\setminus U_n$ and $H_n\triangleq \Gamma_n[U_n]$ to be a \emph{cover-induced subgraph} of $\Gamma_n$ on the vertex cover $U_n$.
For each $w\in W_n$, define its neighborhood inside the chosen cover by
\begin{align}
N_n(w)\triangleq N_{\Gamma_n}(w)\cap U_n.
\end{align}
For every nonempty subset $R\subseteq U_n$, define
\begin{align}
\mathcal W_n(R)\triangleq \{w\in W_n:R\subseteq N_n(w)\}.
\end{align}
For each integer $r\in\{1,\ldots,\tau_n\}$, define
\begin{align}
\Theta_{r,n}
\triangleq
\sum_{\substack{R\subseteq U_n\\ |R|=r}}
|\mathcal W_n(R)|^2=
\sum_{w,w'\in W_n}
\binom{|N_n(w)\cap N_n(w')|}{r}.
\label{eq:Theta_r_cover_reduction}
\end{align}
For each $n$, let $\mathbb{Q}_n\subseteq \binom{[n]}{2}$ be a deterministic non-adaptive query set of size $\s{Q}_n\triangleq |\mathbb{Q}_n|$.

Fix $n$; to lighten notation we remove the subscript $n$ from the various variables below. If $W=\emptyset$, then $U=V_n$ and hence $H=\Gamma_n$. In that case the full planted model and the auxiliary planted model coincide $\pr_{\calH_1,\mathbb{Q}_n}=\pr_{H_n,\mathbb{Q}_n}$. Therefore the conclusion follows immediately from \eqref{eq:assump_core_undetectable}. Thus, for the remainder of the proof, we may and shall assume that $W\neq\emptyset$, in particular, $\tau\leq n-1$. Our goal is to prove
\begin{align}
d_{\s{TV}}(\pr_{\calH_1,\mathbb{Q}_n},\pr_{H_n,\mathbb{Q}_n})\to 0.
\label{eq:goal_full_vs_aux}
\end{align}
Indeed, once \eqref{eq:goal_full_vs_aux} is established, the triangle inequality gives
\begin{align}
d_{\s{TV}}(\pr_{\calH_1,\mathbb{Q}_n},\pr_{\calH_0,\mathbb{Q}_n})
\le
d_{\s{TV}}(\pr_{\calH_1,\mathbb{Q}_n},\pr_{H_n,\mathbb{Q}_n})
+
d_{\s{TV}}(\pr_{H_n,\mathbb{Q}_n},\pr_{\calH_0,\mathbb{Q}_n}),
\end{align}
and the right-hand side tends to $0$ by \eqref{eq:goal_full_vs_aux} and \eqref{eq:assump_core_undetectable}. Now, under the full planted model, let $\Phi_n:V_n\hookrightarrow [n]$ be the random injective embedding of $\Gamma_n$ into $K_n$. Define the ordered image of the chosen cover by
\begin{align}
A_n&\triangleq (a_u)_{u\in U},\\
a_u&\triangleq \Phi_n(u),\quad u\in U.
\end{align}
Under the auxiliary planted model, the cover vertices are embedded by the uniformly random injective map $\varphi_n:U\hookrightarrow [n]$. Define similarly
\begin{align}
A_n'\triangleq (\varphi_n(u))_{u\in U},\quad u\in U.
\end{align}
Since in both models the map from $U$ to $[n]$ is uniformly random and injective, the random ordered tuples $A_n$ and $A_n'$ have the same distribution. Fix an ordered $\tau$-tuple $A=(a_u)_{u\in U}$ of distinct vertices of $[n]$. We define two conditional laws on the queried transcript.

\begin{itemize}
\item $\pr^\Gamma_{A,\mathbb{Q}}$ is the law of the queried transcript under the full planted model, conditioned on $\Phi_n(u)=a_u$, for all $u\in U$.
\item $\pr^H_{A,\mathbb{Q}}$ is the law of the queried transcript under the auxiliary planted model, conditioned on $\varphi_n(u)=a_u$, for all $u\in U$.
\end{itemize}
By conditioning on the random ordered image of the cover, both original laws are mixtures of these conditional laws:
\begin{align}
\pr_{\calH_1,\mathbb{Q}_n}
&=
\bE_{A_n}\bigl[\pr^\Gamma_{A_n,\mathbb{Q}}\bigr],
\\
\pr_{H_n,\mathbb{Q}_n}
&=
\bE_{A_n}\bigl[\pr^H_{A_n,\mathbb{Q}}\bigr].
\label{eq:mixture_representation_cover}
\end{align}
Now fix the tuple $A$. Since $U$ is a vertex cover of $\Gamma_n$, the set $W=V_n\setminus U$ is an independent set in $\Gamma_n$. Consequently, in both models, the planted edges \emph{inside} the set $A$ are exactly the edges of $H=\Gamma_n[U]$, and in both models, no edges are planted with both endpoints outside $A$. Therefore the conditional laws $\pr^\Gamma_{A,\mathbb{Q}}$ and $\pr^H_{A,\mathbb{Q}}$ can differ only on queried edges joining $A$ to vertices outside $A$. Define
\begin{align}
\calB_A
\triangleq
\ppp{
v\in[n]\setminus A:
\exists u\in U\ \s{s.t.}\ \{a_u,v\}\in \mathbb{Q}},
\end{align}
and
\begin{align}
M_A\triangleq |\calB_A|.
\end{align}
Also define
\begin{align}
\mathcal E_A
\triangleq
\ppp{
\{a_u,v\}\in \mathbb{Q}:
u\in U,\ v\in \calB_A
}.
\end{align}
Thus $\mathcal E_A$ is precisely the set of queried edges on which the two conditional laws may differ.

Fix $A$ as above. Let $\Gamma^A$ denote a random planted copy of $\Gamma_n$ under the full model conditioned on the event $\Phi_n(u)=a_u$, for all $u\in U$. For any subset $S\subseteq \mathcal E_A$, define
\begin{align}
p_A(S)\triangleq \pr\pp{S\subseteq e(\Gamma^A)}.
\end{align}
We claim that
\begin{align}
1+\chi^2\p{\pr^\Gamma_{A,\mathbb{Q}},\pr^H_{A,\mathbb{Q}}}
=
\sum_{S\subseteq \mathcal E_A}
(q^{-1}-1)^{|S|}p^2_A(S).
\label{eq:conditional_chi2_cover}
\end{align}
Indeed, to that end, note that under the auxiliary model conditioned on $A$, every queried edge in $\mathcal E_A$ is independent $\s{Bern}(q)$, because the auxiliary model only plants edges of $H$ inside $A$ and does not plant any edge between $A$ and its complement. Conditioned on a realization of the random planted copy $\Gamma^A$, every queried edge in $e(\Gamma^A)\cap \mathcal E_A$ is forced to be present, while every queried edge in $\mathcal E_A\setminus e(\Gamma^A)$ is $\s{Bern}(q)$ and independent of the others. Let $Y=(Y_e)_{e\in\mathcal E_A}$ denote the queried transcript restricted to $\mathcal E_A$. The likelihood ratio of $\pr^\Gamma_{A,\mathbb{Q}}$ with respect to $\pr^H_{A,\mathbb{Q}}$ on the coordinates indexed by $\mathcal E_A$ is
\begin{align}
\calL_A(Y)
=
\bE_{\Gamma^A}
\pp{
\prod_{e\in e(\Gamma^A)\cap \mathcal E_A}
\frac{\mathds{1}\{Y_e=1\}}{q}
}.
\end{align}
Therefore
\begin{align}
1+\chi^2\p{\pr^\Gamma_{A,\mathbb{Q}},\pr^H_{A,\mathbb{Q}}}
=
\bE_{\pr^H_{A,\mathbb{Q}}}\pp{\calL^2_A(Y)}.
\end{align}
Introduce two independent conditioned copies $\Gamma^A$ and $\widetilde\Gamma^A$. Expanding the square and using Fubini's theorem,
\begin{align}
1+\chi^2\p{\pr^\Gamma_{A,\mathbb{Q}},\pr^H_{A,\mathbb{Q}}}
&=
\bE_{\Gamma^A\indep\widetilde\Gamma^A}
\bE_{\pr^H_{A,\mathbb{Q}}}
\pp{
\prod_{e\in e(\Gamma^A)\cap \mathcal E_A}
\frac{\mathds{1}\{Y_e=1\}}{q}
\prod_{e\in e(\widetilde\Gamma^A)\cap \mathcal E_A}
\frac{\mathds{1}\{Y_e=1\}}{q}
}.
\end{align}
Fix an edge $e\in\mathcal E_A$. Under $\pr^H_{A,\mathbb{Q}}$, the variable $Y_e$ is $\s{Bern}(q)$. Hence:
\begin{itemize}
\item if $e\in e(\Gamma^A)\cap e(\widetilde\Gamma^A)$, then the contribution of this edge to the inner expectation equals
\begin{align}
\bE_{\s{Bern}(q)}
\pp{\frac{\mathds{1}\{Y_e=1\}}{q^2}}
=
\frac{q}{q^2}
=
\frac1q.
\end{align}
\item otherwise the contribution equals $1$.
\end{itemize}
Since the queried edges in $\mathcal E_A$ are independent under $\pr^H_{A,\mathbb{Q}}$, the inner expectation is
\begin{align}
q^{-|e(\Gamma^A)\cap e(\widetilde\Gamma^A)\cap \mathcal E_A|}.
\end{align}
Thus
\begin{align}
1+\chi^2\p{\pr^\Gamma_{A,\mathbb{Q}},\pr^H_{A,\mathbb{Q}}}
=
\bE_{\Gamma^A\indep\widetilde\Gamma^A}
\pp{
q^{-|e(\Gamma^A)\cap e(\widetilde\Gamma^A)\cap \mathcal E_A|}
}.
\end{align}
Now use the identity
\begin{align}
q^{-m}=(1+(q^{-1}-1))^m
=
\sum_{S\subseteq [m]}(q^{-1}-1)^{|S|},
\end{align}
applied edge-by-edge. This yields
\begin{align}
q^{-|e(\Gamma^A)\cap e(\widetilde\Gamma^A)\cap \mathcal E_A|}
=
\sum_{S\subseteq \mathcal E_A}
(q^{-1}-1)^{|S|}
\mathds{1}\{S\subseteq e(\Gamma^A)\cap e(\widetilde\Gamma^A)\}.
\end{align}
Taking expectations and using independence of $\Gamma^A$ and $\widetilde\Gamma^A$, we obtain
\begin{align*}
1+\chi^2\p{\pr^\Gamma_{A,\mathbb{Q}},\pr^H_{A,\mathbb{Q}}}
&=
\sum_{S\subseteq \mathcal E_A}
(q^{-1}-1)^{|S|}
\pr\pp{S\subseteq e(\Gamma^A)}^2
\\
&=
\sum_{S\subseteq \mathcal E_A}
(q^{-1}-1)^{|S|}p^2_A(S).
\end{align*}
This proves \eqref{eq:conditional_chi2_cover}. Next, we upper bound $p_A(S)$. Fix $S\subseteq \mathcal E_A$. For each $v\in \calB_A$, define
\begin{align}
\calR_v(S)\triangleq \{u\in U:\{a_u,v\}\in S\}.
\end{align}
Let
\begin{align}
V^+(S)&\triangleq \{v\in \calB_A:\calR_v(S)\neq\emptyset\},\\
s(S)&\triangleq |V^+(S)|.
\end{align}
Thus $V^+(S)$ is the set of outside vertices actually touched by the edge-set $S$. Under the conditioning on $A$, the vertices of $W$ are embedded by a uniformly random injective map $\psi:W\hookrightarrow [n]\setminus A$. For every nonempty subset $R\subseteq U$, define
\begin{align}
\mathcal W(R)\triangleq \{w\in W:R\subseteq N_n(w)\}.
\end{align}
Suppose now that the event $S\subseteq e(\Gamma^A)$ occurs. Then for each $v\in V^+(S)$ there must exist some vertex $w_v\in \mathcal W(\calR_v(S))$ such that $\psi(w_v)=v$. Since $\psi$ is injective and the vertices $v\in V^+(S)$ are distinct, the chosen vertices $w_v$ must also be distinct. Therefore the event $\{S\subseteq e(\Gamma^A)\}$ implies the existence of an ordered family of distinct vertices $(w_v)_{v\in V^+(S)}$ such that
\begin{align}
w_v\in \mathcal W(\calR_v(S))
\quad\text{and}\quad
\psi(w_v)=v
\quad\text{for all }v\in V^+(S).
\end{align}
Now, the number of admissible ordered families $(w_v)_{v\in V^+(S)}$ is at most
\begin{align}
\prod_{v\in V^+(S)} |\mathcal W(\calR_v(S))|.
\end{align}
For each fixed admissible ordered family, since $\psi$ is a uniformly random injective map from $W$ to $[n]\setminus A$, the probability that $\psi(w_v)=v$ for all $v\in V^+(S)$ is exactly $1/(n-\tau)_{s(S)}$. Thus, by the union bound, we deduce
\begin{align}
p_A(S)
\le
\frac{\prod_{v\in V^+(S)} |\mathcal W(\calR_v(S))|}{(n-\tau)_{s(S)}}.
\end{align}
We now use the elementary inequality
\begin{align}
(m)_s\ge \p{\frac{m}{e}}^s
\quad\forall\;m\ge 1,\ 0\leq s\leq m.
\end{align}
Applying this with $m=n-\tau$ and $s=s(S)$ yields
\begin{align}
(n-\tau)_{s(S)}\ge \p{\frac{n-\tau}{e}}^{s(S)}.
\end{align}
Therefore
\begin{align}
p_A(S)
\le
\prod_{v\in V^+(S)}
\frac{e |\mathcal W(\calR_v(S))|}{n-\tau}.\label{eq:pA_sq_bound_cover}
\end{align}
Thus, if we set $c_q\triangleq q^{-1}-1$, then using \eqref{eq:conditional_chi2_cover} and \eqref{eq:pA_sq_bound_cover}, we get
\begin{align}
1+\chi^2\p{\pr^\Gamma_{A,\mathbb{Q}},\pr^H_{A,\mathbb{Q}}}
\le
\sum_{S\subseteq \mathcal E_A}
c_q^{|S|}
\prod_{v\in V^+(S)}
\frac{e^2|\mathcal W(\calR_v(S))|^2}{(n-\tau)^2}.\label{eq:pA_sq_bound_covertt}
\end{align}
Let us simplify the upper bound at the right-hand side of \eqref{eq:pA_sq_bound_covertt}. Observe that a set $S\subseteq \mathcal E_A$ is uniquely determined by the local choices
\begin{align}
\calR_v(S)&\subseteq N_A(v),\\
N_A(v)&\triangleq \{u\in U:\{a_u,v\}\in \mathbb{Q}\},
\end{align}
for any $v\in \calB_A$. Also, notice that
\begin{align}
|S|=\sum_{v\in \calB_A} |\calR_v(S)|.
\end{align}
Therefore the sum over $S$ factorizes into a product over $v\in \calB_A$ as follows
\begin{align}
1+\chi^2\p{\pr^\Gamma_{A,\mathbb{Q}},\pr^H_{A,\mathbb{Q}}}
\le\prod_{v\in \calB_A}
\p{
1+
\frac{e^2}{(n-\tau)^2}
\sum_{\emptyset\neq R\subseteq N_A(v)}
c_q^{|R|} |\mathcal W(R)|^2
}.
\end{align}
For convenience, let us extend the previous definition and set
\begin{align}
N_A(i)\triangleq \{u\in U:\{a_u,i\}\in \mathbb{Q}\},
\qquad i\in[n].
\end{align}
Then $N_A(i)=\emptyset$ whenever $i\notin \calB_A\cup A$, and in particular
\begin{align}
1+\chi^2\p{\pr^\Gamma_{A,\mathbb{Q}},\pr^H_{A,\mathbb{Q}}}
&\le
\prod_{i=1}^{n}
\p{
1+
\frac{e^2}{(n-\tau)^2}
\sum_{\emptyset\neq R\subseteq N_A(i)}
c_q^{|R|} |\mathcal W(R)|^2
}
\nonumber\\
&\le
\exp\pp{
\frac{e^2}{(n-\tau)^2}
\sum_{i=1}^{n}
\sum_{\emptyset\neq R\subseteq N_A(i)}
c_q^{|R|} |\mathcal W(R)|^2
},
\label{eq:chi2_bound_exponential_weighted}
\end{align}
where in the last inequality we used $1+x\le e^x$ for all $x\ge 0$. Define
\begin{align}
Z_{A_n}
\triangleq
\sum_{i=1}^{n}
\sum_{\emptyset\neq R\subseteq N_{A_n}(i)}
c_q^{|R|} |\mathcal W_n(R)|^2.
\label{eq:ZA_weighted_cover}
\end{align}
Then \eqref{eq:chi2_bound_exponential_weighted} yields
\begin{align}
1+\chi^2\p{\pr^\Gamma_{A_n,\mathbb{Q}},\pr^H_{A_n,\mathbb{Q}}}
\le
\exp\pp{\eta_n Z_{A_n}},
\label{eq:eta_cover_reduction0}
\end{align}
where
\begin{align}
\eta_n\triangleq \frac{e^2}{(n-\tau_n)^2}.
\label{eq:eta_cover_reduction}
\end{align}

Next, we analyze $Z_{A_n}$. Taking expectations in \eqref{eq:ZA_weighted_cover} and using Fubini's theorem,
\begin{align}
\bE[Z_{A_n}]
&=
\sum_{i=1}^{n}
\sum_{\emptyset\neq R\subseteq U_n}
c_q^{|R|} |\mathcal W_n(R)|^2
\pr\pp{R\subseteq N_{A_n}(i)}.
\label{eq:EZ_expand_cover}
\end{align}
Fix $i\in[n]$ and a nonempty set $R\subseteq U_n$ with $|R|=r$. Since the ordered tuple $(a_u)_{u\in R}$ is uniformly distributed over all ordered $r$-tuples of distinct vertices of $[n]$, and since $N_{\mathbb{Q}_n}(i)$ has cardinality $t_i^{(n)}$, we have
\begin{align}
\pr\pp{R\subseteq N_{A_n}(i)}
=
\frac{(t_i^{(n)})_r}{(n)_r}.
\label{eq:prob_R_in_NAi}
\end{align}
Substituting \eqref{eq:prob_R_in_NAi} into \eqref{eq:EZ_expand_cover}, grouping terms according to $r=|R|$, and using the definition of $\zeta_{r,n}$ in \eqref{eq:zeta_r_cover_reduction}, we obtain
\begin{align}
\bE[Z_{A_n}]
&=
\sum_{r=1}^{\tau_n}
c_q^r
\p{
\sum_{\substack{R\subseteq U_n\\ |R|=r}}
|\mathcal W_n(R)|^2
}
\p{
\sum_{i=1}^{n}\frac{(t_i^{(n)})_r}{(n)_r}
}
\nonumber\\
&=
\sum_{r=1}^{\tau_n}
c_q^r\Theta_{r,n}\zeta_{r,n}.
\label{eq:EZ_cover_exact}
\end{align}

Define
\begin{align}
\Lambda_n
\triangleq
\eta_n \bE[Z_{A_n}]
=
\frac{e^2}{(n-\tau_n)^2}
\sum_{r=1}^{\tau_n}
c_q^r\Theta_{r,n}\zeta_{r,n}.
\label{eq:Lambda_n_cover}
\end{align}
By assumption \eqref{eq:assump_remainder_negligible}, we have $\Lambda_n\to 0$.
Now define
\begin{align}
T_n\triangleq \sqrt{\frac{\bE[Z_{A_n}]}{\eta_n}}
\qquad\text{and}\qquad
\calG_n\triangleq \{Z_{A_n}\le T_n\}.
\end{align}
By Markov's inequality,
\begin{align}
\pr(\calG_n^c)
&=
\pr(Z_{A_n}>T_n)\\
&\le
\frac{\bE[Z_{A_n}]}{T_n}\\
&=
\sqrt{\eta_n \bE[Z_{A_n}]}
=
\sqrt{\Lambda_n}.
\end{align}
Hence, by \eqref{eq:Lambda_n_cover},
\begin{align}
\pr(\calG_n^c)\to 0.
\label{eq:bad_event_small_cover}
\end{align}

On the good event $\calG_n$, we have $Z_{A_n}\le T_n$, and so \eqref{eq:eta_cover_reduction0} implies that
\begin{align}
1+\chi^2\p{\pr^\Gamma_{A_n,\mathbb{Q}},\pr^H_{A_n,\mathbb{Q}}}
\le
e^{\eta_n T_n}.
\end{align}
But
\begin{align}
\eta_n T_n
=
\eta_n\sqrt{\frac{\bE[Z_{A_n}]}{\eta_n}}
=
\sqrt{\eta_n \bE[Z_{A_n}]}
=
\sqrt{\Lambda_n}
\to 0.
\end{align}
Hence
\begin{align}
\sup_{A\in \calG_n}
\chi^2\p{\pr^\Gamma_{A,\mathbb{Q}},\pr^H_{A,\mathbb{Q}}}
\to 0.
\end{align}
Therefore, uniformly over $A\in \calG_n$,
\begin{align}
d_{\s{TV}}(\pr^\Gamma_{A,\mathbb{Q}},\pr^H_{A,\mathbb{Q}})
\le
\frac12
\sqrt{
\chi^2\p{\pr^\Gamma_{A,\mathbb{Q}},\pr^H_{A,\mathbb{Q}}}
}
\to 0.\label{eqn:unifTota}
\end{align}
We now split the laws according to the event $\calG_n$. Define the good-event conditional mixtures
\begin{align}
\pr_{\calH_1,\mathbb{Q}}^{(g)}
&\triangleq
\bE\pp{\pr^\Gamma_{A_n,\mathbb{Q}}\vert \calG_n},
\\
\pr_{H_n,\mathbb{Q}}^{(g)}
&\triangleq
\bE\pp{\pr^H_{A_n,\mathbb{Q}}\vert \calG_n},
\end{align}
and similarly define the bad-event conditional mixtures $\pr_{\calH_1,\mathbb{Q}}^{(b)}$ and $\pr_{H_n,\mathbb{Q}}^{(b)}$. Then
\begin{align}
\pr_{\calH_1,\mathbb{Q}}
&=
(1-p)\cdot\pr_{\calH_1,\mathbb{Q}}^{(g)} + p\cdot \pr_{\calH_1,\mathbb{Q}}^{(b)},\\
\pr_{H_n,\mathbb{Q}}
&=
(1-p)\cdot\pr_{H_n,\mathbb{Q}}^{(g)} + p\cdot \pr_{H_n,\mathbb{Q}}^{(b)},
\end{align}
where $p\triangleq \pr(\calG_n^c)$. By convexity of the total variation distance,
\begin{align}
d_{\mathrm{TV}}(\pr_{\calH_1,\mathbb{Q}},\pr_{H_n,\mathbb{Q}})
&\le
(1-p) d_{\mathrm{TV}}(\pr_{\calH_1,\mathbb{Q}}^{(g)},\pr_{H_n,\mathbb{Q}}^{(g)})
+
p d_{\mathrm{TV}}(\pr_{\calH_1,\mathbb{Q}}^{(b)},\pr_{H_n,\mathbb{Q}}^{(b)})
\nonumber\\
&\le
d_{\mathrm{TV}}(\pr_{\calH_1,\mathbb{Q}}^{(g)},\pr_{H_n,\mathbb{Q}}^{(g)})
+
p\\
&\leq \pr(\calG_n^c)
+
\sup_{A\in \calG_n}
d_{\s{TV}}(\pr^\Gamma_{A,\mathbb{Q}},\pr^H_{A,\mathbb{Q}}).
\label{eq:tv_split_case2}
\end{align}
The first term tends to zero by \eqref{eq:bad_event_small_cover}, and the second term tends to zero by the uniform bound \eqref{eqn:unifTota}. Hence $d_{\s{TV}}(\pr_{\calH_1,\mathbb{Q}_n},\pr_{H_n,\mathbb{Q}_n})\to 0$, which proves \eqref{eq:goal_full_vs_aux}. Combining this with \eqref{eq:assump_core_undetectable} yields $d_{\s{TV}}(\pr_{\calH_1,\mathbb{Q}_n},\pr_{\calH_0,\mathbb{Q}_n})\to 0$.

\section{Proofs of upper bounds}\label{sec:upper_bounds}

\subsection{Query scan test}

Recall that $H$ is any fixed graph with no isolated vertices. Write $v_H\triangleq |v(H)|$ and $e_H\triangleq |e(H)|$. We will use $H$ as a \emph{witness} pattern that is guaranteed to appear inside the planted $\Gamma_n$, i.e., $H\subseteq \Gamma_n$ as an unlabeled subgraph.

Let us upper bound the $\s{Type}$-$\s{I}$ and $\s{Type}$-$\s{II}$ error probabilities, starting with the former. Under $\calH_0$, the induced queried graph satisfies $\s{G}_S\sim \calG(M,q)$. For any fixed injective map $\varphi:v(H)\hookrightarrow S$, the event that $\varphi$ embeds $H$ as a subgraph of $\s{G}_S$ requires that all $e$ edges of $H$ are present among the corresponding $e$ queried pairs; since edges are independent $\s{Bern}(q)$ under $\calH_0$,
\begin{align}
\pr_{\calH_0}\p{\text{$\varphi$ embeds $H$ into $\s{G}_S$}}=q^{e_H}. \label{eq:embed_prob_null}
\end{align}
The number of distinct copies of $H$ in the complete graph on $S$ is exactly $|\calS_{H}| = \binom{M}{v_H}\frac{v_H!}{|\s{Aut}(H)|}\leq(M)_v\le M^{v_H}$. Applying the union bound over all copies gives
\begin{align}
\pr_{\calH_0}(\calA_{\s{scan}}(\s{G})=1)\leq |\calS_{H}|\cdot q^{e_H}.\label{eqn:TypeIbound}
\end{align}

Next, we bound the $\s{Type}$-$\s{II}$ error probability. We shall use the following Janson-type inequality for a uniformly random subset of fixed size; it is proved in Appendix~\ref{app:janson_hypergeom} by transferring Janson's inequality from the product model.
\begin{lemma}[Janson-type inequality for a uniform $M$-subset]\label{lem:janson_hypergeom}
Let $1\le M\le n$, let $S\subseteq[n]$ be uniformly random with $|S|=M$, and let $\{B_c\}_{c\in\calC}$ be a finite family of distinct nonempty subsets of $[n]$. Let $I_c\triangleq \mathds{1}\{B_c\subseteq S\}$, $Z\triangleq \sum_{c\in\calC} I_c$, $p\triangleq M/n$, and
\begin{align}
\bar\mu\triangleq \sum_{c\in\calC}p^{|B_c|},
\qquad
\bar\Delta\triangleq \sum_{c\sim c'}p^{|B_c\cup B_{c'}|},
\end{align}
where $c\sim c'$ means $c\neq c'$ and $B_c\cap B_{c'}\neq\emptyset$ (the sum is over ordered pairs).
Then
\begin{align}
\pr(Z=0)\leq 2\exp\p{-\frac{\bar\mu^2}{2(\bar\mu+\bar\Delta)}}.
\label{eq:janson_form}
\end{align}
\end{lemma}

Under $\calH_1$, the planted copy of $\Gamma_n$ is obtained by a uniformly random injective embedding $\Phi:v(\Gamma_n)\hookrightarrow[n]$, and then all planted edges appear. Let $\calC\triangleq\calC(H,\Gamma_n)$ be the collection of witness vertex sets of $H$ in $\Gamma_n$, i.e., of $v_H$-subsets $W\subseteq v(\Gamma_n)$ such that $\Gamma_n[W]$ contains a copy of $H$; recall that $N(H,\Gamma_n)=|\calC|$. For $W\in\calC$ define the random $v_H$-subset $B_W\triangleq\Phi(W)\subseteq[n]$, the indicator $I_W\triangleq\mathds{1}\{B_W\subseteq S\}$, and the planted-hit count $Z_H\triangleq\sum_{W\in\calC}I_W$. If $Z_H\ge 1$, then $B_W\subseteq S$ for some $W\in\calC$; all pairs inside $S$ are queried, and all edges of the copy of $H$ inside $\Gamma_n[W]$ are planted and hence present, so $\s{G}_S$ contains a copy of $H$ and Algorithm~\ref{alg:witness_scan} outputs $1$. Therefore,
\begin{align}
\pr_{\calH_1}(\calA_{\s{scan}}=0)\leq\pr_{\calH_1}(Z_H=0).
\label{eq:miss_le_Z0}
\end{align}
Condition on $\Phi$. Then $\{B_W\}_{W\in\calC}$ is a deterministic family of distinct $v_H$-subsets of $[n]$ (distinct because $\Phi$ is injective), and $S$ is a uniform $M$-subset of $[n]$ independent of $\Phi$. For $W\neq W'$ we have $B_W\cap B_{W'}\neq\emptyset$ if and only if $W\cap W'\neq\emptyset$, and $|B_W\cup B_{W'}|=2v_H-|W\cap W'|$. Hence the quantities of Lemma~\ref{lem:janson_hypergeom}, with $p=M/n$, are
\begin{align}
\bar\mu&=\sum_{W\in\calC}p^{|B_W|}=N(H,\Gamma_n)\p{\frac{M}{n}}^{v_H}=\mu,
\nonumber\\
\bar\Delta&=\sum_{\substack{W,W'\in\calC\\ W\neq W',\ W\cap W'\neq\emptyset}}p^{2v_H-|W\cap W'|}
=\sum_{t=1}^{v_H-1}N_t(H,\Gamma_n)\p{\frac{M}{n}}^{2v_H-t}=\Delta,
\label{eq:mu_Delta_scan}
\end{align}
where $\mu$ and $\Delta$ are as in Theorem~\ref{thm:scan_upper}; they do not depend on $\Phi$. Lemma~\ref{lem:janson_hypergeom} therefore gives, uniformly in $\Phi$,
\begin{align}
\pr_{\calH_1}(Z_H=0\vert \Phi)
\leq 2\exp\p{-\frac{\mu^2}{2(\mu+\Delta)}},
\end{align}
and averaging over $\Phi$ and using \eqref{eq:miss_le_Z0},
\begin{align}
\pr_{\calH_1}(\calA_{\s{scan}}=0)\leq 2\exp\p{-\frac{\mu^2}{2(\mu+\Delta)}}.
\label{eq:typeII_janson}
\end{align}
Using \eqref{eqn:TypeIbound} and \eqref{eq:typeII_janson} we finally get that
\begin{align}
\s{R}_n(\calA_{\s{scan}}) \leq |\calS_{H}|\cdot q^{e_H}+2\exp\p{-\frac{\mu^2}{2(\mu+\Delta)}}
\le M^{v_H}q^{e_H}+2\exp\p{-\frac{\mu^2}{2(\mu+\Delta)}},
\label{eqn:RiskScanTest}
\end{align}
which is \eqref{eqn:RiskScanTestMain}. Therefore, if $M^{v_H} q^{e_H}\to0$, $\mu\to\infty$, and $\Delta=o(\mu^2)$, then $\s{R}_n(\calA_{\s{scan}})\to0$, as $n\to\infty$.

\subsection{Degree on a cut}

Throughout this subsection $d$ is a feasible degree level, so that $m\le n/2$, and $C_1,C_2$ are constants depending only on $q$; the proof shows that $C_1\ge 640/(1-q)^2$ and $C_2\ge 16$ suffice. Recall that $S$ is a uniform $m$-subset of $[n]$, that $U$ is a uniform $M$-subset of $[n]\setminus S$ with $M=\min\{n-m,\lceil C_2(n/m_d)\log n\rceil\}$, and that $1\le m_d\triangleq m_d(\Gamma_n)\le n$ because $d\le d_{\s{max}}(\Gamma_n)$. We shall use the following two consequences of the fact that the pair $(S,U)$ is uniformly distributed over all ordered pairs of disjoint subsets of $[n]$ of sizes $(m,M)$, a law which is invariant under permutations of $[n]$:
\begin{enumerate}
\item[(a)] marginally, $U$ is a uniform $M$-subset of $[n]$;
\item[(b)] for every fixed $u\in[n]$, conditionally on $\{u\in U\}$ the set $S$ is a uniform $m$-subset of $[n]\setminus\{u\}$, and conditionally on $\{u\in S\}$ the set $S\setminus\{u\}$ is a uniform $(m-1)$-subset of $[n]\setminus\{u\}$.
\end{enumerate}
We also record that, since $d\le n$,
\begin{align}
\frac{dm}{n}\ge C_1\frac{n\log n}{d}\ge C_1\log n,
\qquad
m\le 2C_1\frac{n^2}{d^2}\log n,
\qquad
d^2\ge 2C_1 n\log n,
\label{eq:degcut_basic_facts}
\end{align}
where the second inequality uses $C_1(n^2/d^2)\log n\ge C_1\log n\ge 1$ and the third is the feasibility condition $m\le n/2$.

\paragraph{Type-I error.}
Under $\calH_0$, all queried edges are independent $\s{Bern}(q)$. Fix $u\in U$. Then $D(u)\sim \s{Binomial}(m,q)$. Let $t_U\triangleq \frac{1-q}{8}\cdot\frac{dm}{n}$, so that $T_U=qm+t_U$. Since $d\le n$ we have $t_U\le m/8$, hence $q(1-q)m+t_U/3\le m$, and Lemma~\ref{lem:Bernstein} gives
\begin{align}
\pr_{\calH_0}(D(u)\ge T_U)
\leq \exp\p{-\frac{t_U^2}{2m}}
= \exp\p{-\frac{(1-q)^2d^2 m}{128\,n^2}}
\le \exp\p{-\frac{(1-q)^2C_1}{128}\log n}
\le n^{-5},
\label{eq:typeI_pointwise_degcut}
\end{align}
using $d^2m/n^2\ge C_1\log n$ and $C_1\ge 640/(1-q)^2$. A union bound over the $M\le n$ vertices of $U$ gives $\pr_{\calH_0}(\max_{u\in U}D(u)\ge T_U)\le n^{-4}$. If $M=n-m$, fix also $s\in S$; then $D'(s)\sim\s{Binomial}(M,q)$ with $M=n-m\ge n/2$. Let $t_S\triangleq (1-q)d/8$, so that $T_S=qM+t_S$. Since $d\le n\le 2M$ we have $t_S\le M/4$, hence $q(1-q)M+t_S/3\le M$, and Lemma~\ref{lem:Bernstein} gives
\begin{align}
\pr_{\calH_0}(D'(s)\ge T_S)
\le\exp\p{-\frac{t_S^2}{2M}}
\le\exp\p{-\frac{(1-q)^2d^2}{128\,n}}
\le\exp\p{-\frac{(1-q)^2C_1}{64}\log n}
\le n^{-5},
\end{align}
by \eqref{eq:degcut_basic_facts}. A union bound over the $m\le n$ vertices of $S$ gives $\pr_{\calH_0}(\max_{s\in S}D'(s)\ge T_S)\le n^{-4}$. Altogether,
\begin{align}
\pr_{\calH_0}(\calA_{\s{cut}}=1)\le 2n^{-4}.
\label{eqn:cutTypeI}
\end{align}

\paragraph{Type-II error.}
Condition on the planted embedding $\Phi$; all probabilities below are conditional on $\Phi$, and the final bounds do not depend on $\Phi$. Let $A_d\subseteq v(\Gamma_n)$ be the set of vertices of $\Gamma_n$ of degree at least $d$, so that $|A_d|=m_d$, and let $K_{\ge d}\triangleq\Phi(A_d)$. For $u\in K_{\ge d}$ let $N_u\triangleq\Phi\p{N_{\Gamma_n}(\Phi^{-1}(u))}$ be its planted neighborhood; then $|N_u|\ge d$, $u\notin N_u$, and every pair $\{u,v\}$ with $v\in N_u$ is a planted edge.

\smallskip\noindent\emph{Step 1: a hub in $U$ is detected.} Fix $u\in K_{\ge d}$. We claim that
\begin{align}
\pr_{\calH_1}\p{D(u)<T_U\,\vert\,u\in U}\le 2n^{-3}.
\label{eq:hub_in_U}
\end{align}
By (b), conditionally on $\{u\in U\}$ the set $S$ is a uniform $m$-subset of $[n]\setminus\{u\}$, so $Y\triangleq|N_u\cap S|\sim\s{Hypergeometric}(n-1,|N_u|,m)$ with $\bE[Y]=|N_u|m/(n-1)\ge dm/n$. By Lemma~\ref{lem:ChernoffChv} with $\delta=1/2$ and \eqref{eq:degcut_basic_facts},
\begin{align}
\pr\p{Y\le \frac{dm}{2n}}
\le\pr\p{Y\le \tfrac12\bE[Y]}
\le\exp\p{-\frac{\bE[Y]}{8}}
\le\exp\p{-\frac{dm}{8n}}
\le n^{-C_1/8}\le n^{-3}.
\label{eq:Y_good_degcut}
\end{align}
Conditionally on $S$ (and on $u\in U$), the $Y$ planted pairs from $u$ to $N_u\cap S$ are present, and the remaining $m-Y$ queried pairs from $u$ to $S$ are independent $\s{Bern}(q)$; thus $D(u)=Y+\widetilde X$ with $\widetilde X\sim\s{Binomial}(m-Y,q)$. On the event $\{Y\ge dm/(2n)\}$,
\begin{align}
\bE[D(u)\vert S]=qm+(1-q)Y\ge qm+\frac{(1-q)dm}{2n}=T_U+t',
\quad
t'\triangleq\frac{3(1-q)dm}{8n}=3t_U\le\frac{3m}{8},
\end{align}
so that $D(u)<T_U$ implies $\widetilde X-\bE[\widetilde X]<-t'$. Since $q(1-q)(m-Y)+t'/3\le m$, Lemma~\ref{lem:Bernstein} yields
\begin{align}
\pr\p{D(u)<T_U\,\vert\,S}
\le\exp\p{-\frac{t'^2}{2m}}
=\exp\p{-\frac{9(1-q)^2}{128}\cdot\frac{d^2m}{n^2}}
\le n^{-3}
\qquad\text{on }\{Y\ge dm/(2n)\},
\end{align}
again by \eqref{eq:degcut_basic_facts} and $C_1\ge 640/(1-q)^2$. Combining with \eqref{eq:Y_good_degcut} proves \eqref{eq:hub_in_U}.

\smallskip\noindent\emph{Step 2: a hub in $S$ is detected when $U=[n]\setminus S$.} Assume $M=n-m$ and fix $u\in K_{\ge d}$. We claim that
\begin{align}
\pr_{\calH_1}\p{D'(u)<T_S\,\vert\,u\in S}\le 2n^{-3}.
\label{eq:hub_in_S}
\end{align}
By (b), conditionally on $\{u\in S\}$ the set $S\setminus\{u\}$ is a uniform $(m-1)$-subset of $[n]\setminus\{u\}$, so $Y'\triangleq|N_u\cap S|\sim\s{Hypergeometric}(n-1,|N_u|,m-1)$ with $\bE[Y']=|N_u|(m-1)/(n-1)\le |N_u|/2$, because $(m-1)/(n-1)\le m/n\le 1/2$. By Lemma~\ref{lem:HoeffdingHyp} (Hoeffding's inequality for sampling without replacement) with $t=|N_u|/4$, and then \eqref{eq:degcut_basic_facts},
\begin{align}
\pr\p{Y'\ge \tfrac34|N_u|}
&\le\pr\p{Y'-\bE[Y']\ge \tfrac14|N_u|}
\le\exp\p{-\frac{2(|N_u|/4)^2}{m-1}}
\le\exp\p{-\frac{d^2}{8m}}
\nonumber\\
&\le\exp\p{-\frac{d^4}{16C_1n^2\log n}}
\le\exp\p{-\frac{C_1}{4}\log n}
\le n^{-3}.
\label{eq:Yprime_good_degcut}
\end{align}
On the event $\{Y'\le\frac34|N_u|\}$, the number of planted neighbors of $u$ inside $U=[n]\setminus S$ equals $|N_u|-Y'\ge|N_u|/4\ge d/4$. Conditionally on $S$, $D'(u)=(|N_u|-Y')+\widetilde X'$ with $\widetilde X'\sim\s{Binomial}(M-(|N_u|-Y'),q)$, so
\begin{align}
\bE[D'(u)\vert S]\ge qM+\frac{(1-q)d}{4}=T_S+t_S,
\qquad t_S=\frac{(1-q)d}{8}\le\frac{M}{4},
\end{align}
and $D'(u)<T_S$ implies $\widetilde X'-\bE[\widetilde X']<-t_S$. Since $q(1-q)M+t_S/3\le M$, Lemma~\ref{lem:Bernstein} gives, on $\{Y'\le\frac34|N_u|\}$,
\begin{align}
\pr\p{D'(u)<T_S\,\vert\,S}
\le\exp\p{-\frac{t_S^2}{2M}}
\le\exp\p{-\frac{(1-q)^2d^2}{128\,n}}
\le n^{-3},
\end{align}
by \eqref{eq:degcut_basic_facts}. Combining with \eqref{eq:Yprime_good_degcut} proves \eqref{eq:hub_in_S}.

\smallskip\noindent\emph{Step 3: conclusion.}
Suppose first that $M<n-m$, so that $M=\lceil C_2(n/m_d)\log n\rceil$. Let $W\triangleq|U\cap K_{\ge d}|$. By (a), $W\sim\s{Hypergeometric}(n,m_d,M)$ with $\bE[W]=Mm_d/n\ge C_2\log n$, and Lemma~\ref{lem:ChernoffChv} with $\delta=1/2$ gives $\pr(W=0)\le\pr(W\le\frac12\bE[W])\le\exp(-\bE[W]/8)\le n^{-C_2/8}\le n^{-2}$. On $\{W\ge1\}$ the test outputs $1$ unless some $u\in U\cap K_{\ge d}$ has $D(u)<T_U$. Hence, by \eqref{eq:hub_in_U},
\begin{align}
\pr_{\calH_1}(\calA_{\s{cut}}=0)
\le\pr(W=0)+\sum_{u\in K_{\ge d}}\pr\p{u\in U,\ D(u)<T_U}
\le n^{-2}+m_d\cdot 2n^{-3}
\le 3n^{-2}.
\end{align}
Suppose next that $M=n-m$, so that $U=[n]\setminus S$. Fix any $u^\star\in K_{\ge d}$ (it exists since $d\le d_{\s{max}}(\Gamma_n)$). Then either $u^\star\in U$ or $u^\star\in S$, and by \eqref{eq:hub_in_U}--\eqref{eq:hub_in_S},
\begin{align}
\pr_{\calH_1}(\calA_{\s{cut}}=0)
\le\pr\p{u^\star\in U,\ D(u^\star)<T_U}+\pr\p{u^\star\in S,\ D'(u^\star)<T_S}
\le 4n^{-3}.
\end{align}
In both cases,
\begin{align}
\pr_{\calH_1}(\calA_{\s{cut}}=0)\le 3n^{-2}.
\label{eqn:cutTypeII}
\end{align}
Combining \eqref{eqn:cutTypeI} and \eqref{eqn:cutTypeII} yields $\s{R}_n(\calA_{\s{cut}};\mathbb{Q}_n)\le 2n^{-4}+3n^{-2}\le 4n^{-2}$ for all sufficiently large $n$.

\paragraph{Query complexity.}
The algorithm makes exactly $\s{Q}_{\s{cut}}=Mm$ queries, and $m\le 2C_1(n^2/d^2)\log n$ by \eqref{eq:degcut_basic_facts}. If $m_d\ge C_2\log n$, then $C_2(n/m_d)\log n\ge 1$ (as $m_d\le n$), so $M\le\lceil C_2(n/m_d)\log n\rceil\le 2C_2(n/m_d)\log n$ and
\begin{align}
\s{Q}_{\s{cut}}\le 4C_1C_2\frac{n^3\log^2 n}{d^2m_d}=4C_1C_2\frac{n^3\log^2 n}{d^2\max\{m_d,\log n\}}.
\end{align}
If $m_d<C_2\log n$, then $M\le n-m\le n$ and $\max\{m_d,\log n\}\le C_2\log n$, so
\begin{align}
\s{Q}_{\s{cut}}\le nm\le 2C_1\frac{n^3\log n}{d^2}\le 2C_1C_2\frac{n^3\log^2 n}{d^2\max\{m_d,\log n\}}.
\end{align}
This proves \eqref{eq:degcut_budget_explicit} with $\s{C}=4C_1C_2$.

\subsection{Count test}\label{subsec:count_proof}

We write the proof for a general edge probability $p\in(q,1]$ on the planted edges; in our model $p=1$ and $\chi^2(p\|q)=(1-q)/q$. Let us upper bound the $\s{Type}$-$\s{I}$ and $\s{Type}$-$\s{II}$ error probabilities, starting with the former. Let $\mathbb{Q}_n\subseteq \binom{[n]}{2}$ be sampled uniformly at random among all subsets of cardinality $\s{Q}_n$, independently of the graph, and define
\begin{align}
\s{X}_{\s{count}}
\triangleq
\sum_{e\in \mathbb{Q}_n}\mathds{1}\{e\in e(\s{G})\}.
\end{align}
For a sequence $\lambda_n\to\infty$, consider the test
\begin{align}
\calA_{\s{count}}
\triangleq
\mathds{1}\left\{
\s{X}_{\s{count}}
\ge
q\s{Q}_n+\lambda_n\sqrt{q(1-q)\s{Q}_n}
\right\}.
\end{align}
Under $\calH_0$, conditional on $\mathbb{Q}_n$, the queried edge indicators are i.i.d. $\s{Bern}(q)$. Hence $\s{X}_{\s{count}}\mid \mathbb{Q}_n\sim
\s{Binomial}(\s{Q}_n,q)$. Therefore, by Bernstein's inequality (Lemma~\ref{lem:Bernstein}), and since $\lambda_n\le\sqrt{q(1-q)\s{Q}_n}$ for all large $n$ under the assumption $\lambda_n^2=o(\s{Q}_n(|e(\Gamma_n)|/\binom n2)^2)$ (note that $|e(\Gamma_n)|\le\binom n2$),
\begin{align}
\pr_{\calH_0}\p{\calA_{\s{count}}=1\vert \mathbb{Q}_n}
&=
\pr_{\calH_0}\pp{\left.
\s{X}_{\s{count}}
\ge
q\s{Q}_n+\lambda_n\sqrt{q(1-q)\s{Q}_n}
\right| \mathbb{Q}_n}
\le
\exp\p{-c_q\lambda_n^2}.
\end{align}
Taking expectation over $\mathbb{Q}_n$, we obtain
\begin{align}
\pr_{\calH_0}\p{\calA_{\s{count}}=1}
\le
\exp\p{-c_q\lambda_n^2}
=o(1).
\end{align}

We now bound the $\s{Type}$-$\s{II}$ error. Let $\Gamma_n^\star$ denote the planted copy of $\Gamma_n$, and write
\begin{align}
Z_n
\triangleq
|e(\Gamma_n^\star)\cap \mathbb{Q}_n|
\end{align}
for the number of planted edges queried by the algorithm. Conditional on $\Gamma_n^\star$ and $\mathbb{Q}_n$, the variable $\s{X}_{\s{count}}$ is a sum of independent Bernoulli random variables, with parameter $p$ on the $Z_n$ planted queried edges and parameter $q$ on the remaining $\s{Q}_n-Z_n$ queried edges. Hence
\begin{align}
\bE_{\calH_1}\pp{\s{X}_{\s{count}}\vert \Gamma_n^\star,\mathbb{Q}_n}
&=
q\s{Q}_n+(p-q)Z_n,\quad
\s{Var}_{\calH_1}\pp{\s{X}_{\s{count}}\vert \Gamma_n^\star,\mathbb{Q}_n}
\le
\s{Q}_n.
\end{align}
Let $T_n
\triangleq
q\s{Q}_n+\lambda_n\sqrt{q(1-q)\s{Q}_n}$. On the event
\begin{align}
\mathcal G_n
\triangleq
\left\{
(p-q)Z_n
\ge
2\lambda_n\sqrt{q(1-q)\s{Q}_n}
\right\},
\end{align}
we have
\begin{align}
\bE_{\calH_1}\pp{\s{X}_{\s{count}}\vert \Gamma_n^\star,\mathbb{Q}_n}
-T_n
\ge
\frac12(p-q)Z_n.
\end{align}
Therefore, conditional on $\Gamma_n^\star$ and $\mathbb{Q}_n$, Bernstein's inequality gives
\begin{align}
\pr_{\calH_1}\p{
\s{X}_{\s{count}}<T_n
\vert
\Gamma_n^\star,\mathbb{Q}_n}
&\le
\exp\p{
-c\frac{(p-q)^2Z_n^2}{\s{Q}_n}
}
\; \s{on}\; \mathcal G_n,
\end{align}
where $c>0$ is an absolute constant. Consequently,
\begin{align}
\pr_{\calH_1}\p{\calA_{\s{count}}=0}
\le
\bE\pp{
\exp\p{
-c\frac{(p-q)^2Z_n^2}{\s{Q}_n}
}
\mathds{1}_{\mathcal G_n}
}
+
\pr(\mathcal G_n^c).
\label{eq:count_typeII_split}
\end{align}

It remains to show that the right-hand side is $o(1)$. Conditional on $\Gamma_n^\star$, since $\mathbb{Q}_n$ is a uniform $\s{Q}_n$-subset of $\binom{[n]}{2}$, the random variable $Z_n$ has a hypergeometric distribution, i.e.,
\begin{align}
Z_n\vert \Gamma_n^\star
\sim
\s{Hypergeometric}\p{
\binom n2,|e(\Gamma_n)|,\s{Q}_n
}.
\end{align}
In particular,
\begin{align}
\bE[Z_n\vert \Gamma_n^\star]
=\s{Q}_n\frac{|e(\Gamma_n)|}{\binom n2}
\triangleq
\alpha_n.
\end{align}
By the Chernoff bound for hypergeometric random variables,
\begin{align}
\pr\p{\left.
Z_n\le \frac{\alpha_n}{2}
\right| \Gamma_n^\star
}
\le
\exp\p{-c\alpha_n}.
\end{align}
Thus
\begin{align}
\pr\p{
Z_n\le \frac{\alpha_n}{2}
}
\le
\exp\p{-c\alpha_n}.
\label{eq:count_Z_concentration}
\end{align}
Assume now that
\begin{align}
\frac{(p-q)^2}{q(1-q)}
\cdot
\s{Q}_n
\p{\frac{|e(\Gamma_n)|}{\binom n2}}^2
\to \infty.
\label{eq:edgecount_snr_clean}
\end{align}
Equivalently,
\begin{align}
(p-q)\alpha_n
\gg
\sqrt{q(1-q)\s{Q}_n}.
\end{align}
 By the assumption $\lambda_n^2=o(\s{Q}_n(|e(\Gamma_n)|/\binom n2)^2)$ of Theorem~\ref{thm:edgecount_query},
\begin{align}
(p-q)\alpha_n
\gg
\lambda_n\sqrt{q(1-q)\s{Q}_n}.
\label{eq:lambda_slow_count}
\end{align}
Then, on the event $Z_n\ge \alpha_n/2$, condition \eqref{eq:lambda_slow_count} implies $\mathcal G_n$ for all sufficiently large $n$. Hence, by \eqref{eq:count_Z_concentration},
\begin{align}
\pr(\mathcal G_n^c)
\le
\pr\p{Z_n<\frac{\alpha_n}{2}}
=o(1).
\end{align}
Moreover, on the same event,
\begin{align}
\frac{(p-q)^2Z_n^2}{\s{Q}_n}
\ge
\frac{(p-q)^2\alpha_n^2}{4\s{Q}_n}
=
\frac14
(p-q)^2\s{Q}_n
\p{\frac{|e(\Gamma_n)|}{\binom n2}}^2
\to \infty
\end{align}
by \eqref{eq:edgecount_snr_clean}. Therefore the expectation term in \eqref{eq:count_typeII_split} also tends to zero. We conclude that
\begin{align}
\pr_{\calH_1}\p{\calA_{\s{count}}=0}=o(1).
\end{align}
Combining the $\s{Type}$-$\s{I}$ and $\s{Type}$-$\s{II}$ bounds yields
\begin{align}
\s{R}_n(\calA_{\s{count}};\mathbb{Q}_n)\to 0.
\end{align}

\subsection{Biclique scan: proof of Theorem~\ref{thm:biclique_scan}}\label{subsec:proof_bscan}

Recall $L=\lceil 3/\log(1/q)\rceil+3$, $C_0=64L$, $s_n=\lceil L\log n\rceil$, $M_n=\lceil 8ns_n/a_n\rceil$ and $m_n=\lceil 8ns_n/b_n\rceil$, with $C_0\log n\le a_n\le b_n$. For all sufficiently large $n$ we have $s_n\le 2L\log n$, hence $8ns_n/a_n\le 16Ln\log n/(64L\log n)=n/4$ and $M_n\le n/4+1\le n/2$; since $b_n\ge a_n$ we also have $m_n\le M_n$. Thus $S$ and $U$ in Algorithm~\ref{alg:biclique_scan} can be sampled as required, and $\s{Q}_{\s{bscan}}=M_nm_n\le n^2/4$. Moreover, since $8ns_n/a_n\ge 8s_n\ge 1$ and $8ns_n/b_n\ge 1$ (recall $a_n\le b_n\le n$), we have $M_n\le 16ns_n/a_n$ and $m_n\le 16ns_n/b_n$, so that
\begin{align}
\s{Q}_{\s{bscan}}=M_nm_n\le\frac{256\,n^2s_n^2}{a_nb_n}\le\frac{1024L^2\,n^2\log^2 n}{a_nb_n},
\end{align}
which proves \eqref{eq:bscan_budget}.

\paragraph{Type-I error.}
Under $\calH_0$ the $M_nm_n$ queried pairs are independent $\s{Bern}(q)$. The test outputs $1$ only if, for some $A\subseteq U$ and $B\subseteq S$ with $|A|=|B|=s_n$, all $s_n^2$ pairs in $A\times B$ are present. By a union bound over the $\binom{M_n}{s_n}\binom{m_n}{s_n}\le n^{2s_n}$ choices of $(A,B)$,
\begin{align}
\pr_{\calH_0}(\calA_{\s{bscan}}=1)
\le n^{2s_n}q^{s_n^2}
=\exp\p{s_n\pp{2\log n-s_n\log\tfrac1q}}
\le\exp\p{-s_n\log n}
\le n^{-3},
\end{align}
where we used $s_n\log(1/q)\ge L\log n\log(1/q)\ge 3\log n$ and $s_n\ge L\ge 3$.

\paragraph{Type-II error.}
Condition on the planted embedding $\Phi$ and let $A^\star,B^\star\subseteq[n]$ be the images of the two sides of a fixed copy of $K_{a_n,b_n}$ inside $\Gamma_n$; thus $|A^\star|=a_n$, $|B^\star|=b_n$, $A^\star\cap B^\star=\emptyset$, and every pair in $A^\star\times B^\star$ is a planted edge. As in the proof of Theorem~\ref{thm:degree_cut_general_unsat_main}, $U$ is marginally a uniform $M_n$-subset of $[n]$ and $S$ is a uniform $m_n$-subset of $[n]$. Hence $X\triangleq|U\cap A^\star|\sim\s{Hypergeometric}(n,a_n,M_n)$ with $\bE[X]=M_na_n/n\ge 8s_n$, and $Y\triangleq|S\cap B^\star|\sim\s{Hypergeometric}(n,b_n,m_n)$ with $\bE[Y]=m_nb_n/n\ge 8s_n$. By Lemma~\ref{lem:ChernoffChv} with $\delta=1/2$,
\begin{align}
\pr(X<s_n)\le\pr\p{X\le\tfrac12\bE[X]}\le\exp\p{-\tfrac18\bE[X]}\le e^{-s_n}\le n^{-L}\le n^{-3},
\end{align}
and likewise $\pr(Y<s_n)\le n^{-3}$. On the event $\{X\ge s_n,\ Y\ge s_n\}$, pick $A\subseteq U\cap A^\star$ and $B\subseteq S\cap B^\star$ with $|A|=|B|=s_n$; all pairs in $A\times B$ are queried and planted, hence present, and the test outputs $1$. Therefore $\pr_{\calH_1}(\calA_{\s{bscan}}=0)\le 2n^{-3}$, uniformly in $\Phi$. Together with the Type-I bound, $\s{R}_n(\calA_{\s{bscan}})\le 3n^{-3}$.

\section{Corollaries}

In this section we derive several corollaries from the general lower bounds derived in the previous subsections, by specializing to specific families of subgraphs, starting with subgraphs with a bounded vertex cover number.

\subsection{Proof of Corollary~\ref{cor:weighted_cover_reduction_simplified_main}}
The proof of Corollary~\ref{cor:weighted_cover_reduction_simplified_main} follows from the observation below.
\begin{lemma}
\label{lem:zeta_basic_bound}
For every integer $r\in\{1,\ldots,\tau_n\}$,
\begin{align}
\zeta_{r,n}(\mathbb{Q})
\le
\frac{2\s{Q}_n}{n}.
\label{eq:zeta_basic_bound}
\end{align}
\end{lemma}

\begin{proof}
Fix $r\in\{1,\ldots,\tau_n\}$ and $i\in[n]$. Since $t_i^{(n)}\le n-1$, we have
\begin{align}
(t_i^{(n)})_r
\le
t_i^{(n)}(n-1)_{r-1}.
\end{align}
Therefore
\begin{align}
\zeta_{r,n}
=
\sum_{i=1}^n
\frac{(t_i^{(n)})_r}{(n)_r}
\le
\sum_{i=1}^n
\frac{t_i^{(n)}(n-1)_{r-1}}{(n)_r}
=
\frac{1}{n}
\sum_{i=1}^n t_i^{(n)}.
\end{align}
Since $\sum_{i=1}^n t_i^{(n)}=2\s{Q}_n$, the claim follows.
\end{proof}

By Lemma~\ref{lem:zeta_basic_bound},
\begin{align}
\sum_{r=1}^{\tau_n}(q^{-1}-1)^r\Theta_{r,n}\zeta_{r,n}
\le
\frac{2\s{Q}_n}{n}
\sum_{r=1}^{\tau_n}(q^{-1}-1)^r\Theta_{r,n}.
\end{align}
Hence \eqref{eq:assump_weighted_simplified} implies condition \eqref{eq:assump_remainder_negligible} in Theorem~\ref{thm:cover_reduction_query_lb}. The conclusion now follows directly from Theorem~\ref{thm:cover_reduction_query_lb}.

\subsection{Proof of Corollary~\ref{prop:pendant_layer_tight_main}}

To prove Corollary~\ref{prop:pendant_layer_tight_main}, it is convenient to first establish the following impossibility result for star forests.
\begin{corollary}[Star forests]
\label{cor:star_forest_weighted}
Fix $q\in(0,1)$ and let $\Gamma_n=\bigsqcup_{j=1}^{r_n} K_{1,d_{j,n}}$, be a disjoint union of stars, where the degrees $d_{j,n}\ge 1$ are arbitrary. Let $S_{2,n}\triangleq \sum_{j=1}^{r_n} d_{j,n}^2$.
If
\begin{align}
\frac{\s{Q}_n S_{2,n}}{n(n-r_n)^2}\to 0,
\label{eq:starforest_condition_exact}
\end{align}
then weak detection is impossible. In particular, if $r_n=o(n)$ and
\begin{align}
\s{Q}_n = o\p{\frac{n^3}{S_{2,n}}},
\label{eq:starforest_condition_asymp}
\end{align}
then weak detection is impossible.
\end{corollary}

\begin{proof}
Choose $U_n$ to consist of the star centers, and let $W_n$ be the set of all leaves. Then $U_n$ is a vertex cover, $\tau_n=r_n$, and the induced graph $H_n=\Gamma_n[U_n]$ is edgeless. Therefore
\begin{align}
\pr_{H_n,\mathbb{Q}_n}=\pr_{\calH_0,\mathbb{Q}_n},
\end{align}
so the core-undetectability condition is automatic. Next, let us compute the quantities $\Theta_{r,n}$. Since each leaf has exactly one neighbor in the cover, for every nonempty $R\subseteq U_n$ with $|R|\ge 2$ we have
\begin{align}
\mathcal W_n(R)=\emptyset.
\end{align}
On the other hand, if $u_j\in U_n$ is the center of the $j$th star, then
\begin{align}
|\mathcal W_n(\{u_j\})|=d_{j,n}.
\end{align}
Therefore
\begin{align}
\Theta_{1,n}
=
\sum_{j=1}^{r_n} d_{j,n}^2
=
S_{2,n},
\qquad
\Theta_{r,n}=0,\quad r\ge 2.
\end{align}
By Theorem~\ref{thm:cover_reduction_query_lb}, it remains only to verify that
\begin{align}
\frac{e^2}{(n-r_n)^2}(q^{-1}-1)\Theta_{1,n}\zeta_{1,n}\to 0.
\end{align}
Now $\zeta_{1,n}=2\s{Q}_n/n$, so
\begin{align}
\frac{e^2}{(n-r_n)^2}(q^{-1}-1)\Theta_{1,n}\zeta_{1,n}
=
\frac{2e^2(q^{-1}-1)\s{Q}_nS_{2,n}}{n(n-r_n)^2},
\end{align}
which tends to $0$ by \eqref{eq:starforest_condition_exact}. This proves the first claim. The second claim follows immediately from \eqref{eq:starforest_condition_exact} when $r_n=o(n)$.
\end{proof}
We are now ready to establish the lower bound in Proposition~\ref{prop:pendant_layer_tight_main} for pendant-leaf graphs; we restate it here for convenience.
\begin{corollary}[Pendant-leaf graphs]
\label{cor:pendant_layer_weighted}
Fix $q\in(0,1)$ and let $(\Gamma_n)_{n\ge 1}$ be a sequence of graphs with a vertex cover $U_n$ such that:
\begin{enumerate}
\item $H_n=\Gamma_n[U_n]$ is edgeless;
\item every vertex $w\in W_n\triangleq V_n\setminus U_n$ has exactly one neighbor in $U_n$.
\end{enumerate}
For each $u\in U_n$, let
\begin{align}
d_{W_n}(u)\triangleq |\{w\in W_n:\{u,w\}\in e(\Gamma_n)\}|.
\end{align}
If
\begin{align}
\frac{\s{Q}_n}{n(n-\tau_n)^2}
\sum_{u\in U_n} d_{W_n}^2(u)
\to 0,
\label{eq:pendant_layer_condition}
\end{align}
then weak detection is impossible.
\end{corollary}

\begin{proof}
Since $H_n$ is edgeless, the core-undetectability condition is automatic. Also, because every $w\in W_n$ has exactly one neighbor in $U_n$, we have
\begin{align}
\Theta_{1,n}
=
\sum_{u\in U_n} d_{W_n}^2(u),
\quad
\Theta_{r,n}=0,\quad r\ge 2.
\end{align}
Thus the conclusion follows exactly as in the proof of Corollary~\ref{cor:star_forest_weighted}.
\end{proof}
The rest of Corollary~\ref{prop:pendant_layer_tight_main} follows from the following result.

\begin{prop}
\label{prop:tightness_star_forest_pendant}
Fix $q\in(0,1)$ and let $(\Gamma_n)_{n\ge1}$ be a sequence of pendant-leaf graphs as in Corollary~\ref{cor:pendant_layer_weighted}, with cover $U_n$, leaf set $W_n$, and $d_{W_n}(u)$, $\Delta_n\triangleq\max_{u\in U_n}d_{W_n}(u)$, $S_{2,n}\triangleq\sum_{u\in U_n}d_{W_n}(u)^2$ as there. (Star forests $\Gamma_n=\bigsqcup_{j=1}^{r_n}K_{1,d_{j,n}}$ are the special case in which $U_n$ is the set of star centers, so that $S_{2,n}=\sum_j d_{j,n}^2$ and $\Delta_n=\max_jd_{j,n}$.) Let $m_d(\Gamma_n)$ and $\kappa(\Gamma_n)=\max_{d\ge1}m_d(\Gamma_n)d^2$ be as in Theorem~\ref{thm:degree_cut_general_unsat_main}. Then:
\begin{enumerate}
\item[(i)] for every $d\ge 2$, $m_d(\Gamma_n)=|\{u\in U_n:d_{W_n}(u)\ge d\}|$, and consequently $\max_{d\ge2}m_d(\Gamma_n)d^2\le S_{2,n}$;
\item[(ii)] $S_{2,n}\le 4\p{1+\lfloor\log_2\Delta_n\rfloor}\kappa(\Gamma_n)$;
\item[(iii)] if some feasible degree level $d^\star$ (in the sense of Theorem~\ref{thm:degree_cut_general_unsat_main}) attains the maximum defining $\kappa(\Gamma_n)$, then $d^\star\ge2$, $\kappa(\Gamma_n)=\max_{d\ge 2}m_d(\Gamma_n)d^2\le S_{2,n}$, and strong detection is possible whenever $\s{Q}_n\ge\s{C}\frac{n^3}{\kappa(\Gamma_n)}\log^2 n$, whereas weak detection is impossible whenever $\s{Q}_n=o(n^3/S_{2,n})$. Hence the two bounds differ by at most a factor of order $\log\Delta_n\cdot\log^2 n$, and only by $\log^2n$ if the $d_{W_n}(u)$ are all of the same order.
\end{enumerate}
\end{prop}

\begin{proof}
(i) Every $w\in W_n$ has exactly one neighbor (in $U_n$), so $\deg_{\Gamma_n}(w)=1$, and every $u\in U_n$ has $\deg_{\Gamma_n}(u)=d_{W_n}(u)$ because $H_n=\Gamma_n[U_n]$ is edgeless. Hence for $d\ge2$ the vertices of degree at least $d$ are exactly the $u\in U_n$ with $d_{W_n}(u)\ge d$, and
\begin{align}
m_d(\Gamma_n)d^2\le\sum_{u\in U_n:\ d_{W_n}(u)\ge d}d_{W_n}(u)^2\le S_{2,n}.
\end{align}
(ii) Partition $U_n$ into dyadic degree classes: for $\ell\in\{0,1,\ldots,\lfloor\log_2\Delta_n\rfloor\}$ let $I_\ell\triangleq\{u\in U_n:2^\ell\le d_{W_n}(u)<2^{\ell+1}\}$. Then
\begin{align}
S_{2,n}=\sum_{\ell}\sum_{u\in I_\ell}d_{W_n}(u)^2\le\sum_\ell|I_\ell|\,2^{2\ell+2}.
\end{align}
For $\ell\ge1$, (i) gives $|I_\ell|\le m_{2^\ell}(\Gamma_n)$; for $\ell=0$, $|I_0|\le|v(\Gamma_n)|=m_1(\Gamma_n)$ (every vertex has degree at least one). Therefore $S_{2,n}\le 4\sum_\ell m_{2^\ell}(\Gamma_n)2^{2\ell}\le 4(1+\lfloor\log_2\Delta_n\rfloor)\kappa(\Gamma_n)$.

(iii) A feasible level satisfies $(d^\star)^2\ge 2C_1n\log n\ge 4$, so $d^\star\ge2$ and $\kappa(\Gamma_n)=m_{d^\star}(\Gamma_n)(d^\star)^2\le\max_{d\ge2}m_d(\Gamma_n)d^2\le S_{2,n}$ by (i). Theorem~\ref{thm:degree_cut_general_unsat_main} at the level $d^\star$ gives strong detection with $\s{Q}_{\s{cut}}\le\s{C}n^3\log^2n/((d^\star)^2\max\{m_{d^\star},\log n\})\le\s{C}n^3\log^2 n/\kappa(\Gamma_n)$ queries, and Corollary~\ref{cor:pendant_layer_weighted} (whose condition \eqref{eq:pendant_layer_condition} reads $\s{Q}_n=o(n^3/S_{2,n})$ because $\tau_n=|U_n|\le|W_n|$ implies $n-\tau_n\ge n/2$) gives the lower bound. The last claim follows from (ii) and from $S_{2,n}\asymp|U_n|\Delta_n^2\asymp m_{c\Delta_n}(\Gamma_n)(c\Delta_n)^2\le\kappa(\Gamma_n)$ when $c\Delta_n\le d_{W_n}(u)\le\Delta_n$ for all $u\in U_n$.
\end{proof}

\subsection{Proof of Corollary~\ref{cor:path_weighted_main}}

Write the path vertices as $1,2,\ldots,k_n$ in their natural order, and choose
\begin{align}
U_n\triangleq \{2,4,6,\ldots\}\cap [k_n].
\end{align}
Then $U_n$ is a vertex cover, $W_n=V_n\setminus U_n$ is independent, and the induced graph $H_n=\Gamma_n[U_n]$ is edgeless. Therefore
\begin{align}
\pr_{H_n,\mathbb{Q}_n}=\pr_{\calH_0,\mathbb{Q}_n},
\end{align}
so the core-undetectability condition is automatic.

Next, we bound $\Theta_{r,n}$. Every vertex $w\in W_n$ has at most two neighbors in $U_n$, hence
\begin{align}
\Theta_{r,n}=0,\qquad r\ge 3.
\end{align}
Also, for each $u\in U_n$, the set $\mathcal W_n(\{u\})$ contains at most the two path-neighbors of $u$, so
\begin{align}
|\mathcal W_n(\{u\})|\le 2.
\end{align}
Therefore
\begin{align}
\Theta_{1,n}
=
\sum_{u\in U_n} |\mathcal W_n(\{u\})|^2
\le
4|U_n|
\le
2k_n.
\label{eq:path_theta1}
\end{align}
Similarly, if $R\subseteq U_n$ has cardinality $2$, then $\mathcal W_n(R)$ contains at most one vertex, namely the unique odd-indexed vertex whose two neighbors are the two vertices of $R$, if such a vertex exists. Hence
\begin{align}
|\mathcal W_n(R)|\le 1
\qquad\forall R\subseteq U_n,\ |R|=2,
\end{align}
and the number of such relevant pairs is at most $|W_n|$. Thus
\begin{align}
\Theta_{2,n}\le |W_n|\le k_n.
\label{eq:path_theta2}
\end{align}
Combining \eqref{eq:path_theta1} and \eqref{eq:path_theta2}, and using Corollary~\ref{cor:weighted_cover_reduction_simplified_main}, it is enough to check that
\begin{align}
\frac{\s{Q}_n}{n(n-\tau_n)^2}
\Bigl((q^{-1}-1)\Theta_{1,n}+(q^{-1}-1)^2\Theta_{2,n}\Bigr)\to 0.
\end{align}
Since $\tau_n\le k_n/2\le n/2$ for all sufficiently large $n$, we have $n-\tau_n\asymp n$. Using \eqref{eq:path_theta1} and \eqref{eq:path_theta2}, the left-hand side is
\begin{align}
O_q\p{\frac{\s{Q}_n k_n}{n^3}},
\end{align}
which tends to $0$ whenever $\frac{\s{Q}_n k_n}{n^3}\to 0$. This proves the first claim. Finally, if $\s{Q}_n=\binom n2$, then $\s{Q}_n\asymp n^2$, and $\frac{\s{Q}_n k_n}{n^3}\to 0$ reduces to $k_n/n\to 0$, namely $k_n=o(n)$.

\subsection{Proof of Corollary~\ref{cor:triangles_full_observation_main}}

We apply Theorem~\ref{thm:cover_reduction_query_lb} in the simplified form of Corollary~\ref{cor:weighted_cover_reduction_simplified_main}, for an arbitrary query budget $\s{Q}_n$ with $\s{Q}_nk_n=o(n^3)$. For each triangle, choose two of its vertices and put them into the cover $U_n$; let the remaining vertex belong to $W_n$. Then $U_n$ is a vertex cover, $\tau_n=2m_n=2k_n/3\le 2n/3$, so that $n-\tau_n\ge n/3$, and the induced graph $H_n=\Gamma_n[U_n]$ contains exactly one edge from each triangle and no edges between different triangles. Therefore $H_n$ is a matching with exactly $m_n$ edges.

\emph{Core condition.} The auxiliary model $\pr_{H_n,\mathbb{Q}_n}$ is the planted-subgraph model with planted graph $H_n$. A matching is a star forest with $m_n$ stars of degree one, i.e., a pendant-leaf graph with cover $U_n'$ consisting of one endpoint of each matching edge, $d_{W_n'}(u)=1$ for all $u\in U_n'$, $S_{2,n}'=m_n$, and $\tau_n'=m_n=k_n/3\le n/3$. By Corollary~\ref{cor:pendant_layer_weighted} applied to $H_n$,
\begin{align}
\sup_{|\mathbb{Q}_n|\le\s{Q}_n}d_{\s{TV}}\p{\pr_{H_n,\mathbb{Q}_n},\pr_{\calH_0,\mathbb{Q}_n}}\to0
\qquad\text{whenever}\qquad
\frac{\s{Q}_nm_n}{n(n-m_n)^2}\to0,
\end{align}
and the latter holds since $\s{Q}_nm_n/(n(n-m_n)^2)\le (9/4)\s{Q}_nk_n/(3n^3)\to0$. Thus condition \eqref{eq:assump_core_undetectable} holds. (This is the recursive use of Theorem~\ref{thm:cover_reduction_query_lb} described in Remark~\ref{rem:recursive}.)

\emph{Remainder condition.} Each vertex $w\in W_n$ belongs to exactly one triangle, and its neighborhood inside the cover consists of the two chosen cover vertices of that triangle. Hence, for every singleton $R=\{u\}\subseteq U_n$ there is exactly one vertex in $W_n$ adjacent to $u$; for every pair $R\subseteq U_n$ formed by the two cover vertices of one triangle there is exactly one vertex in $W_n$ whose neighborhood contains $R$; and for every other nonempty $R\subseteq U_n$ the set $\mathcal W_n(R)$ is empty. Therefore
\begin{align}
\Theta_{1,n}=2m_n=\frac{2k_n}{3},
\qquad
\Theta_{2,n}=m_n=\frac{k_n}{3},
\qquad
\Theta_{r,n}=0,\quad r\ge 3.
\label{eq:triangle_Theta_values}
\end{align}
By Corollary~\ref{cor:weighted_cover_reduction_simplified_main}, it suffices that
\begin{align}
\frac{\s{Q}_n}{n(n-\tau_n)^2}
\Bigl((q^{-1}-1)\Theta_{1,n}+(q^{-1}-1)^2\Theta_{2,n}\Bigr)
&\le
\frac{9\s{Q}_n}{n^3}\cdot\frac{k_n}{3}\Bigl(2(q^{-1}-1)+(q^{-1}-1)^2\Bigr)
\nonumber\\
&=O_q\p{\frac{\s{Q}_nk_n}{n^3}}
\to 0,
\end{align}
which is our assumption. The conclusion follows from Theorem~\ref{thm:cover_reduction_query_lb}. Finally, under full observation $\s{Q}_n=\binom n2\asymp n^2$ and the condition $\s{Q}_nk_n=o(n^3)$ reduces to $k_n=o(n)$.

\subsection{Proof of Corollary~\ref{cor:trees_full_observation_main}}

We apply Theorem~\ref{thm:cover_reduction_query_lb} in the simplified form of Corollary~\ref{cor:weighted_cover_reduction_simplified_main}, for an arbitrary query budget $\s{Q}_n$ with $\s{Q}_nk_n=o(n^3)$. Every tree is bipartite. Let $v(T_n)=A_n\sqcup B_n$ be a bipartition of the tree, and choose $U_n$ to be the smaller side of the bipartition. Then every edge of $T_n$ joins a vertex of $U_n$ to a vertex of $W_n\triangleq v(T_n)\setminus U_n$, so $U_n$ is a vertex cover and $H_n=T_n[U_n]$ is edgeless. Consequently $\pr_{H_n,\mathbb{Q}_n}=\pr_{\calH_0,\mathbb{Q}_n}$, so the core-undetectability condition is automatic. Moreover $\tau_n=|U_n|\le k_n/2\le n/2$, so $n-\tau_n\ge n/2$. It remains to verify the weighted remainder condition.

Since $d_{\s{max}}(T_n)\le D$, for every $w\in W_n$ we have $|N_n(w)|\le D$. Therefore, for every $r>D$ and every $R\subseteq U_n$ with $|R|=r$, we have $\mathcal W_n(R)=\emptyset$, i.e.,
\begin{align}
\Theta_{r,n}=0,\quad r>D.
\label{eq:Theta_tree_zero}
\end{align}
For $r\in\{1,\ldots,D\}$ we use the elementary bound $\binom{x}{r}\le 2^{x}\mathds{1}\{x\ge1\}\le 2^{D}x$, valid for all integers $0\le x\le D$. Hence
\begin{align}
\Theta_{r,n}
=
\sum_{w,w'\in W_n}
\binom{|N_n(w)\cap N_n(w')|}{r}
\le
2^{D}\sum_{w,w'\in W_n}|N_n(w)\cap N_n(w')|
=
2^{D}\sum_{u\in U_n} d_{W_n}^2(u),
\label{eq:Theta_tree_first}
\end{align}
where $d_{W_n}(u)\triangleq \abs{\{w\in W_n:\{u,w\}\in e(T_n)\}}$ and the last equality holds because both sides count the triples $(w,w',u)$ with $u\in N_n(w)\cap N_n(w')$. Since $d_{W_n}(u)\le D$ for every $u\in U_n$ and $\sum_{u\in U_n}d_{W_n}(u)=|e(T_n)|=k_n-1$, we get
\begin{align}
\Theta_{r,n}\le 2^{D}D(k_n-1)=O_D(k_n),
\qquad r=1,\ldots,D.
\label{eq:Theta_tree_final}
\end{align}
By Corollary~\ref{cor:weighted_cover_reduction_simplified_main}, \eqref{eq:Theta_tree_zero}, \eqref{eq:Theta_tree_final} and $n-\tau_n\ge n/2$, it suffices that
\begin{align}
\frac{\s{Q}_n}{n(n-\tau_n)^2}
\sum_{r=1}^{D}
(q^{-1}-1)^r\Theta_{r,n}
\le
\frac{4\s{Q}_n}{n^3}\cdot 2^DD(k_n-1)\sum_{r=1}^{D}(q^{-1}-1)^r
=
O_{q,D}\p{\frac{\s{Q}_nk_n}{n^3}}
\to 0,
\end{align}
which is our assumption. The conclusion follows from Theorem~\ref{thm:cover_reduction_query_lb}; under full observation, $\s{Q}_n\asymp n^2$ and the condition reduces to $k_n=o(n)$.

\subsection{Proof of Corollary~\ref{cor:bounded_vertex_cover_weighted_main}}

Fix $n$. Choose once and for all a minimum vertex cover $U_n=\{u_1,\dots,u_{\tau_n}\}\subseteq v(\Gamma_n)$, and let $W_n\triangleq v(\Gamma_n)\setminus U_n$. Since $U_n$ is a vertex cover, the set $W_n$ is independent in $\Gamma_n$, i.e., $e(\Gamma_n)\cap \binom{W_n}{2}=\emptyset$. Hence every edge of $\Gamma_n$ is either inside $U_n$ or between $U_n$ and $W_n$. Define $H_n\triangleq \Gamma_n[U_n]$. Then $H_n$ has at most $\tau_0$ vertices and at most $R_0\triangleq \binom{\tau_0}{2}$ edges. Let $R_n\triangleq |e(H_n)|\le R_0$. Take two independent planted copies of $H_n$, and let $\s{X}^{(H)}_{\mathbb{Q}_n}
$ denote the number of queried edges that belong to both planted copies. We have
\begin{align}
1+\chi^2(\pr_{H_n,\mathbb{Q}_n},\pr_{0,\mathbb{Q}_n})
=
\bE\left[q^{-\s{X}^{(H)}_{\mathbb{Q}_n}}\right].
\label{eq:bounded_cover_chi2_identity}
\end{align}
Since
\begin{align}
0\le \s{X}^{(H)}_{\mathbb{Q}_n}\le R_n\le R_0,
\end{align}
we have
\begin{align}
q^{-\s{X}^{(H)}_{\mathbb{Q}_n}}-1
\le
(q^{-R_0}-1)\mathds{1}\{\s{X}^{(H)}_{\mathbb{Q}_n}\ge 1\}
\le
(q^{-R_0}-1)\s{X}^{(H)}_{\mathbb{Q}_n}.
\end{align}
Taking expectations and using \eqref{eq:bounded_cover_chi2_identity},
\begin{align}
\chi^2(\pr_{H_n,\mathbb{Q}_n},\pr_{0,\mathbb{Q}_n})
\le
(q^{-R_0}-1)\bE\pp{\s{X}^{(H)}_{\mathbb{Q}_n}}.
\label{eq:bounded_cover_chi2_reduce}
\end{align}
We now compute $\bE[\s{X}^{(H)}_{\mathbb{Q}_n}]$.
For each queried edge $e\in\mathbb{Q}_n$, by symmetry of the uniformly random injective embedding of $H_n$ into $[n]$,
\begin{align}
\pr[e\in e(H_n^\star)]=\frac{R_n}{\binom n2},
\end{align}
where $H_n^\star$ denotes one planted copy of $H_n$. Thus, for two independent copies,
\begin{align}
\pr(e\in e(H_n^\star)\cap e(\widetilde H_n^\star))
=
\left(\frac{R_n}{\binom n2}\right)^2.
\end{align}
Therefore
\begin{align}
\bE[\s{X}^{(H)}_{\mathbb{Q}_n}]
&=
\sum_{e\in\mathbb{Q}_n}
\pr(e\in e(H_n^\star)\cap e(\widetilde H_n^\star))
\\
&=
\s{Q}_n\left(\frac{R_n}{\binom n2}\right)^2
\le
\binom n2\left(\frac{R_0}{\binom n2}\right)^2
=
\frac{R_0^2}{\binom n2}.
\label{eq:bounded_cover_EX}
\end{align}
Since $\binom n2\to\infty$, we obtain $\bE[\s{X}^{(H)}_{\mathbb{Q}_n}]\to 0$. Combining this with \eqref{eq:bounded_cover_chi2_reduce}, we get $\chi^2(\pr_{H_n,\mathbb{Q}_n},\pr_{0,\mathbb{Q}_n})\to 0$, and hence $d_{\mathrm{TV}}(\pr_{H_n,\mathbb{Q}_n},\pr_{0,\mathbb{Q}_n})\to0$.

It remains to verify the remainder condition. Since $\tau_n\le \tau_0$, every $w\in W_n$ satisfies $|N_n(w)|\le \tau_0$, and therefore
\begin{align}
\Theta_{r,n}=0,\qquad r>\tau_0.
\end{align}
Also,
\begin{align}
\Theta_{1,n}
=
\sum_{u\in U_n} d_{W_n}(u)^2
\le
\tau_0\Delta_n^2.
\end{align}
More generally, for each fixed $r\in\{1,\ldots,\tau_0\}$,
\begin{align}
\Theta_{r,n}
=
\sum_{\substack{R\subseteq U_n\\ |R|=r}}
|\mathcal W_n(R)|^2
\le
\binom{\tau_0}{r}\Delta_n^2,
\end{align}
because $|\mathcal W_n(R)|\le \Delta_n$ for every nonempty $R\subseteq U_n$. Hence
\begin{align}
\sum_{r=1}^{\tau_n}(q^{-1}-1)^r\Theta_{r,n}
\le
C_{q,\tau_0}\Delta_n^2
\end{align}
for a constant $C_{q,\tau_0}<\infty$. Since $\tau_n\le \tau_0$, we have $n-\tau_n\asymp n$, and Corollary~\ref{cor:weighted_cover_reduction_simplified_main} gives
\begin{align}
\frac{2e^2\s{Q}_n}{n(n-\tau_n)^2}
\sum_{r=1}^{\tau_n}(q^{-1}-1)^r\Theta_{r,n}
=
O\p{\frac{\s{Q}_n\Delta_n^2}{n^3}}.
\end{align}
This tends to $0$ by assumption, and the conclusion follows.

\subsection{Proof of Corollary~\ref{cor:cliques_direct_main}}\label{subsec:proof_cliques}

(i) Since $|e(K_{k_n})|=\binom{k_n}{2}\le k_n^2/2$, Theorem~\ref{thm:edge_hit_lower_bound} applies directly: if $\s{Q}_n=o(n^2/k_n^2)$ then $\s{Q}_n|e(\Gamma_n)|/\binom n2\to0$, and weak detection is impossible. We record for later use that the proof of Theorem~\ref{thm:edge_hit_lower_bound} gives the stronger statement \eqref{eq:edge_hit_quant}: $\sup_{|\mathbb{Q}_n|\le\s{Q}_n}d_{\s{TV}}(\pr_{\calH_1,\mathbb{Q}_n},\pr_{\calH_0,\mathbb{Q}_n})\le\s{Q}_n\binom{k_n}2/\binom n2\to0$.

(ii) Split $v(K_{k_n})$ into two parts of sizes $a_n\triangleq\lfloor k_n/2\rfloor$ and $b_n\triangleq\lceil k_n/2\rceil$; all pairs between the two parts are edges of $K_{k_n}$, so $K_{a_n,b_n}\subseteq\Gamma_n$. If $k_n\ge 3C_0\log n$ then $a_n\ge k_n/2-1/2\ge C_0\log n$ for all large $n$, and Theorem~\ref{thm:biclique_scan} yields strong detection with
\begin{align}
\s{Q}_{\s{bscan}}\le\frac{1024L^2n^2\log^2n}{a_nb_n}\le\frac{8192L^2\,n^2\log^2 n}{k_n^2}
\end{align}
queries, where we used $a_nb_n\ge k_n(k_n-1)/4\ge k_n^2/8$ for $k_n\ge2$. For comparison, the witness-scan test with witness $H_n=K_{h_n}$, $h_n=\lceil L\log n\rceil$, and $M_n=\min\{n,\lceil An h_n^2/k_n\rceil\}$ can be analyzed exactly as in Section~\ref{subsec:proof_bscan} (Type I: $M_n^{h_n}q^{\binom{h_n}2}\to0$; planted hit: $\mu_n\ge\binom{k_n}{h_n}(M_n/n)^{h_n}\ge(k_nM_n/(h_nn))^{h_n}=(Ah_n)^{h_n}\to\infty$; overlap: $\Delta_n/\mu_n^2\le\sum_{t\ge1}(eh_n^2n/(k_nM_n))^t=\sum_{t\ge1}(e/A)^t\to0$ as $A\to\infty$), and gives the weaker budget $\binom{M_n}{2}\asymp n^2\log^4n/k_n^2$.

\subsection{Proof of Corollary~\ref{cor:clique_pendant_main}}

Choose the vertex cover $U_n$ to be the $k_n$ clique vertices. Then $H_n=\Gamma_n[U_n]=K_{k_n}$,
while every vertex in $W_n=v(\Gamma_n)\setminus U_n$ is a pendant leaf adjacent to exactly one vertex of the core.

We first verify the core condition in
Theorem~\ref{thm:cover_reduction_query_lb}. Since
\begin{align}
H_n=K_{k_n},
\end{align}
 the quantitative edge-hit bound \eqref{eq:edge_hit_quant}, applied to the planted graph $H_n=K_{k_n}$, implies that
\begin{align}
\sup_{|\mathbb{Q}_n|\le\s{Q}_n}
d_{\mathrm{TV}}
(
\pr_{H_n,\mathbb Q_n},
\pr_{\calH_0,\mathbb Q_n}
)
\to0
\end{align}
whenever
\begin{align}
Q_n=o\left(\frac{n^2}{k_n^2}\right).
\end{align}

It remains to analyze the attachment term.
Since every vertex of $W_n$ has exactly one neighbor in the core,
\begin{align}
\Theta_{1,n}
=
k_n\Delta_n^2,
\qquad
\Theta_{r,n}=0,
\quad r\ge2.
\end{align}
Furthermore, by Lemma~\ref{lem:zeta_basic_bound},
\begin{align}
\zeta_{1,n}
\le
\frac{2Q_n}{n}.
\end{align}
Therefore
\begin{align}
\frac{1}{(n-k_n)^2}
\sum_{r=1}^{k_n}
(q^{-1}-1)^r
\Theta_{r,n}\zeta_{r,n}
=
O\!\left(
\frac{Q_nk_n\Delta_n^2}{n^3}
\right),
\end{align}
where we used $|v(\Gamma_n)|=k_n(1+\Delta_n)\le n$ with $\Delta_n\ge1$, hence $k_n\le n/2$ and $(n-k_n)^2\ge n^2/4$.

Consequently,
condition~\eqref{eq:assump_remainder_negligible}
holds whenever
\begin{align}
Q_n
=
o\left(
\frac{n^3}{k_n\Delta_n^2}
\right).
\end{align}
Applying
Theorem~\ref{thm:cover_reduction_query_lb}
now yields
\begin{align}
Q_n
=
o\left(
\min\left\{
\frac{n^2}{k_n^2},
\,
\frac{n^3}{k_n\Delta_n^2}
\right\}
\right)
\Longrightarrow
d_{\mathrm{TV}}
(
\pr_{\calH_1,\mathbb Q_n},
\pr_{\calH_0,\mathbb Q_n}
)
\to0,
\end{align}
which is equivalent to weak detection being impossible.

\subsection{Proof of Corollary~\ref{prop:complete_bipartite_tradeoff_main}}

\paragraph{Proof of (i).}
Since $|e(\Gamma_n)|=a_n b_n$, Theorem~\ref{thm:edge_hit_lower_bound} yields that $\s{Q}_n=o(n^2/|e(\Gamma_n)|)=o(n^2/(a_nb_n))$ implies that weak detection is impossible.

\paragraph{Proof of (ii).}
In the complete bipartite graph $K_{a_n,b_n}$, the $a_n$ vertices on the left have degree $b_n$, while the $b_n$ vertices on the right have degree $a_n$. Since $a_n\le b_n$, choosing $d=b_n$ yields $m_{b_n}(\Gamma_n)\ge a_n$ (with equality if $a_n<b_n$), and therefore $d^2\max\{m_d(\Gamma_n),\log n\}\ge a_nb_n^2$. Thus, if the level $d=b_n$ is feasible, Theorem~\ref{thm:degree_cut_general_unsat_main} and \eqref{eq:degcut_budget_explicit} imply that strong detection is possible with $\s{Q}_{\s{cut}}\le\s{C}n^3\log^2n/(a_nb_n^2)$ queries, and any larger budget can run the same test.

\paragraph{Proof of (iii).}
This is Theorem~\ref{thm:biclique_scan} applied with the complete bipartite subgraph $K_{a_n,b_n}=\Gamma_n$ itself.

\subsection{Proof of Corollary~\ref{cor:biclique_small_side}}

Let $U_n$ be the side of size $a_n$ and $W_n$ the side of size $b_n$. Then $U_n$ is a vertex cover, $\tau_n=a_n\le n/2$ (because $a_n\le b_n$ and $a_n+b_n\le n$), so $n-\tau_n\ge n/2$, and $H_n=\Gamma_n[U_n]$ is edgeless, so that $\pr_{H_n,\mathbb{Q}_n}=\pr_{\calH_0,\mathbb{Q}_n}$ and condition \eqref{eq:assump_core_undetectable} is automatic. Every $w\in W_n$ has $N_n(w)=U_n$, hence $\mathcal W_n(R)=W_n$ for every nonempty $R\subseteq U_n$ and
\begin{align}
\Theta_{r,n}=\binom{a_n}{r}b_n^2,
\qquad
\sum_{r=1}^{a_n}(q^{-1}-1)^r\Theta_{r,n}=b_n^2\pp{\p{1+(q^{-1}-1)}^{a_n}-1}=b_n^2\p{q^{-a_n}-1}.
\end{align}
By Corollary~\ref{cor:weighted_cover_reduction_simplified_main}, weak detection is impossible as soon as
\begin{align}
\frac{\s{Q}_n}{n(n-a_n)^2}\,b_n^2\p{q^{-a_n}-1}\le\frac{4\s{Q}_n\,b_n^2(q^{-a_n}-1)}{n^3}\to0,
\end{align}
which is \eqref{eq:biclique_small_side}; the second form follows from $q^{-a_n}-1\le q^{-a_n}$.

\subsection{Proof of Corollary~\ref{prop:biclique_bounded_left_main}}

We prove Corollary~\ref{prop:biclique_bounded_left_main}, restating it here for convenience with additional details. (The lower bound also follows from Corollary~\ref{cor:biclique_small_side}, since $q^{-a_n}\le q^{-a_0}$ is bounded.)

\begin{corollary}
\label{prop:biclique_bounded_left}
Fix $q\in(0,1)$ and let $\Gamma_n = K_{a_n,b_n}$ and $1\le a_n\le b_n$, where $a_n\le a_0$ for some absolute constant $a_0<\infty$ and all $n$. Then:
\begin{enumerate}
\item If
\begin{align}
\s{Q}_n=o\p{\frac{n^3}{b_n^2}},
\label{eq:biclique_bounded_left_lower}
\end{align}
then weak detection is impossible.
\item If the level $d=b_n$ is feasible in the sense of Theorem~\ref{thm:degree_cut_general_unsat_main}, then there exists a constant $C=C(q,a_0)$ such that
\begin{align}
\s{Q}_n\ge C\frac{n^3}{b_n^2}\log^2 n
\label{eq:biclique_bounded_left_upper}
\end{align}
implies strong detection.
\end{enumerate}
Consequently, in the regime $a_n=O(1)$, the lower and upper bounds match up to a factor of $\log^2 n$.
\end{corollary}

\begin{proof}[Proof of Corollary~\ref{prop:biclique_bounded_left}]
Since $a_n\le a_0$, the graph family $(\Gamma_n)$ has bounded vertex cover number $\tau(\Gamma_n)=a_n\le a_0$. Also, $d_{\max}(\Gamma_n)=b_n$. Therefore Corollary~\ref{cor:bounded_vertex_cover_weighted_main} applies directly and yields
\begin{align}
\s{Q}_n=o\p{\frac{n^3}{d_{\max}(\Gamma_n)^2}}
=
o\p{\frac{n^3}{b_n^2}}
\quad\Longrightarrow\quad
\text{weak detection is impossible}.
\end{align}
This proves \eqref{eq:biclique_bounded_left_lower}.

For the upper bound, note that in $K_{a_n,b_n}$ the $a_n$ vertices on the left all have degree $b_n$. Hence, taking $d=b_n$, we have
\begin{align}
m_{b_n}(\Gamma_n)=a_n.
\end{align}
Since $a_n\le a_0$, Theorem~\ref{thm:degree_cut_general_unsat_main} (see \eqref{eq:degcut_budget_explicit}) implies that, whenever the level $d=b_n$ is feasible, there exists a constant $C=C(q,a_0)$ such that
\begin{align}
\s{Q}_n
\ge
C\frac{n^3}{m_{b_n}(\Gamma_n)b_n^2}\log^2 n
\ge
\frac{C}{a_0}\frac{n^3}{b_n^2}\log^2 n
\end{align}
implies strong detection. Absorbing the factor $a_0^{-1}$ into the constant proves \eqref{eq:biclique_bounded_left_upper}. Therefore the lower and upper bounds differ by at most the $\log^2 n$ factor from Theorem~\ref{thm:degree_cut_general_unsat_main}.
\end{proof}

\subsection{Proof of Corollary~\ref{cor:biclique_summary}}

Part (a) is Corollary~\ref{prop:complete_bipartite_tradeoff_main}(i) and (iii). For part (b), the lower bound is Corollary~\ref{cor:biclique_small_side} together with $q^{-a_n}=e^{a_n\log(1/q)}=n^{o(1)}$ when $a_n=o(\log n)$, and the upper bound is Corollary~\ref{prop:complete_bipartite_tradeoff_main}(ii) together with $a_n=n^{o(1)}$; if $a_n\le a_0$ then $q^{-a_n}\le q^{-a_0}$ and $a_n\le a_0$ are constants, so the two bounds differ by $O(\log^2n)$.

\subsection{Proof of Corollary~\ref{cor:globally_dense}}

(i) If $\s{Q}_n\to\infty$, then \eqref{eq:edgecount_snr_thm} holds because $|e(\Gamma_n)|/\binom n2\ge c$, so the edge-count test achieves strong detection by Theorem~\ref{thm:edgecount_query}. Conversely, if $\s{Q}_n\not\to\infty$ then $\s{Q}_n\le Q_0$ infinitely often for some $Q_0<\infty$, and \eqref{eq:bounded_budget} gives $\s{R}^\star_{n,\s{Q}_n}\ge q^{Q_0}$ along that subsequence, so strong detection is impossible.

(ii) Fix any pair $e\in\binom{[n]}2$ and query it. Under $\calH_0$, $\pr(\s{G}_e=1)=q$. Under $\calH_1$, by symmetry of the uniform random embedding, $\pr(e\in e(\Gamma^\star))=|e(\Gamma_n)|/\binom n2\triangleq\rho_n\ge c$, and $\pr(\s{G}_e=1)=\rho_n+(1-\rho_n)q$. The test $\mathds{1}\{\s{G}_e=1\}$ therefore has risk $q+1-\rho_n-(1-\rho_n)q=1-(1-q)\rho_n\le 1-c(1-q)<1$.

\bibliographystyle{alpha}
\bibliography{bibfile}
\appendix

\section{Preliminaries}
\subsection{Proof of Lemma~\ref{lem:janson_hypergeom}}\label{app:janson_hypergeom}

\emph{Step 1: the product model.} Let $S_p\subseteq[n]$ contain each element independently with probability $p=M/n$, and let $Z_p\triangleq\sum_{c\in\calC}\mathds{1}\{B_c\subseteq S_p\}$. Then $\bE[Z_p]=\sum_cp^{|B_c|}=\bar\mu$ and $\sum_{c\sim c'}\bE[\mathds{1}\{B_c\subseteq S_p\}\mathds{1}\{B_{c'}\subseteq S_p\}]=\sum_{c\sim c'}p^{|B_c\cup B_{c'}|}=\bar\Delta$. Janson's inequality for the product model (see, e.g., \cite[Theorems~8.1.1 and 8.1.2]{AlonSpencer} or \cite[Theorem~2.18]{JLR2000}) states that
\begin{align}
\pr(Z_p=0)\le\exp\p{-\bar\mu+\frac{\bar\Delta}{2}}
\qquad\text{and}\qquad
\pr(Z_p=0)\le\exp\p{-\frac{\bar\mu^2}{2\bar\Delta}}\ \ \text{if }\bar\Delta\ge\bar\mu.
\end{align}
If $\bar\Delta\le\bar\mu$, then $-\bar\mu+\bar\Delta/2\le-\bar\mu/2\le-\bar\mu^2/(2(\bar\mu+\bar\Delta))$; if $\bar\Delta\ge\bar\mu$, then $-\bar\mu^2/(2\bar\Delta)\le-\bar\mu^2/(2(\bar\mu+\bar\Delta))$. In both cases,
\begin{align}
\pr(Z_p=0)\le\exp\p{-\frac{\bar\mu^2}{2(\bar\mu+\bar\Delta)}}.
\label{eq:janson_product}
\end{align}

\emph{Step 2: transfer to the uniform $M$-subset.} For $0\le m\le n$ let $S_m$ denote a uniform $m$-subset of $[n]$ and $Z(S_m)\triangleq\sum_c\mathds{1}\{B_c\subseteq S_m\}$. Conditionally on $|S_p|=m$, the set $S_p$ is distributed as $S_m$, so
\begin{align}
\pr(Z_p=0)=\sum_{m=0}^{n}\pr(|S_p|=m)\,\pr(Z(S_m)=0).
\end{align}
The event $\{Z(S)=0\}$ is decreasing in $S$, and the uniform subsets can be coupled monotonically ($S_m\subseteq S_{m'}$ for $m\le m'$, e.g. by taking the first $m$ and $m'$ elements of a uniformly random permutation of $[n]$); hence $m\mapsto\pr(Z(S_m)=0)$ is nonincreasing, and
\begin{align}
\pr(Z_p=0)\ge\sum_{m\ge M}\pr(|S_p|=m)\,\pr(Z(S_m)=0)\ge\pr(|S_p|\ge M)\,\pr(Z(S_M)=0).
\end{align}
Finally, $|S_p|\sim\s{Binomial}(n,M/n)$ has the integer mean $M$, and the median of a binomial distribution with integer mean equals its mean \cite{KaasBuhrman1980}; hence $\pr(|S_p|\ge M)\ge1/2$. Combining with \eqref{eq:janson_product},
\begin{align}
\pr(Z(S_M)=0)\le 2\,\pr(Z_p=0)\le 2\exp\p{-\frac{\bar\mu^2}{2(\bar\mu+\bar\Delta)}},
\end{align}
which is \eqref{eq:janson_form}.

\subsection{Auxiliary lemmata}

Recall the following standard concentration inequalities

\begin{lemma}[Bernstein's inequality]\label{lem:Bernstein}
Let $X$ be a sum of $m$ independent random variables with values in $[0,1]$, and let $\sigma^2\triangleq\s{Var}(X)$. For any $t\ge 0$,
\begin{align}
\pr(X\ge \bE[X]+t)\leq \exp\p{-\frac{t^2}{2(\sigma^2+t/3)}}
\quad\text{and}\quad
\pr(X\le \bE[X]-t)\leq \exp\p{-\frac{t^2}{2(\sigma^2+t/3)}}.
\label{eq:bernstein_bin_degcut}
\end{align}
In particular, if $X\sim\s{Binomial}(m,q)$ then $\sigma^2=mq(1-q)\le qm$, and both tails are bounded by $\exp\p{-t^2/(2(qm+t/3))}$.
\end{lemma}

\begin{lemma}\label{lem:ChernoffChv}
Let $H\sim\s{Hypergeometric}(N,K,n)$ with mean $\mu=nK/N$. For any $\delta\in(0,1)$
\begin{align}
\pr(H\leq (1-\delta)\mu)\leq \exp\p{-\frac{\delta^2}{2}\mu}.
\label{eq:hyp_lower_degcut}
\end{align}
\end{lemma}

Both lemmas are standard: Bernstein's inequality for bounded independent summands can be found, e.g., in \cite[Appendix~A]{AlonSpencer}, and the multiplicative Chernoff bound for the hypergeometric distribution follows from Hoeffding's observation that sampling without replacement is dominated, for the exponential moment, by sampling with replacement \cite[Theorem~4]{Hoeffding1963}. The same observation yields the following additive bound, which is used in the proof of Theorem~\ref{thm:degree_cut_general_unsat_main}.
\begin{lemma}[Hoeffding's inequality for sampling without replacement]\label{lem:HoeffdingHyp}
Let $H\sim\s{Hypergeometric}(N,K,n)$ with mean $\mu=nK/N$. For any $t\ge0$,
\begin{align}
\pr(H\ge\mu+t)\le\exp\p{-\frac{2t^2}{n}}
\qquad\text{and}\qquad
\pr(H\le\mu-t)\le\exp\p{-\frac{2t^2}{n}}.
\end{align}
\end{lemma}

\end{document}